\documentclass[12pt]{amsart}

\usepackage{
	amsmath,
 	amsfonts,
  	amssymb,
 	amsthm,
        datetime,
	}

\usepackage{tikz-cd}
\usepackage{tikz}

\usepackage{rotating}

\usetikzlibrary{decorations.pathmorphing, decorations.pathreplacing}
 \usetikzlibrary{decorations.pathreplacing,backgrounds,decorations.markings}
 
\usetikzlibrary{arrows,shapes,positioning}
\usetikzlibrary{tikzmark}

\usepackage{color}

\usepackage{stmaryrd}
\usepackage[cal=boondoxo]{mathalfa}
\usepackage{mathtools}
\usepackage{bbm}

\usepackage{hyperref}
\usepackage[capitalise]{cleveref}

\usepackage{a4wide}

\usepackage{mathabx}
\usepackage{tabularx}

\usepackage{etex,bbm,xspace}
\usepackage[cal=boondoxo]{mathalfa}

\theoremstyle{plain}
\newtheorem{thm}{Theorem}[section]
\newtheorem{cor}[thm]{Corollary}
\newtheorem{lem}[thm]{Lemma}
\newtheorem{prop}[thm]{Proposition}

\theoremstyle{definition}
\newtheorem{rem}[thm]{Remark}

\newtheorem{ex}[thm]{Example}
\newtheorem{exe}[thm]{Example}

\theoremstyle{definition}
\newtheorem{defn}[thm]{Definition}

\newenvironment{citethm}[1]{%
	\renewcommand\thethm{\ref{#1}}\thm}{\endthm\addtocounter{thm}{-1}}
	
\newenvironment{citecor}[1]{%
	\renewcommand\thethm{\ref{#1}}\cor}{\endthm\addtocounter{thm}{-1}}
	
\theoremstyle{plain} 
\newcommand{\thistheoremname}{}
\newtheorem{genericthm}[thm]{\thistheoremname}

 \newtheorem*{genericthm*}{\thistheoremname}
\newenvironment{namedthm*}[1]
  {\renewcommand{\thistheoremname}{#1}%
   \begin{genericthm*}}
  {\end{genericthm*}}

\def\makeautorefname#1#2{\expandafter\def\csname#1autorefname\endcsname{#2}}
\makeautorefname{lem}{Lemma}%
\makeautorefname{prop}{Proposition}%
\makeautorefname{rem}{Remark}%
\makeautorefname{section}{Section}%

\newcommand{\cC}{\mathcal{C}}

\newcommand{\bC}{\mathbb{C}}

\newcommand{\bF}{\mathbb{F}}

\newcommand{\bN}{\mathbb{N}}

\newcommand{\bP}{\mathbb{P}}

\newcommand{\bR}{\mathbb{R}}

\newcommand{\bZ}{\mathbb{Z}}

\newcommand{\ba}{\mathbf{a}}
\newcommand{\oba}{ \mathbf{\overline a}}
\newcommand{\bb}{\mathbf{b}}
\newcommand{\obb}{ \mathbf{\overline b}}
\newcommand{\bc}{\mathbf{c}}
\newcommand{\obc}{ \mathbf{\overline c}}
\newcommand{\bd}{\mathbf{d}}

\newcommand{\vsimeq}{\rotatebox{-90}{\(\simeq\)}}
\newcommand{\vsimeqop}{\rotatebox{90}{\(\simeq\)}}

\newcommand{\brak}[1]{\langle #1\rangle}

\newcommand{\setbuild}[2]{\left\{#1\,|\,#2\right\}}

\newcommand{\bigBV}{\bigwedge\nolimits}

\DeclareMathOperator{\id}{Id}

\DeclareMathOperator{\Image}{im}

\newcommand{\slt}{\mathfrak{sl_2}}

\newcommand{\crossingless}[2]{{B^{#1}_{#2}}}
\newcommand{\torus}{T}
\newcommand{\oSpgrFib}[2]{{T^{#1}_{#2}}}
\newcommand{\spgrFib}[2]{{\mathfrak{B}^{#1}_{#2}}}

\DeclareMathOperator{\Gr}{Gr}
\DeclareMathOperator{\PGL}{PGL}
\DeclareMathOperator{\SL}{SL}

\DeclareMathAlphabet{\pazocal}{OMS}{zplm}{m}{n}

\newcommand{\TQFT}{\pazocal{F} }
\newcommand{\oddTQFT}{\pazocal{OF} }
\newcommand{\linoddTQFT}{\bZ\text{-}\pazocal{OF} }

\newcommand{\chcobcat}{\text{\textbf{ChCob}}}
\newcommand{\linchcobcat}{\bZ\text{-\textbf{ChCob}}}
\newcommand{\Top}{\text{\textbf{Man}}_{co}}
\newcommand{\DoubleTop}{\text{\textbf{Man}}_{co}^{\overset{\rightarrow}{\leftarrow}}}
\newcommand{\MDoubleTop}{\underline{\text{\textbf{Man}}}_{co}^{\overset{\rightarrow}{\leftarrow}}}
\newcommand{\linMDoubleTop}{\bZ\text{-\textbf{\underline{Man}}}_{co}^{\overset{\rightarrow}{\leftarrow}}}
\DeclareMathOperator{\htimes}{{\hat{\times}}}

\DeclareMathOperator{\gmod}{\mathrm{-}gmod}

\DeclareMathOperator{\OPol}{OPol}

\DeclareMathOperator{\TopFunctor}{Top}
\DeclareMathOperator{\linTopFunctor}{\bZ\text{-}Top}

\definecolor{myblue}{rgb}{0,.5,1}
\definecolor{mygreen}{rgb}{.3,.75,.1}

\newcolumntype{C}{>{\centering\arraybackslash}X}

\newcommand{\tikzdiag}[2][]{\tikz[#1,thick,baseline={([yshift=1ex+#2]current bounding box.center)}]}
\newcommand{\tikzdiagc}[1][]{\tikzdiag[#1]{-1ex}}
\newcommand{\tikzdiagh}[1][]{\tikzdiag[#1]{-2ex}}

\tikzstyle{tikzdot}=[fill, circle, inner sep=2pt]

\newcommand{\tikzdiagcm}[1][]{\tikz[#1,yscale=.75,scale=.5,thick,baseline={([yshift=-1ex]current bounding box.north)}]}
\newcommand{\tikzdiagocm}[1][]{\tikz[#1,yscale=-.75,scale=.5,thick,baseline={([yshift=0ex]current bounding box.south)}]}
\newcommand{\tikzdiagtcm}[1][]{\tikz[#1,yscale=.75,scale=.5,thick,baseline={([yshift=-.5ex]current bounding box.center)}]}
\newcommand{\cupdiag}[4][0]{\draw (#2,#1) .. controls (#2,-#4+#1) and (#3,-#4+#1) .. (#3,#1)}
\newcommand{\raydiag}[3][0]{\draw (#2,#1) --  (#2,#1-#3)}
\newcommand{\mtikzdot}{node[midway, tikzdot]{}}

\tikzset{
    partial ellipse/.style args={#1:#2:#3}{
        insert path={+ (#1:#3) arc (#1:#2:#3)}
    }
}

\usetikzlibrary{arrows.meta, 
		  decorations.markings}

\tikzstyle directed=[postaction={decorate,decoration={markings,
    mark=at position #1 with {\arrow{>}}}}]
\tikzstyle ddirected=[postaction={decorate,decoration={markings,
    mark=at position #1 with {\arrow{>>}}}}]
\tikzstyle rdirected=[postaction={decorate,decoration={markings,
    mark=at position #1 with {\arrow{<}}}}]

\tikzstyle{orientedcup}=[arrows = {angle 90 reversed - angle 90}]

\title[Odd arc algebras]
{Real Springer fibers and odd arc algebras}


\author{Jens Niklas Eberhardt}
\address{Max-Planck Institute for Mathematics\\
 Vivatsgasse 7 \\ 
53111 Bonn\\ 
Germany}
\email{mail@jenseberhardt.com}

\author{Gr\'egoire Naisse}
\address{Max-Planck Institute for Mathematics\\
 Vivatsgasse 7 \\ 
53111 Bonn\\ 
Germany}
\email{gregoire.naisse@gmail.com}

\author{Arik Wilbert}
\address{University of Georgia\\
Athens, GA 30602 \\
United States of America}
\email{arik.wilbert@uga.edu}
\begin{document}

\begin{abstract}
We give a topological description of the two-row Springer fiber over the real numbers. We show its cohomology ring coincides with the oddification of the cohomology ring of the complex Springer fiber introduced by Lauda--Russell. 
We also realize Ozsv\'ath--Rasmussen--Szab\'o's odd TQFT from pullbacks and exceptional pushforwards along inclusion and projection maps between hypertori. 
Using these results, we construct the odd arc algebra as a convolution algebra over components of the real Springer fiber, giving an odd analog of a construction of Stroppel--Webster. 
\end{abstract}


\maketitle


\section{Introduction}\label{sec:intro}

\emph{Arc algebras}, denoted $H^n_n$,  were introduced by Khovanov \cite{khovanov02} to extend his $\slt$-link homology theory~\cite{khovanov00} to tangles. His invariant associates to a tangle with $2n$ bottom and $2m$ top points an object in the homotopy category of $H^m_m$-$H^n_n$-bimodules. 
Stroppel~\cite{stroppel09} and independently Chen--Khovanov~\cite{chenkhovanov06} later defined \emph{generalized arc algebras} $H^{n-k}_k$, and their quasi-hereditary covers $K^{n-k}_k$, which allow to extend the invariant to all tangles.
These algebras are studied in depth by Brundan--Stroppel in the series of papers~\cite{brundan-Stroppel1,brundan-Stroppel2,brundan-Stroppel3,brundan-Stroppel4}. They can also be generalized further to $\mathfrak{sl}_n$-type~\cite{mackaay-tubbenhauer-pan, mackaay1, tubbenhauer1, tubbenhauer2}, and to other types~\cite{ehrig-stroppel1,ehrig-stroppel2}.

Arc algebras are closely related to \emph{two-row Springer fibers}, that is, varieties of flags of vector spaces fixed by a nilpotent operator with two Jordan blocks. 
In fact, Khovanov~\cite{khovanov04} showed that the center of $H^n_n$ is isomorphic to the cohomology ring of the complex $(n,n)$-Springer fiber $\spgrFib{n}{n}(\bC)$, and Stroppel--Webster~\cite{stroppelwebster12} gave a geometric construction of $H^{n-k}_k$ and $K^{n-k}_k$ as convolution algebras built from the cohomology of the irreducible components of $\spgrFib{n-k}{k}(\bC)$.

Khovanov homology is originally 
constructed using the 2d-TQFT associated to the Frobenius algebra $\bZ[X]/(X^2),$ which coincides with the cohomology ring $H^*(S^2)$ of the 2-sphere $S^2$. As conjectured by Khovanov \cite{khovanov04} and proved by Wehrli~\cite{wehrli},  and independently by Russell--Tymoczko~\cite{russelltymoczko} and Russell~\cite{russell11}, the complex two-row Springer fiber $\spgrFib{n-k}{k}(\bC)$ is homeomorphic to the \emph{topological complex Springer fiber}, a union of certain `diagonals' in $(S^2)^{\times 2n}.$ This prompted a geometric construction of Khovanov homology by the third author \cite{wilbert13} using a setup of pullbacks and (exceptional) pushforwards along maps between irreducible components of $\spgrFib{n-k}{k}(\bC).$

In another direction, Ozsv\'ath--Rasmussen--Szab\'o~\cite{ors} introduced \emph{odd Khovanov homology} for links, an invariant distinct from the classical `\emph{even}' Khovanov homology. 
It arises from a \emph{chronological TQFT} $\oddTQFT : \chcobcat \rightarrow \bZ\gmod$, where $\chcobcat$ is the category of \emph{chronological cobordisms} as developed by Putyra~\cite{putyra14}, and $\bZ\gmod$ is category of graded abelian groups. 
This chronological TQFT was used by Vaz and the second author~\cite{naissevaz18} in their construction of an \emph{odd arc algebra} $OH^n_n$, and by Putyra and the second author in~\cite{naisseputyra} in their extension of odd Khovanov homology from links to tangles.

This raises the question whether the above connections between arc algebras and Springer fibers can be transferred from the even to the odd world.

As a first step 
in this direction, 
\cite{naissevaz18} introduced a topological space $T_n^n$ similar to the topological complex Springer fiber, but with $2$-spheres $S^2$ replaced with 1-spheres $S^1$. 
The authors of \cite{naissevaz18} also constructed isomorphisms between the graded center of $OH^n_n$, the cohomology ring $H^*(T^n_n)$ and  $OH^*(\spgrFib{n}{n}(\bC))$, where $OH^*(\spgrFib{n}{n}(\bC))$ is the `oddification' of  $H^*(\spgrFib{n}{n}(\bC))$ defined by Lauda--Russell~\cite{laudarussell14}.

\smallskip

The goal of this paper is to further the analogies between the even and odd worlds by providing geometric constructions of the chronological TQFT $\oddTQFT$ and of the odd arc algebras, similar to the even case described above.
The starting point is the simple observation that there are homeomorphisms $\bP^1(\bC)\cong S^2$ and $\bP^1(\bR)\cong S^1.$ This suggests the following idea, which we explore here: the \emph{real} Springer fibers $\spgrFib{n-k}{n}(\bR)$ should play the same role in the odd world as the complex Springer fibers $\spgrFib{n-k}{n}(\bC)$  do in the even world. 
We show that this is indeed the case.

\smallskip

It would be interesting to investigate whether real Springer fibers for nilpotent elements with more than two Jordan blocks can be used to construct odd versions of $\mathfrak{sl}_n$-link homology theories. However, this seems to be non-trivial. For example, the cohomology ring of $\mathbb{P}^n(\mathbb{R})$ has torsion for $n>1$, and a different rank than the one for $\mathbb{P}^n(\mathbb{C})$ (meaning that it would probably not give a categorification of the $\mathfrak{sl}_n$-Reshetikhin--Turaev invariant). In particular, we expect a geometric construction of an `odd version of $\mathfrak{sl}_n$-link homology' to reveal new geometric phenomena which are quite different from the $\mathfrak{sl}_2$ case.
Furthermore, our geometric story diverges from Lauda--Russell oddification of the cohomology of the Springer fiber outside of the two-row case, as in general their construction does not deliver a ring structure, while taking the cohomology of a real Springer fiber always gives a ring.

\subsection*{Main results} 
After extending $T^n_n$ to the $(n-k, k)$-case, we show our first main result:

\begin{citethm}{thm:topdescofrealspringer} 
There is a homeomorphism
\[
\spgrFib{n-k}{k}(\bR) \xrightarrow{\simeq} \oSpgrFib{n-k}{k}.
\]
\end{citethm}

This leads us directly to our next main result:

\begin{citethm}{thm:hiso}
There is an explicit isomorphism of rings
\[
h : OH^*(\spgrFib{n-k}{k}(\bC)) \xrightarrow{\simeq}
H^*(\oSpgrFib{n-k}{k}).
\]
\end{citethm}

This isomorphism together with \cref{thm:topdescofrealspringer} actually gives an explicit combinatorial description of $H^*(\spgrFib{n-k}{k}(\bR))$, since $OH^*(\spgrFib{n-k}{k}(\bC))$ is defined as a certain quotient of a (odd) polynomial ring. 
 
\smallskip

We then explain in \cref{sec:oddTQFT} how the chronological TQFT $\oddTQFT : \chcobcat \rightarrow \bZ\gmod$ from~\cite{ors} can be obtained from pullbacks and (exceptional) pushfowards along certain inclusion and projection maps between real hypertori. We formalize this by introducing the category $\DoubleTop$ of closed oriented manifolds and zigzags of continuous maps between them (see \cref{def:doubletop}).
Denote by $H^* : \DoubleTop \rightarrow \bZ\gmod$ the cohomology functor, sending covariant maps to pushforwards and contravariant maps to pullbacks. We show that there exists a  functor $\TopFunctor : \chcobcat \rightarrow \DoubleTop$, such that we have the following:
\begin{citecor}{cor:geomTQFT}
The diagram of functors 
\[
\begin{tikzcd}
\chcobcat \ar{rr}{\oddTQFT} \ar[swap]{dr}{\TopFunctor} && \bZ\gmod \\
&\DoubleTop \ar[swap]{ur}{H^*}&
\end{tikzcd}
\]
is commutative.  
\end{citecor}

The proof of \cref{cor:geomTQFT} uses the following theorem for $\bZ$-linear versions $\linchcobcat$ and $\linMDoubleTop$ of $\chcobcat$ and $\DoubleTop$ respectively, which are $\bZ$-graded symmetric monoidal categories (similar to monoidal supercategories considered in \cite{supermonoidal}, but with a full $\bZ$-grading instead of a $\bZ/2\bZ$ one, as in \cite[\S4]{naisseputyra}, see \cref{sec:gradedmon} below). 

\begin{citethm}{thm:geomTQFT}
The diagram of $\bZ$-graded symmetric monoidal functors 
\[
\begin{tikzcd}
\linchcobcat \ar{rr}{\linoddTQFT} \ar[swap]{dr}{\linTopFunctor} && \bZ\gmod \\
&\linMDoubleTop \ar[swap]{ur}{H^*}&
\end{tikzcd}
\]
is commutative.  
\end{citethm}

Using this result, we mimic in \cref{sec:oddArcAlg} the construction from~\cite{stroppelwebster12} in order to present the (generalized) odd arc algebra $OH_k^{n-k}$ as a convolution algebra with product $\star_f$ over the cohomology of the components 
$\{ T_\ba | \ba \in B_k^{n-k} \}$ 
of the real $(n-k,k)$-Springer fiber. 
We  also obtain an odd version $OK^{n-k}_k$ of the quasi-hereditary cover $K^{n-k}_k$ of the arc algebra.

 \begin{namedthm*}{\cref{lem:OHnnsame}, \cref{thm:convOHn} and \cref{thm:oddqhcover}}
 The composition law $\star_f$ gives $OH_k^{n-k}$ the structure of a graded non-associative algebra, which is isomorphic to the the odd arc algebra from~\cite{naissevaz18} in the case $OH^n_n$. Moreover,  in general $OH_k^{n-k}$ agrees modulo 2 with the usual generalized arc algebra $H^{n-k}_k$ from \cite{stroppel09,  chenkhovanov06}. 
 Similarly, $OK_k^{n-k}$ is a graded non-associative algebra, which is isomorphic modulo 2 with the quasi-hereditary cover $K^{n-k}_k$.
\end{namedthm*}

 Furthermore, by the same arguments, we also obtain that the complexes of bimodules used to define odd Khovanov homology for tangles in~\cite{naisseputyra} arise from similar geometric constructions with convolution products. 

Finally, we show that  the `odd' center $OZ(OH^{n-k}_k)$ of the odd arc algebra coincides with the cohomology of the real $(n-k,k)$-Springer fiber, using the same arguments as in the $(n,n)$-case considered in \cite{naissevaz18}.
\begin{citethm}{thm:oddcenter}
There is an isomorphism of rings
\[
H^*(\spgrFib{n-k}{k}(\bR))
 \xrightarrow{\simeq}
OZ(OH^{n-k}_k) 
\]
induced by the inclusions $T_\ba \hookrightarrow \spgrFib{n-k}{k}(\bR)$. 
\end{citethm}


\subsection*{Acknowledgments}
We thank Matthew B. Young for helpful conversations and providing numerous comments on a first version of the paper. 
We also warmly thank the referee for his/her awesome, detailed report. 
We also thank the Newton Institute in Cambridge and the Hausdorff Research Institute for Mathematics in Bonn for their hospitality. 
The first two authors are grateful to the Max Planck Institute for Mathematics in Bonn for its hospitality and financial support.


\section{A topological description of real two-row Springer fibers}\label{sec:topdescspgr}
We give an explicit topological description of the real two-row Springer fiber in terms of a space we call the topological real Springer fiber. 
This space is constructed as the union of certain diagonals inside a hypertorus. 

This section deals with real and complex varieties. We use \cite[Section 2]{poonen2017rational} as a standard reference for varieties over fields that are not necessarily algebraically closed. Further, for a variety $X$ over some subfield $F\subset \bR$ we will implicitly equip the sets of real and complex points $X(\bR)$ and $X(\bC)$ with their analytic topology, that is, the topology inherited from the Euclidian topology on $\bR$ and $\bC.$
If furthermore $X$ is smooth, we consider the topological spaces $X(\bR)$ and $X(\bC)$ with their natural structure of differentiable manifolds.

\subsection{Springer fibers}
Let $F$ be a field. For a partition $\lambda = (\lambda_1, \dots, \lambda_l)$  of $n \in \bN,$ denote by $z_\lambda : F^n \rightarrow F^n$ the nilpotent linear operator with $l$ Jordan blocks of size $\lambda_1, \dots, \lambda_l$. 
The \emph{Springer fiber} associated to $\lambda$ is the set
\[
\spgrFib{}{\lambda}(F) := \setbuild{V_\bullet= (V_1 \subset \dots \subset V_n=F^n)}{\dim_k V_i=i\text{ and }z_\lambda V_i \subset V_{i-1}},
\]
of all full flags of subspaces of $F^n$ preserved by $z_\lambda.$ In fact, $\spgrFib{}{\lambda}(F)$ is more than just a plain set. It is the set of $F$-points of a projective algebraic variety $\spgrFib{}{\lambda}.$ 
Moreover, $\spgrFib{}{\lambda}(\bR)$ and $\spgrFib{}{\lambda}(\bC)$, the \emph{real} and \emph{complex} Springer fibers, become compact Hausdorff spaces when equipped with their analytic topology.
\smallskip

We are concerned with the two-row case $\lambda = (n-k,k)$, where we assume w.l.o.g. $k\leq n/2$, and abbreviate
\[
\spgrFib{n-k}{k}:=\spgrFib{}{\lambda}.
\]

\subsection{Topological real Springer fiber} 
Denote by $S^1 := \setbuild{ (x,y) \in \bR^2}{x^2+y^2 = 1}$ the 1-sphere and by $p := (0,1)$ its north pole. Let $\torus^n := (S^1)^{\times n}$ be the $n$-dimensional torus with its usual topology. 
The topological real Springer fiber is the union of certain diagonal subsets in $\torus^n$ labeled by crossingless matchings.  It is an extension of the construction of the $(n,n)$ topological real Springer fiber from \cite{naissevaz18}, and follows the topological description of the complex two-row Springer fiber but replacing 2-spheres by 1-spheres. 

In the following definition we count points from left to right:
\begin{defn}
A (generalized) \emph{crossingless matching} of type $(n-k,k)$ is a combinatorial datum represented by a way of connecting $n$ points on a horizontal line by placing  $k$ arcs below and $n-2k$ semi-infinite vertical rays, without having them crossing each other. Each endpoint of an arc/ray is connected to a single point and we cannot have two arcs/rays connected to the same point.
\end{defn}

We write $\crossingless{n-k}{k}$ for the set of crossingless matchings of type $(n-k,k)$. For $\ba \in \crossingless{n-k}{k}$ we write $(i,j) \in \ba$ 
for $i < j$ 
 if there is an arc in $\ba$ connecting the $i$th point to the $j$th one. Similarly, we write $(i) \in \ba$ if there is a ray connected to the $i$th dot.

\begin{exe}
Here is an example of a $(8-3,3)$ crossingless matching:
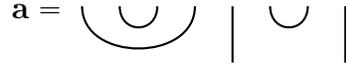
\[
	\ba =\  
	\tikzdiagcm{
		\cupdiag{0}{3}{2};
		\cupdiag{1}{2}{1};
		\raydiag{4}{2};
		\cupdiag{5}{6}{1};
		\raydiag{7}{2};
	 }
\]
where $(5) \in \ba$ and $(8) \in \ba$, and $(1,4) \in \ba, (2,3) \in \ba$ and $(6,7) \in \ba$.
\end{exe}

\begin{defn}\label{def:topoddSpgr}
For $\ba \in \crossingless{n-k}{k}$ we define the subspace
\[
\torus_{\ba; n-k,k} := \left\{(x_1, \dots, x_n) \in \torus^n \left|
\begin{array}{ll}
 x_i = x_j,  &\text{ if } (i,j) \in \ba,\\
 x_i =  (-1)^i p, &\text{ if } (i) \in \ba ,
\end{array}
\right.
\right\} \subset T^n.
\]
Then, we write
\[
\oSpgrFib{n-k}{k} := \bigcup_{\ba \in \crossingless{n-k}{k}} \torus_{\ba; n-k,k} \subset T^n.
\]
We refer to $\oSpgrFib{n-k}{k}$ as the \emph{topological real Springer fiber}. 
\end{defn}

If no confusion can arise we will write $T_\ba$ instead of $T_{\ba;n-k,k}$.

\smallskip

In the remainder of this section will prove the following theorem.
\begin{thm}\label{thm:topdescofrealspringer} 
There is a homeomorphism
\[
\spgrFib{n-k}{k}(\bR) \xrightarrow{\simeq} \oSpgrFib{n-k}{k},
\]
such that real points of the irreducible components of $\spgrFib{n-k}{k}$ are mapped to the subsets $T_\ba\subset \oSpgrFib{n-k}{k}.$
\end{thm}

\subsection{The complex Springer fiber} 
We first recall how the topological description of the \emph{complex two-row Springer fiber} $\spgrFib{n-k}{k}(\bC)$ works.

The main idea is to embed $\spgrFib{n-k}{k}(\bC)$ into an ambient space $Y_n(\bC),$ which was introduced and shown to be diffeomorphic to $\bP^1(\bC)^n$ by Cautis and Kamnitzer~\cite{cautis-kamnitzer}. 
The irreducible components of $\spgrFib{n-k}{k}(\bC),$ which have been described in the work of Spaltenstein~\cite{spaltenstein76},
Vargas~\cite{vargas79} and Fung~\cite{fung2003}, then admit an explicit description as certain `antidiagonal' subsets of $\bP^1(\bC)^n$ labeled by crossingless matchings.

\smallskip

Fix some integer $N\geq n.$ Denote the standard basis of $F^{2N}$ by $e_1,\dots,e_N,f_1,\dots,f_N$ and the nilpotent linear operator with Jordan blocks $(N,N)$ by $z.$ Hence $$z(e_i)=e_{i-1}\text{ and } z(f_i)=f_{i-1},$$ where $e_{0}=f_{0}=0.$

Let $Y_n(F)$ be the following set of partial flags of subspaces of $F^{2N}$ preserved by $z$,
\[
Y_n(F):=\setbuild{V_\bullet = (V_1 \subset \dots \subset V_n \subset F^{2N})}{\dim_F V_i=i\text{ and }z V_i \subset V_{i-1}}.
\]
In fact, $Y_n(F)$ is the set of $F$-points of a smooth projective variety $Y_n.$ We will identify $\spgrFib{n-k}{k}$ with a closed subvariety in $Y_n,$ such that
\[
\spgrFib{n-k}{k}(F)=\setbuild{V_\bullet\in Y_n(F)}{V_n=\langle e_1,\dots e_{n-k},f_1,\dots,f_{k}\rangle_F}\subset Y_n(F),
\]
where $\langle e_1,\dots e_{n-k},f_1,\dots,f_{k}\rangle_F$ is the $F$-span of $e_1,\dots e_{n-k},f_1,\dots,f_{k}$.

In general, $Y_n$ is a tower of algebraically non-trivial projective bundles. The complex points $Y_n(\bC)$, interpreted as a smooth manifold, split diffeomorphically into a product $Y_n(\bC)\cong\bP^1(\bC)^n.$ 

To see this, denote the standard basis of $\bC^2$ by $e,f$ and define a linear map 
\begin{equation}\label{eq:themapC}
C:\bC^{2N}\rightarrow \bC^2,\, C(e_i)=e\text{ and }C(f_i)=f.
\end{equation}
Equip $\bC^{2N}$ and $\bC^{2}$ with the standard Hermitian inner product, and denote $\bP^1(\bC)=\bP(\bC^2)$. Consider the map
$$\ell_\bC: Y_n(\bC)\rightarrow \bP^1(\bC)^n,\, (V_\bullet)\mapsto (C(V_i\cap V_{i-1}^\perp))_{i=1,\dots, n}.$$
\begin{prop}\emph{(Cautis--Kamnitzer \cite[Theorem 2.1]{cautis-kamnitzer}.)}
The map $\ell_\bC$ is a well-defined diffeomorphism.
\end{prop}
For a crossingless matching $\ba\in B^{n-k}_{k}$, define
\[
\bP_\ba(\bC):=\setbuild{(L_i)_{i=1,\dots, n}}{L_i=L_j^\perp\text{ if } (i,j)\in \ba\text{ and }L_i=\langle e\rangle_\bC\text{ if } (i)\in \ba}\subset \bP^1(\bC)^n,
\]
and let 
\[
\bP_n^{n-k}(\bC):=\bigcup_{\ba\in B^{n-k}_{k}}\bP_\ba(\bC)\subset \bP^1(\bC)^n.
\]

\begin{thm}\label{thm:complexspringer} The map $\ell_\bC$ identifies $\spgrFib{n-k}{k}(\bC)$ with $\bP_n^{n-k}(\bC).$ Under $\ell_\bC$ the complex points of the irreducible components of $\spgrFib{n-k}{k}(\bC)$ correspond to the $\bP_\ba(\bC)$ for $\ba\in B^{n-k}_{k}.$
\end{thm}
\begin{proof} For $k=n/2$, see \cite[Theorem 5.5]{russelltymoczko} and \cite[Theorem 1.2]{wehrli}. For general $k$, see \cite[Theorem 2.7]{russell11} and \cite[Theorem 1.15]{wilbert13}.
\end{proof}
\subsection{From Complex to Real} The topological description of the complex Springer fiber $\spgrFib{n-k}{k}(\bC)$ descends to a topological description of the real Springer fiber $\spgrFib{n-k}{k}(\bR).$ To see this, we first recall some standard facts about descent for the Galois extension $\bC/\bR$.

\smallskip

Denote by $\Gamma=\langle \sigma \rangle=\operatorname{Gal}(\bC/\bR)$ the Galois group of $\bC$ over $\bR,$ where $\sigma$ denotes complex conjugation. Let $m \in \bN$ be some integer. Then, $\Gamma$ acts coordinatewise on $\bC^m.$ This action is antilinear and we identify the fixed point set $(\bC^m)^\Gamma$ with $\bR^m.$ We equip $\bR^m$ and $\bC^m$ with their standard (Hermitian) inner products. 
For any subspace $V$ of the complex vector space $\bC^m$
the set of fixed points $V^\Gamma\subset\bR^m$ is naturally a subspace of the real vector space $\bR^{m}.$

We will need the following standard facts.
\begin{prop}\label{prop:complextoreal} 
\leavevmode
\begin{enumerate}
\item There is a bijection between $\Gamma$-stable subspaces of the complex vector space $\bC^m$ and subspaces of the real vector space $\bR^m$ given by
\[
\setbuild{V\subset \bC^m}{\sigma(V)=V} \xrightarrow{\simeq} \{V\subset \bR^m\},\, V\mapsto V^\Gamma.
\]
The bijection is inclusion and intersection preserving, and $\dim_\bC(V)=\dim_\bR(V^\Gamma).$

\item  Let $V\subset \bC^m$ be a subspace. Then
\[
\sigma(V)^\perp=\sigma(V^\perp)\subset \bC^m.
\]

\item Let $V\subset \bC^m$ be a $\Gamma$-stable subspace. Then, $V^\perp\subset\bC^m$ is also $\Gamma$-stable and 
\[
(V^\perp)^\Gamma=(V^\Gamma)^\perp\subset \bR^m.
\]
\end{enumerate}
\end{prop}
Let $X$ be an algebraic variety defined over a subring of $\bR$. Then, $\Gamma$ acts on the set of complex points $X(\bC)$ of $X$ by complex conjugation, and one can identify the set of real points $X(\bR)$ of $X$ as $\Gamma$-fixed points in $X(\bC)$
\[
X(\bR)\cong X(\bC)^\Gamma.
\]
\begin{ex}
\label{ex:realflagvarasfixedpoints} 
\leavevmode
\begin{enumerate}
\item
Denote the set of flags in $F^m$ by
\[
X(F):=\setbuild{V_\bullet = (V_1 \subset \dots \subset V_m=F^{m})}{\dim_F V_i=i}.
\]
Then, $X(F)$ is the set of $F$-points of the \emph{flag variety} $X$, which is defined over $\bZ.$ 
The group $\Gamma$ acts on $V_\bullet\in X(\bC)$ via $V_i\mapsto \sigma(V_i)$.
Hence, the fixed points $X(\bC)^\Gamma$ are $\Gamma$-stable flags in $\bC^m$, which are naturally identified with flags in $\bR^m$ via
\[
X(\bC)^\Gamma \xrightarrow{\simeq} X(\bR), \quad (V_1 \subset \dots \subset V_m= \bC^{m}) \mapsto(V_1^\Gamma \subset \dots \subset V_m^\Gamma= \bR^{m}),
\]
using \cref{prop:complextoreal}(1).
\item In the same way, the spaces $Y_n$, $\spgrFib{n-k}{k}$ and $(\bP^1)^n$ are algebraic varieties defined over $\bZ$, and we can identify
\begin{align*}
Y_n(\bC)^\Gamma&\xrightarrow{\simeq}Y_n(\bR),\\
\spgrFib{n-k}{k}(\bC)^\Gamma&\xrightarrow{\simeq} \spgrFib{n-k}{k}(\bR), \\ 
 (\bP^1(\bC)^n)^\Gamma&\xrightarrow{\simeq} \bP^1(\bR)^n.
\end{align*}
\item The set $\bP_n^{n-k}(\bC)\subset \bP^1(\bC)^n$ is not a subvariety since its definition involves complex conjugation. Still, the action of $\Gamma$ on $\bP^1(\bC)^n$ restricts to an action on $\bP_n^{n-k}(\bC)$ using \cref{prop:complextoreal}(2). The set of $\Gamma$-fixed points can then be identified with $\bP_n^{n-k}(\bR)$ which is defined as follows. For a crossingless matching $\ba\in B^{n-k}_{k}$, define
\[
\bP_\ba(\bR):=\setbuild{(L_i)_{i=1,\dots, n}}{L_i=L_j^\perp\text{ if } (i,j)\in \ba\text{ and }L_i=\langle e\rangle_\bR\text{ if } (i)\in \ba}\subset \bP^1(\bR)^n,
\]
and let 
\[
\bP_n^{n-k}(\bR):=\bigcup_{\ba\in B^{n-k}_{k}}\bP_\ba(\bR)\subset \bP^1(\bR)^n.
\]
It is clear that the natural identification $(\bP^1(\bC)^n)^\Gamma\cong\bP^1(\bR)^n$ descends to $\bP_n^{n-k}(\bC)^\Gamma\cong\bP_n^{n-k}(\bR).$
\end{enumerate}
\end{ex}
Putting everything together we can now prove \cref{thm:topdescofrealspringer}.
\begin{proof}[Proof of \cref{thm:topdescofrealspringer}]
First, \cref{thm:complexspringer} provides us with the commutative diagram
\[
\begin{tikzcd}
 \spgrFib{n-k}{k}(\bC) 
 \ar{r}{\simeq}
 \ar{d}
 & 
 \bP_n^{n-k}(\bC)
 \ar{d}
 \\
  Y_n(\bC) 
  \ar["\ell_\bC"']{r}{\simeq}
  & 
  \bP^1(\bC)^n 
\end{tikzcd}
\]
where the vertical arrows are closed embeddings and the horizontal arrows are homeomorphisms. 

As described in \cref{ex:realflagvarasfixedpoints}, there is an action of the Galois group $\Gamma$ on the spaces in the diagram.
The action clearly commutes with the vertical arrows which are inclusions. 
To see that $\ell_\bC$ commutes with the action of $\Gamma$ we use 
\cref{prop:complextoreal} and that the linear map $C$ defined in \eqref{eq:themapC} is compatible with complex conjugation.

By passing to $\Gamma$-fixed points and using the identifications from \cref{ex:realflagvarasfixedpoints}, we obtain the diagram
\[
\begin{tikzcd}
 \spgrFib{n-k}{k}(\bR) 
 \ar{r}{\simeq}
 \ar{d}
 & 
 \bP_n^{n-k}(\bR)
 \ar{d}
 \\
  Y_n(\bR) 
  \ar["\ell_\bR"']{r}{\simeq}
  & 
  \bP^1(\bR)^n 
\end{tikzcd}
\]

Recall that $S^1 := \{ (x,y) \in \bR^2 | x^2+y^2 = 1\}$ denotes the 1-sphere with northpole $p := (0,1)$ and that the odd topological Springer fiber $\oSpgrFib{n-k}{k}$ is defined as a certain subspace of  $\torus^n := (S^1)^{n}.$   

Choose a diffeomorphism $\Psi:\bP^1(\bR) \xrightarrow{\simeq} S^1$ which maps $\langle e \rangle_\bR$ to $p$ and $\langle f \rangle_\bR$ to $-p$. Consider the diffeomorphism
\begin{align*}
\Phi:\bP^1(\bR)^n\rightarrow \torus^n,  \quad (L_1,L_2,L_3,L_4\dots)\mapsto (\Psi(L_1),\Psi(L_2^\perp),\Psi(L_3),\Psi(L_4^\perp),\dots),
\end{align*}
sending the $i$th line $L_i$ to $\Psi(L_i)$ if $i$ is odd, and to $\Psi(L_i^\perp)$ if $i$ is even. 
It is clear by the definitions that $\Phi$ identifies the subspaces $\bP_n^{n-k}(\bR)$ and $\oSpgrFib{n-k}{k}$, and the statement is proven.
\end{proof}
\begin{rem}
\leavevmode
\begin{enumerate}
\item The proof of \cref{thm:topdescofrealspringer} shows that $Y_n(\bR)$ is diffeomorphic to the direct product $\bP^1(\bR)^n$. We want to remark that this is not immediately implied by the fact that $Y_n(\bC)$ and $\bP^1(\bC)^n$ are diffeomorphic. 
For example, let $\bF_n=\bP(\mathcal{O}\oplus\mathcal{O}(n))$ be the Hirzebruch surface of degree $n\in \bZ,$ that is, the projectivization of the vector bundle $\mathcal{O}\oplus\mathcal{O}(n)$ on $\bP^1.$ Then, $\bF_n(\bC)$ is diffeomorphic to the product $\bP^1(\bC)^2.$ But the real points $\bF_n(\bR)$ are diffeormorphic to the product $\bP^1(\bR)^2$ if and only if $n$ is even. If $n$ is odd, $\bF_n(\bR)$ is diffeomorphic to a Klein bottle. The proof shows that ``there are no Klein bottles'' in $Y_n(\bR)$, or in the real Springer fiber $\spgrFib{n-k}{k}(\bR).$

\item In \cite{wilbert13}, the third author extends the topological description of complex two-row Springer fibers to complex two-row \emph{Spaltenstein varieties}, which are a partial flag generalization of Springer fibers. It is easy to check that the descent arguments from above also work in this case and hence yields a topological description of \emph{real} two-row Spaltenstein varieties.

\item One of the original motivations of Cautis--Kamnitzer to consider $Y_n(\bC)$ comes from the geometric Satake equivalence which links the representation theory of $\SL_2(\bC)$ with the geometry of the affine Grassmannian $\Gr(\bC)$ of $\PGL_2(\bC).$ Under this equivalence the $n$-fold tensor product $V^{\otimes n}$
of the fundamental representation $V$ of $\SL_2(\bC)$ corresponds to a certain $n$-fold convolution product $\Gr^{\omega}\tilde{\times}\dots\tilde{\times}\Gr^{\omega}(\bC)$ of a Schubert variety $\Gr^{\omega}(\bC)\subset\Gr(\bC)$ corresponding to the highest weight $\omega$ of $V.$ They prove that the varieties $\Gr^{\omega}\tilde{\times}\dots\tilde{\times}\Gr^{\omega}$ and  $Y_n$ are isomorphic. Hence, the convolution product  $\Gr^{\omega}\tilde{\times}\dots\tilde{\times}\Gr^{\omega}(\bR)$ in the \emph{real affine Grassmannian} is isomorphic to $Y_n(\bR).$
\end{enumerate}
\end{rem}

%
%


\section{Cohomology of real two-row Springer fibers}\label{sec:oddTopSpringer}
In \cite{laudarussell14}, Lauda and Russel define the \emph{odd cohomology} of the complex Springer fiber $OH^*(\spgrFib{n-k}{k}(\bC))$ \emph{ad hoc}, namely by using generators and relations: 
 they slightly modify a description of the cohomology ring of the complex Springer fiber $H^*(\spgrFib{n-k}{k}(\bC))$ by letting the generators anticommute instead of commute, and modifying the relations accordingly. 
The goal of this section is to show that $OH^*(\spgrFib{n-k}{k}(\bC))$ is isomorphic to the cohomology of the real Springer fiber $H^*(\spgrFib{n-k}{k}(\bR)).$ We thereby obtain an explicit algebraic description of $H^*(\spgrFib{n-k}{k}(\bR)).$

\smallskip
In order to do this, we will make use of the homeomorphism $T_k^{n-k} \cong \spgrFib{n-k}{k}(\bR)$ of \cref{thm:topdescofrealspringer}, and work with the space $T_k^{n-k}$. 
Because of this homeomorphism, we will call \emph{components of $T_k^{n-k}$} the subsets $T_{\ba}$, as they agree with the irreducible components of $ \spgrFib{n-k}{k}(\bR)$ under the above-mentioned homeomorphism.
From now on, we will always write $H^*(M) = H^*(M, \bZ)$ for the integral singular cohomology ring of a topological space $M$. Similarly $H_*(M)$ denotes the homology over $\bZ$.

\subsection{Odd cohomology of complex Springer fibers}

We recall Lauda--Russel's definition of $OH^*(\spgrFib{n-k}{k}(\bC))$ from \cite{laudarussell14}. Define the ring of \emph{odd polynomials} as
\begin{align*}
\OPol_n &:= \frac{\bZ\brak{x_1, \dots, x_n}}{(x_ix_j + x_jx_i = 0, \text{ for all } i \neq j)}.
\end{align*}

\begin{rem}
While one could be tempted to call $\OPol_n$ the `ring of supercommutative polynomials' this would suggest that the relation $x_i^2=0$ holds in $\OPol_n$ which is not the case. Hence the name `odd'.
\end{rem}

Let $S$ be an ordered subset of $\{x_1, \dots, x_n\}$, and write
\[
x_i^S := 
\begin{cases}
(-1)^{S(i) -1} x_i, & \text{ if } x_i \in S, \\
 0, &\text{ otherwise,}
\end{cases}
\]
where $S(i)$ is the position of $x_i$ in $S$.
Then, define the \emph{odd partially symmetric functions}
\[
\varepsilon_r^S := \sum_{1 \leq i_1 < \dots < i_r \leq n} x_{i_1}^S \cdots x_{i_r}^S.
\]
The \emph{odd Tanisaki ideal} $OI^{n-k}_{k}$ is the left ideal of $\OPol_n$ generated by the elements
\begin{align*}
OC^{n-k}_{k} := \left \{\varepsilon_r^S  \bigg| 
\parbox{20em}{$\ell \in \{1,\dots,n-k\}$, $ r \in \{1+\delta_\ell, \dots,k+\ell\}$,\\ \centering{$S \subset\{x_1,\dots, x_n\}$, $|S| = k + \ell$}} 
\right\},
\end{align*}
where 
\[
\delta_\ell := 
\begin{cases}
k, &\text{if } \ell \le n - 2k,  \\
n-k-\ell, &\text{otherwise.}
\end{cases}
\]

\begin{defn}[\cite{laudarussell14}]
The \emph{odd cohomology} of the complex $(n-k,k)$-Springer fiber is defined as the $\OPol_n$-module
\[
OH^*(\spgrFib{n-k}{k}(\bC)) := \OPol_n/OI^{n-k}_{k}.
\]
\end{defn}

In general, the odd cohomology of a complex Springer fiber does not possess a ring structure (at least not induced by the odd polynomial ring structure).   However, in the two-row case, it does. This is because $x_i^2 \in OI^{n-k}_{k}$ for all $i$, and therefore $OI^{n-k}_{k}$ is a two-sided ideal.

\subsection{Cohomology of the components of $\oSpgrFib{n-k}{k}$}

As in \cite{khovanov04,brundan-Stroppel1} we describe the components of $\oSpgrFib{n-k}{k}$ using arc and circle diagrams. 
For $\ba \in \crossingless{n-k}{k}$, we write $\oba$ for the horizontal mirror image of the crossingless matching.
For $\bb \in \crossingless{n-k}{k}$, we denote by $\obb\ba$ the circle diagram given by gluing $\obb$ on top of $\ba$. 
Note that we can view $\obb\ba$ as consisting of rays (connecting bottom-top), turnbacks (connecting bottom-bottom or top-top) and circles, and we denote the number of the latter by $|\obb\ba|$.

\begin{exe}
In $\crossingless{7-3}{3}$, we can have for example:
\begin{align*}
	\ba  &=  \  
	\tikzdiagcm{
		\cupdiag{0}{1}{1};
		\cupdiag{2}{3}{1};
		\raydiag{4}{2};
		\cupdiag{5}{6}{1};
	 }
&
	\bb  &=  \  
	\tikzdiagcm{
		\cupdiag{0}{3}{2};
		\cupdiag{1}{2}{1};
		\cupdiag{4}{5}{1};
		\raydiag{6}{2};
	 }
\\
	\obb &=  \  
	\tikzdiagocm{
		\cupdiag{0}{3}{2};
		\cupdiag{1}{2}{1};
		\cupdiag{4}{5}{1};
		\raydiag{6}{2};
	 }
&
	\obb\ba &= \ 
	\tikzdiagtcm{
		\cupdiag{0}{3}{-2};
		\cupdiag{1}{2}{-1};
		\cupdiag{4}{5}{-1};
		\raydiag{6}{-2};
		\cupdiag{0}{1}{1};
		\cupdiag{2}{3}{1};
		\raydiag{4}{2};
		\cupdiag{5}{6}{1};
	 }
\end{align*}
and thus $|\obb\ba| = 1$.
\end{exe}

Let $T^0 := \{\star\}$ be a point manifold. 
Since a crossingless matching $\ba \in B^{n-k}_k$ contains $k$ arcs and $n-2k$ rays, each component $\torus_{\ba; n-k,k}$ is homeomorphic to $\torus^k \times\{\star\}^{\times (n-2k)} \cong \torus^k$, and thus
\begin{align*}
H^*(\torus_{\ba})&\cong\bigBV^\bullet \brak{X_1, \dots, X_k}.
\end{align*}
If there is a turnback in $\obb\ba,$ then $T_\bb$ has one of its coordinate fixed to $\pm p$ and $T_\ba$ has the same coordinate fixed to $\mp p$ (by parity reasons), and thus $\torus_{\bb} \cap \torus_{\ba} = \emptyset$. 
Otherwise, $\torus_{\bb} \cap \torus_{\ba}$ is homeomorphic to $\torus^{|\obb\ba|}$. Indeed, arcs in $\ba$ and in $\bb$ equates coordinates in $\torus_{\bb} \cap \torus_{\ba}$, so that each circle component in $\obb\ba$ gives a set of equated coordinates in $\torus_{\bb} \cap \torus_{\ba}$. The rays fix the other coordinates to the same $\pm p$ (again by parity reasons), see \cref{ex:topspgrn4k1} below.

\begin{exe}\label{ex:topspgrn4k1}
Consider $n=4, k = 1$. We have $\crossingless{4-1}{1} = \{ \ba,\bb,\bc \}$ where 
\begin{align*}
\ba &=\  
	\tikzdiagcm{
		\cupdiag{0}{1}{1};
		\raydiag{2}{2};
		\raydiag{3}{2};
	 } 
&
\bb &= \ 
	\tikzdiagcm{
		\raydiag{0}{2};
		\cupdiag{1}{2}{1};
		\raydiag{3}{2};
	 }
& 
\bc &= \ 
	\tikzdiagcm{
		\raydiag{0}{2};
		\raydiag{1}{2};
		\cupdiag{2}{3}{1};
	 }
\end{align*}
and for example we obtain 
\begin{align*}
\obb\ba &=\  
	\tikzdiagtcm{
		\raydiag{0}{-2};
		\cupdiag{1}{2}{-1};
		\raydiag{3}{-2};
		\cupdiag{0}{1}{1};
		\raydiag{2}{2};
		\raydiag{3}{2};
	 } 
& 
\obc\ba &=\  
	\tikzdiagtcm{
		\raydiag{0}{-2};
		\raydiag{1}{-2};
		\cupdiag{2}{3}{-1};
		\cupdiag{0}{1}{1};
		\raydiag{2}{2};
		\raydiag{3}{2};
	 }
& 
\oba\ba &=\  
	\tikzdiagtcm{
		\cupdiag{0}{1}{-1};
		\raydiag{2}{-2};
		\raydiag{3}{-2};
		\cupdiag{0}{1}{1};
		\raydiag{2}{2};
		\raydiag{3}{2};
	 }
\end{align*}
Thus, we have
\begin{align*}
\torus_{\ba} \cap \torus_{\bb} &= \{(x,x,-p, p) | x \in S^1\} \cap \{(-p,x, x, p)  | x \in S^1\} = \{-p,-p,-p,p\}, \\
\torus_{\ba} \cap \torus_{\bc} &= \{(x,x,-p, p) | x \in S^1\} \cap \{(-p, p, x, x)  | x \in S^1\} = \emptyset\text{ and}\\
\torus_{\ba} \cap \torus_{\ba} &= \{(x,x,-p, p) | x \in S^1\} \cong S^1.
\end{align*}
\end{exe}

\subsubsection{Diagrammatic description}\label{ssec:diagdesccohomology}
As in \cite{russell11} (replacing homology with cohomology), one can describe $H^*(\torus_{\ba})$ diagrammatically using $\bZ$-linear combinations of dotted crossingless matchings of shape $\ba$. These are crossingless matchings which can carry extra decorations on their arcs, which we draw as a dots. For example we can have:
\begin{equation}\label{eq:exdotBn}
	\tikzdiagcm{
		\cupdiag{0}{3}{2} \mtikzdot;
		\cupdiag{1}{2}{1};
		\raydiag{4}{2};
		\cupdiag{5}{6}{1} \mtikzdot;
	}
\end{equation}
We also require that there can be no two dots at the same height, and a dot on a ray, or two dots on the same arc is zero:
\begin{align*}
	\tikzdiagh{
		\raydiag{0}{1}  node[pos=.5,tikzdot]{};
	 }
	 \ &:= 0,
&
	\tikzdiagcm{
		\cupdiag{0}{1}{1} node[pos=.4,tikzdot]{} node[pos=.8,tikzdot]{};
	 }
	 \ &:= 0,
\end{align*}
Furthermore, we consider these diagrams up to sliding of dots on the arcs, with the rule that there can be no two dots at the same height, and exchanging two distant dots permutes the sign: 
\begin{align*}
	\tikzdiagh{
		\raydiag{0}{1}  node[pos=.25,tikzdot]{};
		\raydiag{1}{1}  node[pos=.75,tikzdot]{};
	 }
\ &=- \ 
	\tikzdiagh{
		\raydiag{0}{1}  node[pos=.75,tikzdot]{};
		\raydiag{1}{1}  node[pos=.25,tikzdot]{};
	 }
\end{align*}
To translate from diagrams to elements of $H^*(\torus_{\ba})$, we first associate bijectively to each arc in $\ba$ a copy of $T^1$ in $\torus_{\ba} \cong T^1 \times \cdots \times T^1$. 
Then, we read a diagram from bottom to top, and whenever we encounter a dot we cup on the left with the corresponding  generating cohomology element of $H^*(\torus_{\ba})$.
For example, if $\ba$ is the same as in the  example above, we have $H^*(\torus_{\ba}) \cong \bigBV^\bullet \brak{X_1, X_2, X_3}$, with the $X_i$'s associated to the arcs as follow:
\begin{equation*}
	\tikzdiagcm{
		\cupdiag{0}{3}{2};
		\cupdiag{1}{2}{1};
		\raydiag{4}{2};
		\cupdiag{5}{6}{1};
		\node at(-.5,-1) {$X_1$};
		\node at(1.5,.5) {$X_2$};
		\node at(5.5,.5) {$X_3$};
	}
\end{equation*}
and then \cref{eq:exdotBn} gives $X_3 \wedge X_1$. 
One could extend the definition to handle multiplication from $H^*(\torus_{\ba})$, but the signs make the definition cumbersome. Since we will not use it, we leave the details to the reader.

\smallskip

Similarly, we describe the cohomology $H^*( \torus_{\bb}  \cap  \torus_{\ba} )$ by dotted circle diagrams of shape $\obb\ba$. For example, we can have the following:
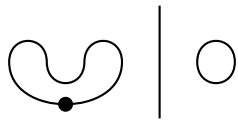
\begin{equation*}
	\tikzdiagtcm{
		\cupdiag{0}{1}{-1};
		\cupdiag{2}{3}{-1};
		\raydiag{4}{-2};
		\cupdiag{5}{6}{-1};
		\cupdiag{0}{3}{2}\mtikzdot;
		\cupdiag{1}{2}{1};
		\raydiag{4}{2};
		\cupdiag{5}{6}{1};
	}
\end{equation*}
Then, the morphism $H^*(\torus_{\ba}) \rightarrow H^*(\torus_{\bb}  \cap  \torus_{\ba} )$ induced by the inclusion $ \torus_{\bb}  \cap  \torus_{\ba} \hookrightarrow   \torus_{\ba}$ is given by gluing a dotted crossingless matching of shape $\ba$ below the crossingless matching $\obb$. Similarly, $H^*(\torus_{\bb}) \rightarrow H^*( \torus_{\bb}  \cap  \torus_{\ba} )$ is given by gluing on top. We still have the rule that a dot on a ray is zero, and two dots on the same component is zero. 

\begin{rem}
The dotted diagram we use here are distinct from the dotted diagrams used to construct type D arc algebras from \cite{ehrig-stroppel1}.
\end{rem}

\subsection{Geometric realization of $OH^*(\mathfrak B_{n-k,k})$}

Our goal now will be to show the following theorem, which is a generalization to the $(n-k,k)$-case of one of the main results in \cite{naissevaz18}:
\begin{thm}\label{thm:hiso}
There is an isomorphism of rings
\[
h : OH^*(\spgrFib{n-k}{k}(\bC)) \xrightarrow{\simeq}
H^*(\oSpgrFib{n-k}{k}).
\]
\end{thm}

We first show the existence of a ring morphism $h :  OH^*(\spgrFib{n-k}{k}(\bC)) \rightarrow H^*(\oSpgrFib{n-k}{k})$ by constructing a morphism $h_0 : \OPol_n \rightarrow H^*(\oSpgrFib{n-k}{k})$ and proving $OC^{n-k}_{k} \subset \ker h_0$. Then, we show that $h$ is both injective and surjective. 

\subsubsection{Cell decomposition}
Similarly as in~\cite[Lemma 4]{khovanov04} and in~\cite[\S 3.2]{russell11}, 
we construct a cell decomposition of $\torus^{n-k}_{k}$. The main difference is that in our case the cells can also be of odd dimension, since we work with copies of $S^1$ and not $S^2$.

\begin{defn}
We say that there is an arrow $\ba \rightarrow \bb$, for $\ba,\bb \in \crossingless{n-k}{k}$ if either there is a quadruple $1 \leq i < j < r < s \leq n$ such that
\begin{itemize}
\item $\ba$ and $\bb$ are the same everywhere except on $i,j,r,s$,
\item $(i,j),(r,s) \in \ba$, 
\item $(i,s),(j,r) \in \bb$,
\end{itemize}
or if there is a triple $1 \leq i < j < \ell \leq n$ such that
\begin{itemize}
\item $\ba$ and $\bb$ are the same everywhere except on $i,j,\ell$,
\item $(i), (j,\ell) \in \ba$,
\item $(i,j),(\ell) \in \bb$.
\end{itemize}
\end{defn}

Visually, we can draw this as
\begin{align*}
	\tikzdiagtcm[xscale=1.25]{
		\draw[white]  (0,0) .. controls (0,-2) and (3, -2) ..  (3,0); 
		\draw  (0,0) node[above]{$i$} .. controls (0,-1) and (1, -1) ..  (1,0) node[above]{$j$};
		\draw  (2,0) node[above]{$r$} .. controls (2,-1) and (3, -1) ..  (3,0) node[above]{$s$};
	 }
	\quad &\rightarrow \quad 
	\tikzdiagtcm[xscale=1.25]{
		\draw  (0,0) node[above]{$i$} .. controls (0,-2) and (3, -2) ..  (3,0) node[above]{$s$};
		\draw  (1,0) node[above]{$j$} .. controls (1,-1) and (2, -1) ..  (2,0) node[above]{$r$};
	 }
&\text{ or }&&
	\tikzdiagtcm[xscale=1.25]{
		\draw  (0,0) node[above]{$i$}  -- (0,-2) ;
		\draw  (1,0) node[above]{$j$} .. controls (1,-1) and (2, -1) ..  (2,0) node[above]{$\ell$};
	 }
	\quad &\rightarrow \quad 
	\tikzdiagtcm[xscale=1.25]{
		\draw  (0,0) node[above]{$i$} .. controls (0,-1) and (1, -1) ..  (1,0) node[above]{$j$};
		\draw  (2,0) node[above]{$\ell$}  -- (2,-2);
	 }
\end{align*}

 Then, we define a partial order on $\crossingless{n-k}{k}$ by saying $\ba \prec \bb$ if there exists a chain of arrows $\ba \rightarrow \ba_1\rightarrow \dots \rightarrow \ba_m \rightarrow \bb$. We extend this order arbitrarily to a total order $<$ on $\crossingless{n-k}{k}$.
 
\smallskip

For $\ba \in \crossingless{n-k}{k}$, let $\Gamma$ be the graph with vertices given by arcs of $\ba$ (hence $\Gamma$ has $k$ vertices), and there is an edge between $v$ and $w$ in $\Gamma$ whenever there exists $\bb \in \crossingless{n-k}{k}$ such that $\bb \rightarrow \ba$ and $\bb$ is obtained from $\ba$ by only reconnecting the endpoints of $v$ and $w$. In other words, there is an arc between $v$ and $w$ if one is contained in the other, without any other arc in-between. Let $M$ be the set of roots of $\Gamma$ (i.e. outermost arcs of $\ba$) and $E$ be the set of edges. For an arc $v = (i,j) \in \ba$ we write $v_1 := i$ and $v_2 := j$.

\begin{exe}\label{exe:arctograph}
We can have
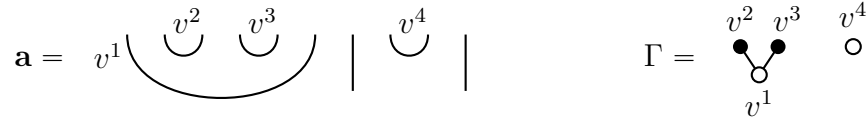
\begin{align*}
	\ba &=\  
	\tikzdiagtcm{
		\cupdiag{0}{5}{3}; \node at(-.5,-.75){$v^1$};
		\cupdiag{1}{2}{1}; \node at(1.6,.5){$v^2$};
		\cupdiag{3}{4}{1}; \node at(3.6,.5){$v^3$};
		\raydiag{6}{2};
		\cupdiag{7}{8}{1}; \node at(7.6,.5){$v^4$};
		\raydiag{9}{2};
	 }
&
	\Gamma &= \ 
	\tikzdiagtcm{
		\draw (1,0) -- (.5,-1) node[pos=0, tikzdot]{};
		\draw (0,0) -- (.5,-1) node[pos=0, tikzdot]{} node[pos=1,fill=white, draw=black,circle,inner sep=2pt]{};
		\draw (3,0) -- (3,0)node[pos=0, fill=white, draw=black,circle,inner sep=2pt]{};
		\node at(.5,-2){$v^1$};
		\node at(0,1){$v^2$};
		\node at(1.25,1){$v^3$};
		\node at(3,1.125){$v^4$};
	 }
\end{align*}
where the roots $M$ are represented by a hollowed dot, and we have given a name to the arc to make them correspond with vertices of $\Gamma$
\end{exe}

\smallskip

 For each subset $J \subset E\sqcup M$ we construct a cell $c(J) \subset T_{\ba} \subset \torus^n$ by setting $(x_1, \dots, x_n) \in c(J)$ if
\begin{itemize}
\item $x_{v_1} = x_{w_1}$ if there is an edge between $v$ and $w$ in $\Gamma$ and it is contained in $J$,
\item $x_{v_1} \neq x_{w_1}$ if there is an edge between $v$ and $w$ in $\Gamma$ but it is not contained in $J$,
\item $x_{v_1} = (-1)^{v_1}p$ if $v \in M \cap J$,
\item $x_{v_1} \neq (-1)^{v_1}p$ if $v \in M$ and $v \notin J$. 
\end{itemize}
Clearly, $c(J) \cong \bR^{k-|J|}$ and $\bigsqcup_{J} c(J) = T_{\ba}$.

\begin{exe}
Consider $\Gamma$ in \cref{exe:arctograph}, and let $e^{1,2}$ be the edge between $v^1$ and $v^2$, and $e^{1,3}$ the edge between $v^1$ and $v^3$. 
We have $T_\ba = \{ x_1, x_2, x_2, x_3, x_3, x_1, -p, x_4, x_4, p  \} \subset T^{2n}$. 
Then, for example we can have
\begin{align*}
c(\emptyset) &= \{x_1, x_2, x_2, x_3, x_3, x_1, -p, x_4, x_4, p | x_1 \neq x_2, x_1 \neq x_3, x_1 \neq -p, x_4 \neq p \}, \\
c(\{ v^4 \} ) &= \{x_1, x_2, x_2, x_3, x_3, x_1, -p, p, p, p | x_1 \neq x_2, x_1 \neq x_3, x_1 \neq -p \}, \\
c(\{ e^{1,2} \}) &= \{ x_1, x_1, x_1, x_3, x_3, x_1, -p, x_4, x_4, p | x_1 \neq x_3, x_1 \neq -p, x_4 \neq p  \},  \\
c(\{ v^1, v^4, e^{1,2}, e^{1,3} \}) &= \{ -p,-p,-p,-p,-p,-p, -p, p, p, p \}.
\end{align*}
\end{exe}


Basically all results from~\cite[\S3.3]{russell11} (and their duals as in Khovanov's original paper~\cite{khovanov04}) hold for the odd case. However, the existence of cells in odd degree adds some subtleties. For example, it is not immediate that the inclusion 
\[
T_{<\ba} \cap T_\ba \hookrightarrow T_{\ba},
\]
 with $T_{<\ba} := \bigcup_{\bb < \ba} T_\bb$, induces a surjective morphism 
 \[
 H^*(T_{\ba}) \rightarrow H^*(T_{<\ba} \cap T_\ba).
 \]
 It is however still true and can be shown using an observation made in~\cite[Lemma~4.25]{naissevaz18}:

\begin{lem}\label{lem:Tbasurject}
The morphism  $H^*(T_{\ba}) \rightarrow H^*(T_{<\ba} \cap T_\ba)$ induced by the inclusion  $T_{<\ba} \cap T_\ba \hookrightarrow T_{\ba}$  is surjective.
\end{lem}

\begin{proof}
Note that $|E \sqcup M| = k$. 
Thus, the cell decomposition of $T_\ba$ described above has the same number of $\ell$-dimensional cells as the rank of $H^\ell(T_\ba)$, since $T_\ba \cong T^k$. 
 Hence, characteristic functions on these cells are independent additive generators for the cohomology. 
 By~\cite[Lemma~3.12]{russell11}, which can be directly converted to the odd case, the cell decomposition of $T_\ba$ restricts to a cell decomposition of $T_{< \ba} \cap T_\ba$. Therefore, $H^\ell(T_{<\ba} \cap T_\ba)$ is free and $H^\ell(T_{\ba}) \rightarrow H^\ell(T_{<\ba} \cap T_\ba)$ is surjective. 
\end{proof}

There is a  Mayer--Vietoris sequence associated to the union  $T_{\leq \ba} := T_{<\ba} \cup T_{\ba}$.
By exactness of it and \cref{lem:Tbasurject}, we deduce its boundary maps are all zero. 
As a consequence, we get the following proposition, which is the odd version of the dual of~\cite[Theorem~3.25]{russell11}.

\begin{prop}\label{thm:exactOH}
The following sequence of abelian groups is exact
\[
0 \rightarrow H^*(T_{\leq \ba}) \xrightarrow{\phi} \bigoplus_{\bb \leq \ba} H^*(T_\bb) \xrightarrow{\psi^-} \bigoplus_{\bb<\bc\leq \ba} H^*(T_\bb \cap T_\bc),
\]
where $\phi$ is induced by the inclusions $T_\bb \hookrightarrow T_{\leq \ba}$, and where we put
\[
\psi^- := \sum_{\bb < \bc \leq \ba} (\psi_{\bb, \bc} - \psi_{\bc, \bb}),
\]
with $\psi_{\bb, \bc} : H^*(T_\bb) \rightarrow H^*(T_\bb \cap T_\bc)$ induced by the inclusion $(T_\bb \cap T_\bc) \hookrightarrow T_\bb$.
\end{prop}

As a direct corollary, by taking $\ba$ maximal, we obtain:
\begin{cor}\label{cor:injectiveTntosum}
The morphism 
\[
\rho : 
 H^*( \oSpgrFib{n-k}{k}) \hookrightarrow  \bigoplus_{\ba \in \crossingless{n-k}{k}} H^*(T_{\ba} )
\]
induced by the inclusions $T_{\ba } \hookrightarrow \oSpgrFib{n-k}{k}$ is injective.
\end{cor}

Similarly to~\cite{khovanov04}, the inclusion $\oSpgrFib{n-k}{k} \hookrightarrow \torus^n$ induces a surjective (see \cref{lem:surjTn} below) morphism $H^*(\torus^n) \rightarrow H^*(\oSpgrFib{n-k}{k}).$ Composition with the projection $\torus^n \twoheadrightarrow S^1$ onto the $i$th factor induces a morphism $H^*(S^1) \rightarrow H^*(\oSpgrFib{n-k}{k})$. We write $X_i \in H^*(\oSpgrFib{n-k}{k}) $ for the image of the generator of $H^*(S^1)$ under this morphism. 
Let $\ba_i \in H^*(\torus_\ba)$ be the dotted crossingless matching of shape $\ba$ given by putting a dot on the arc connected to the $i$th endpoint, if there is one, and $\ba_i$ is zero if a ray is connected to the $i$th endpoint. 
Then, the morphism $H^*(\oSpgrFib{n-k}{k}) \rightarrow H^*(\torus_{\ba})$ induced by the inclusion $T_{\ba } \hookrightarrow \oSpgrFib{n-k}{k}$ sends $X_i$ to $\ba_i$. 
From that, we deduce that the morphism in \cref{cor:injectiveTntosum} takes the form
\begin{equation}\label{eq:monotorus}
\rho : H^*(\oSpgrFib{n-k}{k}) \hookrightarrow \bigoplus_{\ba \in \crossingless{n-k}{k}} H^*(\torus_{\ba}), \quad X_i \mapsto \sum_\ba \ba_i.
\end{equation}
Hence, we can describe $H^*(\oSpgrFib{n-k}{k})$ as the subring of $\bigoplus_{\ba \in \crossingless{n-k}{k}} H^*(\torus_{\ba})$ generated by the elements $\sum_\ba \ba_i$ for all $1 \leq i \leq n$.

\subsubsection{Existence of $h$}

This subsection is a generalization of~\cite[\S4]{naissevaz18}.
Define the ring morphism
\[
h_0 : \OPol_n \rightarrow H^*(\oSpgrFib{n-k}{k}), \quad x_i \mapsto X_i.
\]
Note that in contrast to the even case~\cite{khovanov04} there is no sign involved here. This is similar to the disappearance of the sign in the odd nilHecke algebra in~\cite{naisseputyra}.

\smallskip

We will now focus on proving that $OC^{n-k}_{k} \subset \ker h_0$. In fact, we will prove that $OC^{n-k}_{k} \subset \ker h_\ba$ for each $\ba \in \crossingless{n-k}{k}$,  where
\[
h_\ba : \OPol_n \rightarrow H^*(\torus_{\ba}), \quad x_i \mapsto \ba_i, 
\]
is the morphism obtained by composing $h_0$ with $\rho$ and projecting onto the summand $H^*(\torus_{\ba}).$

\begin{exe}
Consider $n = 4, k = 1$ as in \cref{ex:topspgrn4k1}. We identify elements of $H^*(T_\ba)$ with dotted diagrams as explained in~\cref{ssec:diagdesccohomology}. 
We have, among others, $-x_ix_j \in OC^{4-1}_{1}$ and $-x_1+x_2-x_3+x_4 \in OC^{4-1}_{1}$. 
Then, we compute
\begin{align*}
h_\ba(-x_1 x_2) &= - \ 
\tikzdiagcm{
		\cupdiag{0}{1}{1} node[pos=.2,tikzdot]{} node[pos=.6,tikzdot]{};
		\raydiag{2}{2};
		\raydiag{3}{2};
	 }  \ = 0, 
&
h_\ba(-x_1 x_3) &= - \ 
\tikzdiagcm{
		\cupdiag{0}{1}{1} node[pos=.5,tikzdot]{};
		\raydiag{2}{2}  node[pos=.65,tikzdot]{};
		\raydiag{3}{2};
	 }  \ = 0, 
\end{align*}
since two dots on the same component or a dot on a ray is zero, 
and
\[
h_\ba(-x_1+x_2-x_3+x_4)  = - \ 
\tikzdiagcm{
		\cupdiag{0}{1}{1} node[pos=.5,tikzdot]{};
		\raydiag{2}{2};
		\raydiag{3}{2};
	 } 
\ + \ 
\tikzdiagcm{
		\cupdiag{0}{1}{1}  node[pos=.5,tikzdot]{};
		\raydiag{2}{2};
		\raydiag{3}{2};
	 } 
\ - \ 
\tikzdiagcm{
		\cupdiag{0}{1}{1} ;
		\raydiag{2}{2} node[pos=.5,tikzdot]{};
		\raydiag{3}{2};
	 } 
\ + \ 
\tikzdiagcm{
		\cupdiag{0}{1}{1};
		\raydiag{2}{2};
		\raydiag{3}{2} node[pos=.5,tikzdot]{};
	 } 
\ = 0.
\]
Thus, $-x_ix_j \in \ker(h_\ba)$ and $-x_1+x_2-x_3+x_4 \in \ker(h_\ba)$. 
\end{exe}

Let $S \subset \{x_1, \dots, x_n\}$ with $|S| = k + \ell$. For each $R = \{x_{i_1}, \dots, x_{i_r}\} \subset S$ we write 
\[
\varepsilon_R^S := x_{i_1}^S \cdots x_{i_r}^S.
\]
Note that 
$
\varepsilon_r^S = \sum_{R \subset S, |R| = r} \varepsilon_R^S.
$

Let $E_{\ba}$ be the subset of $\{x_1, \dots, x_n\}$ such that $x_i \in E_{\ba}$ if $(i) \notin \ba$ ($i$ is not connected to a ray, thus is connected to an arc). Thus, $|E_{\ba}| = 2k$ and 
\[
h_\ba(\varepsilon_r^S) = h_\ba(\sum_{\substack{R \subset S \cap E_{\ba},\\ |R| = r}} \varepsilon_{R}^S ),
\]
since $h_\ba(x_i) = 0$ whenever $(i) \in \ba$.

Therefore, all arguments in~\cite[\S4]{naissevaz18} apply directly to the general case, and we have $OC_{n-k,k} \subset \ker h_\ba$. In particular, we get the following:

\begin{prop}
The morphism
\[
h : OH^*(\spgrFib{n-k}{k}(\bC)) \rightarrow H^*(\oSpgrFib{n-k}{k}), \quad x_i \mapsto X_i,
\]
is well-defined.
\end{prop}

\subsubsection{Injectivity of $h$}
To show that $h$ is injective, it is enough to show that its reduction modulo 2 is injective, since $h$ is a morphisms between free abelian groups. 
For this, we use the fact that the even and odd cohomology coincide modulo $2$. 

\begin{lem}\label{lem:isoOHTMod2}
Tensoring with $\bZ/2\bZ$, $h$ induces an isomorphism
\[
\tilde h : OH^*(\spgrFib{n-k}{k}(\bC)) \otimes \bZ/2\bZ \xrightarrow{\simeq} H^*(\oSpgrFib{n-k}{k}) \otimes \bZ/2\bZ.
\]
\end{lem}

\begin{proof}
By~\cite{laudarussell14} we know that $$ OH^*(\spgrFib{n-k}{k}(\bC)) \otimes \bZ/2\bZ \cong  H^*(\spgrFib{n-k}{k}(\bC)) \otimes \bZ/2\bZ,$$ identifying the $x_i$'s in both constructions while describing $ H^*(\spgrFib{n-k}{k}(\bC))$ using the usual Tanisaki ideal of the polynomial ring $\bZ[x_1,\dots,x_n]$. 
Moreover, we have 
\[
H^*(S^{n-k}_{k}) \otimes \bZ/2\bZ \cong H^*(\oSpgrFib{n-k}{k}) \otimes \bZ/2\bZ,
\]
where $S^{n-k}_{k}$ it the topological complex Springer fiber obtained by replacing $S^1$ with $S^2$ in \cref{def:topoddSpgr}. 
This isomorphism can be achieved by observing that $\bigoplus_{\ba \in \crossingless{n-k}{k}} H^*(S_{\ba})$ and $\bigoplus_{\ba \in \crossingless{n-k}{k}} H^*(\torus_{\ba})$ are isomorphic modulo 2, where $S_\ba$ is constructed as $T_\ba$ but using $S^2$ instead of $S^1$, and then that both injections $H^*(S^{n-k}_{k}) \hookrightarrow\bigoplus_{\ba \in \crossingless{n-k}{k}} H^*(S_{\ba}) $ and $H^*(\oSpgrFib{n-k}{k}) \hookrightarrow \bigoplus_{\ba \in \crossingless{n-k}{k}} H^*(\torus_{\ba})$ are the same modulo 2. All of this means that the following diagram is commutative:
\[
\begin{tikzcd}
OH^*(\spgrFib{n-k}{k}(\bC)) \otimes \bZ/2\bZ \ar[r, "\tilde h"] \ar{d}{\vsimeq}
& 
H^*(\oSpgrFib{n-k}{k}) \otimes \bZ/2\bZ,  \\
H^*(\spgrFib{n-k}{k}(\bC)) \otimes \bZ/2\bZ \ar[r, "\simeq"] 
& 
H^*(S^{n-k}_{k}) \otimes \bZ/2\bZ  \ar{u}{\vsimeqop}.
\end{tikzcd}
\]
In particular, $\tilde h$ is an isomorphism.
\end{proof}

\begin{prop}\label{prop:hinjective}
The morphism $h$ is injective.
\end{prop}

\begin{proof} 
A morphism between free abelian groups inducing an isomorphism modulo 2 is injective. It is the case of $h$ because of~\cref{lem:isoOHTMod2}.
\end{proof}

\subsubsection{Surjectivity of $h$}

\begin{lem}\label{lem:surjTn}
The morphism 
\[
\pi : H^*(\torus^n) \twoheadrightarrow H^*(\oSpgrFib{n-k}{k}),
\]
induced by the inclusion $\oSpgrFib{n-k}{k}  \hookrightarrow \torus^n$ is surjective.
\end{lem}

\begin{proof}
As in~\cite[Lemma~4.28]{naissevaz18}, the morphism $H_*(\oSpgrFib{n-k}{k}) \rightarrow H_*(\torus^n)$ induced on the homology is split injective. Hence, the dual morphism is surjective.
\end{proof}

Note that if we identify $H^*(\torus^n)$ with $H^*(\torus^1) \otimes \cdots \otimes H^*(\torus^1)$, then the cohomology generator of the $i$th factor $H^*(\torus^1)$ is sent by $\pi$ to $X_i$. 

\begin{prop} \label{prop:hsurjective}
The morphism $h$ is surjective.
\end{prop}

\begin{proof}
By definition of $h_0$, we obtain a commutative diagram
\[
\begin{tikzcd}
\OPol_{n} \ar[dr,"\simeq", sloped]  \ar[rr,"h_0"] & & H^*(\oSpgrFib{n-k}{k}).  \\
& H^*(\torus^n) \ar[ur,"\pi"',twoheadrightarrow]  &
\end{tikzcd}
\]
Thus, $h_0$ is surjective. It follows that $h$ is also surjective.
\end{proof}

\begin{proof}[Proof of \cref{thm:hiso}]
It follows immediatly from \cref{prop:hinjective} and \cref{prop:hsurjective} that $h$ is an isomorphism
\[
h :  OH^*(\spgrFib{n-k}{k}(\bC)) \xrightarrow{\simeq} H^*(\oSpgrFib{n-k}{k}),
\]
concluding the proof.
\end{proof}

Then, as a consequence of \cref{thm:topdescofrealspringer} and \cref{thm:hiso}, we obtain the following:

\begin{cor}\label{cor:hiso}
There is an isomorphism of rings
\[
H^*(\spgrFib{n-k}{k}(\bR))  \cong OH^*(\spgrFib{n-k}{k}(\bC)).
\]
\end{cor}

%
%


\section{Geometric contruction of  Ozsv\'ath--Rasmussen--Szab\'o  odd TQFT}\label{sec:oddTQFT}

Stroppel--Webster~\cite{stroppelwebster12} and the third author~\cite{wilbert13}  describe a realization of the TQFT $\TQFT$ associated to the Frobenius algebra $\bZ[X]/(X^2)$ 
using pullbacks and pushwards along maps associated to certain diagonal maps and projections between products of $2$-spheres. 

In \cite{ors} Ozsv\'ath--Rasmussen--Szab\'o define an \emph{odd} version $\oddTQFT$ of the TQFT $\TQFT$ by replacing tensor products of $\bZ[X]/(X^2)$ with exterior products.  
This section shows that there is also a realization of the odd TQFT $\oddTQFT$ obtained by replacing 2-spheres by 1-spheres in \cite{stroppelwebster12,wilbert13}.
To do so, we first recall an algebraic definition of the odd TQFT using the presentation given in \cite{putyra14}, and we discuss some general results for pullbacks and pushforwards in cohomology.

\subsection{The odd TQFT}\label{sec:defoddTQFT}

 \subsubsection{Category of chronological cobordisms}\label{sec:chcobcat}
 
 A \emph{chronological (2-)cobordism}~\cite{putyra14} is a $2$-cobordism equipped with a framed Morse function that separates critical
points, in the sense of Igusa~\cite{igusa}. Concretely, this means that the cobordism is equipped with a height function such that there are no two critical points (i.e. saddle point or cup/cap) at the same height, and saddle points and caps are given an orientation. We picture these cobordisms by diagrams read from bottom to top, and the orientation is given by an arrow. 

Chronological cobordisms form an `almost monoidal' category $\chcobcat$ where objects are disjoint unions of circles and morphisms are chronological cobordisms between them (up to certain isotopy equivalences not altering the relative height of the critical points, see \cite{putyra14} for details). 
Vertical composition is given by gluing along the boundary as usual, and horizontal composition of objects is given by disjoint union. The horizontal composition is given by \emph{right-then-left juxtaposition} of cobordisms, pushing all critical points of the cobordism on the left to the top:
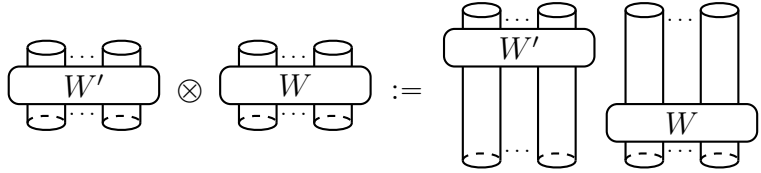
\[
\tikzdiagh[scale=.5]{
	\draw (0,0) .. controls (0,.-.25) and (1,-.25) .. (1,0);
	\draw[dashed] (0,0) .. controls (0,.25) and (1,.25) .. (1,0);
	\draw (0,0) -- (0,2);
	\draw (1,0) -- (1,2);
	\draw (0,2) .. controls (0,1.75) and (1,1.75) .. (1,2);
	\draw (0,2) .. controls (0,2.25) and (1,2.25) .. (1,2);
	\node at(1.5,.25) {\tiny $\dots$};
	\node at(1.5,1.75) {\tiny $\dots$};
	\draw (2,0) .. controls (2,.-.25) and (3,-.25) .. (3,0);
	\draw[dashed] (2,0) .. controls (2,.25) and (3,.25) .. (3,0);
	\draw (2,0) -- (2,2);
	\draw (3,0) -- (3,2);
	\draw (2,2) .. controls (2,1.75) and (3,1.75) .. (3,2);
	\draw (2,2) .. controls (2,2.25) and (3,2.25) .. (3,2);
	\filldraw [fill=white, draw=black,rounded corners] (-.5,.5) rectangle (3.5,1.5) node[midway] { $W'$};
}
\  \otimes \ 
\tikzdiagh[scale=.5]{
	\draw (0,0) .. controls (0,.-.25) and (1,-.25) .. (1,0);
	\draw[dashed] (0,0) .. controls (0,.25) and (1,.25) .. (1,0);
	\draw (0,0) -- (0,2);
	\draw (1,0) -- (1,2);
	\draw (0,2) .. controls (0,1.75) and (1,1.75) .. (1,2);
	\draw (0,2) .. controls (0,2.25) and (1,2.25) .. (1,2);
	\node at(1.5,.25) {\tiny $\dots$};
	\node at(1.5,1.75) {\tiny $\dots$};
	\draw (2,0) .. controls (2,.-.25) and (3,-.25) .. (3,0);
	\draw[dashed] (2,0) .. controls (2,.25) and (3,.25) .. (3,0);
	\draw (2,0) -- (2,2);
	\draw (3,0) -- (3,2);
	\draw (2,2) .. controls (2,1.75) and (3,1.75) .. (3,2);
	\draw (2,2) .. controls (2,2.25) and (3,2.25) .. (3,2);
	\filldraw [fill=white, draw=black,rounded corners] (-.5,.5) rectangle (3.5,1.5) node[midway] { $W$};
}
\ := \  \tikzdiagh[scale=.5]{
	\draw (0,0) .. controls (0,.-.25) and (1,-.25) .. (1,0);
	\draw[dashed] (0,0) .. controls (0,.25) and (1,.25) .. (1,0);
	\draw (0,0) -- (0,4);
	\draw (1,0) -- (1,4);
	\draw (0,4) .. controls (0,3.75) and (1,3.75) .. (1,4);
	\draw (0,4) .. controls (0,4.25) and (1,4.25) .. (1,4);
	\node at(1.5,.25) {\tiny $\dots$};
	\node at(1.5,3.75) {\tiny $\dots$};
	\draw (2,0) .. controls (2,.-.25) and (3,-.25) .. (3,0);
	\draw[dashed] (2,0) .. controls (2,.25) and (3,.25) .. (3,0);
	\draw (2,0) -- (2,4);
	\draw (3,0) -- (3,4);
	\draw (2,4) .. controls (2,3.75) and (3,3.75) .. (3,4);
	\draw (2,4) .. controls (2,4.25) and (3,4.25) .. (3,4);
	\filldraw [fill=white, draw=black,rounded corners] (-.5,2.5) rectangle (3.5,3.5) node[midway] { $W'$};
} \ \tikzdiagh[scale=.5]{
	\draw (0,0) .. controls (0,.-.25) and (1,-.25) .. (1,0);
	\draw[dashed] (0,0) .. controls (0,.25) and (1,.25) .. (1,0);
	\draw (0,0) -- (0,4);
	\draw (1,0) -- (1,4);
	\draw (0,4) .. controls (0,3.75) and (1,3.75) .. (1,4);
	\draw (0,4) .. controls (0,4.25) and (1,4.25) .. (1,4);
	\node at(1.5,.25) {\tiny $\dots$};
	\node at(1.5,3.75) {\tiny $\dots$};
	\draw (2,0) .. controls (2,.-.25) and (3,-.25) .. (3,0);
	\draw[dashed] (2,0) .. controls (2,.25) and (3,.25) .. (3,0);
	\draw (2,0) -- (2,4);
	\draw (3,0) -- (3,4);
	\draw (2,4) .. controls (2,3.75) and (3,3.75) .. (3,4);
	\draw (2,4) .. controls (2,4.25) and (3,4.25) .. (3,4);
	\filldraw [fill=white, draw=black,rounded corners] (-.5,.5) rectangle (3.5,1.5) node[midway] { $W$};
} 
\]

\begin{rem}
The category $\chcobcat$ equipped with the right-then-left juxtaposition is not monoidal because in general 
 \[
 (f \otimes g) \circ (f' \otimes g') \neq (f \circ f') \otimes (g \circ g').
 \]
 In particular, this means it is not truly a horizontal composition in the 2-categorical sense. However, as we will see in  \cref{sec:linchcobcat}, $\chcobcat$ can be linearized to become a $\bZ$-graded monoidal category (alternatively, one can consider the Gray monoidal $2$-category of chronological cobordisms, giving a 3-category, see \cite{putyra14}). This means the name `horizontal composition' still make sense in our context. 
\end{rem}

Note that (pre)composing a split (resp. merge) with the twist gives the split (resp. merge) with the opposite orientation:
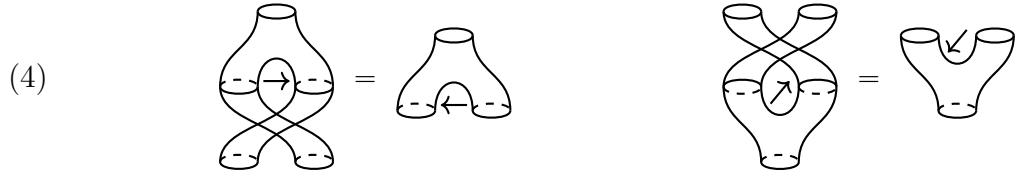
\begin{align} \label{eq:twisterorientation}
 \tikzdiagc[scale=.5]{
  \begin{scope}[shift={(0,2)}]
		\draw (0,0) .. controls (0,1) and (1,1) .. (1,2);
		\draw (1,0) .. controls (1,1) and (2,1) .. (2,0);
		\draw (3,0) .. controls (3,1) and (2,1) .. (2,2);
		\draw (0,0) .. controls (0,-.25) and (1,-.25) .. (1,0);
		\draw[dashed] (0,0) .. controls (0,.25) and (1,.25) .. (1,0);
		\draw (2,0) .. controls (2,-.25) and (3,-.25) .. (3,0);
		\draw[dashed] (2,0) .. controls (2,.25) and (3,.25) .. (3,0);
		\draw (1,2) .. controls (1,1.75) and (2,1.75) .. (2,2);
		\draw (1,2) .. controls (1,2.25) and (2,2.25) .. (2,2);
		\draw [->] (1.15,.15) -- (1.85,.15);
  \end{scope}
	\draw (0,0) .. controls (0,-.25) and (1,-.25) .. (1,0);
	\draw[dashed] (0,0) .. controls (0,.25) and (1,.25) .. (1,0);
	\draw (2,0) .. controls (2,-.25) and (3,-.25) .. (3,0);
	\draw[dashed] (2,0) .. controls (2,.25) and (3,.25) .. (3,0);
	\draw (0,0) .. controls (0,1) and (2,1) .. (2,2);
	\draw (1,0) .. controls (1,1) and (3,1) .. (3,2);
	\draw (2,0) .. controls (2,1) and (0,1) .. (0,2);
	\draw (3,0) .. controls (3,1) and (1,1) .. (1,2);
	%
	%
} 
 &\ = \ 
 \tikzdiagh[scale=.5]{
	\draw (0,0) .. controls (0,1) and (1,1) .. (1,2);
	\draw (1,0) .. controls (1,1) and (2,1) .. (2,0);
	\draw (3,0) .. controls (3,1) and (2,1) .. (2,2);
	\draw (0,0) .. controls (0,-.25) and (1,-.25) .. (1,0);
	\draw[dashed] (0,0) .. controls (0,.25) and (1,.25) .. (1,0);
	\draw (2,0) .. controls (2,-.25) and (3,-.25) .. (3,0);
	\draw[dashed] (2,0) .. controls (2,.25) and (3,.25) .. (3,0);
	\draw (1,2) .. controls (1,1.75) and (2,1.75) .. (2,2);
	\draw (1,2) .. controls (1,2.25) and (2,2.25) .. (2,2);
	\draw [<-] (1.15,.15) -- (1.85,.15);
} 
&
 \tikzdiagc[xscale=.5,yscale=-.5]{
  \begin{scope}[shift={(0,-2)}]
		\draw (0,0) .. controls (0,-.25) and (1,-.25) .. (1,0);
		\draw (0,0) .. controls (0,.25) and (1,.25) .. (1,0);
		\draw (2,0) .. controls (2,-.25) and (3,-.25) .. (3,0);
		\draw (2,0) .. controls (2,.25) and (3,.25) .. (3,0);
		\draw (0,0) .. controls (0,1) and (2,1) .. (2,2);
		\draw (1,0) .. controls (1,1) and (3,1) .. (3,2);
		\draw (2,0) .. controls (2,1) and (0,1) .. (0,2);
		\draw (3,0) .. controls (3,1) and (1,1) .. (1,2);
		\draw (0,2) .. controls (0,1.75) and (1,1.75) .. (1,2);
		\draw[dashed] (0,2) .. controls (0,2.25) and (1,2.25) .. (1,2);
		\draw (2,2) .. controls (2,1.75) and (3,1.75) .. (3,2);
		\draw[dashed] (2,2) .. controls (2,2.25) and (3,2.25) .. (3,2);
  \end{scope}
	\draw (0,0) .. controls (0,1) and (1,1) .. (1,2);
	\draw (1,0) .. controls (1,1) and (2,1) .. (2,0);
	\draw (3,0) .. controls (3,1) and (2,1) .. (2,2);
	%
	\draw[dashed] (1,2) .. controls (1,1.75) and (2,1.75) .. (2,2);
	\draw (1,2) .. controls (1,2.25) and (2,2.25) .. (2,2);
	\draw [->] (1.25,.45) -- (1.75,-.15);
} 
 &\ = \ 
 \tikzdiagh[xscale=.5,yscale=-.5]{
	\draw (0,0) .. controls (0,1) and (1,1) .. (1,2);
	\draw (1,0) .. controls (1,1) and (2,1) .. (2,0);
	\draw (3,0) .. controls (3,1) and (2,1) .. (2,2);
	\draw (0,0) .. controls (0,-.25) and (1,-.25) .. (1,0);
	\draw (0,0) .. controls (0,.25) and (1,.25) .. (1,0);
	\draw (2,0) .. controls (2,-.25) and (3,-.25) .. (3,0);
	\draw (2,0) .. controls (2,.25) and (3,.25) .. (3,0);
	\draw[dashed] (1,2) .. controls (1,1.75) and (2,1.75) .. (2,2);
	\draw (1,2) .. controls (1,2.25) and (2,2.25) .. (2,2);
	\draw [<-] (1.25,.45) -- (1.75,-.15);
} 
\end{align}

The morphism spaces in $\chcobcat$ are generated by the five elementary cobordisms:
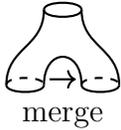
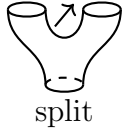
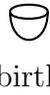
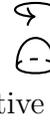
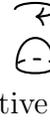
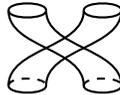
\begin{center}
\begin{tabularx}{\textwidth}{CCCCC}
\centering \tikzdiagh[scale=.5]{
	\draw (0,0) .. controls (0,1) and (1,1) .. (1,2);
	\draw (1,0) .. controls (1,1) and (2,1) .. (2,0);
	\draw (3,0) .. controls (3,1) and (2,1) .. (2,2);
	\draw (0,0) .. controls (0,-.25) and (1,-.25) .. (1,0);
	\draw[dashed] (0,0) .. controls (0,.25) and (1,.25) .. (1,0);
	\draw (2,0) .. controls (2,-.25) and (3,-.25) .. (3,0);
	\draw[dashed] (2,0) .. controls (2,.25) and (3,.25) .. (3,0);
	\draw (1,2) .. controls (1,1.75) and (2,1.75) .. (2,2);
	\draw (1,2) .. controls (1,2.25) and (2,2.25) .. (2,2);
	\draw [->] (1.15,.15) -- (1.85,.15);
} 
&
 \tikzdiagh[xscale=.5,yscale=-.5]{
	\draw (0,0) .. controls (0,1) and (1,1) .. (1,2);
	\draw (1,0) .. controls (1,1) and (2,1) .. (2,0);
	\draw (3,0) .. controls (3,1) and (2,1) .. (2,2);
	\draw (0,0) .. controls (0,-.25) and (1,-.25) .. (1,0);
	\draw (0,0) .. controls (0,.25) and (1,.25) .. (1,0);
	\draw (2,0) .. controls (2,-.25) and (3,-.25) .. (3,0);
	\draw (2,0) .. controls (2,.25) and (3,.25) .. (3,0);
	\draw[dashed] (1,2) .. controls (1,1.75) and (2,1.75) .. (2,2);
	\draw (1,2) .. controls (1,2.25) and (2,2.25) .. (2,2);
	\draw [->] (1.25,.45) -- (1.75,-.15);
} 
 &
 \tikzdiagc[scale=.5]{
	\draw (1,2) .. controls (1,1) and (2,1) .. (2,2);
	\draw (1,2) .. controls (1,1.75) and (2,1.75) .. (2,2);
	\draw (1,2) .. controls (1,2.25) and (2,2.25) .. (2,2);
} 
 &
 \tikzdiagh[scale=.5]{
	\draw (1,0) .. controls (1,1) and (2,1) .. (2,0);
	\draw (1,0) .. controls (1,-.25) and (2,-.25) .. (2,0);
	\draw[dashed] (1,0) .. controls (1,.25) and (2,.25) .. (2,0);
	\draw[->] (1.5,1.5) [partial ellipse=0:270:3ex and 1ex];
} 
 &
 \tikzdiagh[xscale=-.5,yscale=.5]{
	\draw (1,0) .. controls (1,1) and (2,1) .. (2,0);
	\draw (1,0) .. controls (1,-.25) and (2,-.25) .. (2,0);
	\draw[dashed] (1,0) .. controls (1,.25) and (2,.25) .. (2,0);
	\draw[->] (1.5,1.5) [partial ellipse=0:270:3ex and 1ex];
} 
\\
 merge 
&
 split
 &
 birth
 &
 positive death
 &
 negative death
\end{tabularx}
\end{center}
together with a twist $$\tikzdiagc[scale=.5]{
	\draw (0,0) .. controls (0,-.25) and (1,-.25) .. (1,0);
	\draw[dashed] (0,0) .. controls (0,.25) and (1,.25) .. (1,0);
	\draw (2,0) .. controls (2,-.25) and (3,-.25) .. (3,0);
	\draw[dashed] (2,0) .. controls (2,.25) and (3,.25) .. (3,0);
	\draw (0,0) .. controls (0,1) and (2,1) .. (2,2);
	\draw (1,0) .. controls (1,1) and (3,1) .. (3,2);
	\draw (2,0) .. controls (2,1) and (0,1) .. (0,2);
	\draw (3,0) .. controls (3,1) and (1,1) .. (1,2);
	\draw (0,2) .. controls (0,1.75) and (1,1.75) .. (1,2);
	\draw (0,2) .. controls (0,2.25) and (1,2.25) .. (1,2);
	\draw (2,2) .. controls (2,1.75) and (3,1.75) .. (3,2);
	\draw (2,2) .. controls (2,2.25) and (3,2.25) .. (3,2);
}$$ acting as a Gray symmetry (see \cite{putyra14}). 
The unit object is given by the empty space $\emptyset$.

In other words, the hom-spaces of $\chcobcat$ can be combinatorially described as generated by vertical and horizontal compositions of the elementary cobordisms and the twist, modulo the relations that the twist $\tau$ distantly commutes with everything: 
\begin{align}\label{eq:twistcommutes}
(\tau \otimes 1) \circ (1 \otimes f) &= (1 \otimes f)  \circ (\tau \otimes 1), 
&
(1 \otimes \tau) \circ (f \otimes 1) &= (f \otimes 1)  \circ (1 \otimes \tau), 
\end{align}
for all cobordism $f$, and respects the symmetry relations:
\begin{align}\label{eq:symtwist}
\tau^2 &= 1,
&
(\tau \otimes 1) \circ (1 \otimes \tau) \circ (\tau \otimes 1) = (1 \otimes \tau)  \circ (\tau \otimes 1) \circ (1 \otimes \tau),
\end{align}
and  the exchange relations:
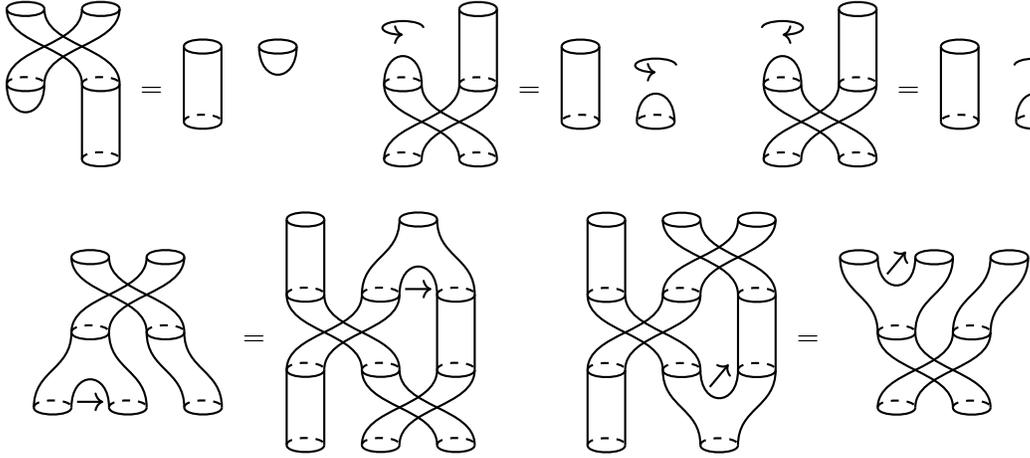
\begin{align*}
 \tikzdiagh[scale=.5]{
 
  \begin{scope}[shift={(0,2)}]
		%
		\draw (0,0) .. controls (0,1) and (2,1) .. (2,2);
		\draw (1,0) .. controls (1,1) and (3,1) .. (3,2);
		\draw (2,0) .. controls (2,1) and (0,1) .. (0,2);
		\draw (3,0) .. controls (3,1) and (1,1) .. (1,2);
		\draw (0,2) .. controls (0,1.75) and (1,1.75) .. (1,2);
		\draw (0,2) .. controls (0,2.25) and (1,2.25) .. (1,2);
		\draw (2,2) .. controls (2,1.75) and (3,1.75) .. (3,2);
		\draw (2,2) .. controls (2,2.25) and (3,2.25) .. (3,2);
  \end{scope}

	\draw (2,0) .. controls (2,.-.25) and (3,-.25) .. (3,0);
	\draw[dashed] (2,0) .. controls (2,.25) and (3,.25) .. (3,0);
	\draw (2,0) -- (2,2);
	\draw (3,0) -- (3,2);
	\draw (2,2) .. controls (2,1.75) and (3,1.75) .. (3,2);
	\draw[dashed] (2,2) .. controls (2,2.25) and (3,2.25) .. (3,2);
	\draw (0,2) .. controls (0,1) and (1,1) .. (1,2);
	\draw (0,2) .. controls (0,1.75) and (1,1.75) .. (1,2);
	\draw[dashed] (0,2) .. controls (0,2.25) and (1,2.25) .. (1,2);
}
\ &= \ 
\tikzdiagh[scale=.5]{
  \begin{scope}[shift={(2,0)}]
	\draw (1,2) .. controls (1,1) and (2,1) .. (2,2);
	\draw (1,2) .. controls (1,1.75) and (2,1.75) .. (2,2);
	\draw (1,2) .. controls (1,2.25) and (2,2.25) .. (2,2);
  \end{scope}
	\draw (1,0) .. controls (1,.-.25) and (2,-.25) .. (2,0);
	\draw[dashed] (1,0) .. controls (1,.25) and (2,.25) .. (2,0);
	\draw (1,0) -- (1,2);
	\draw (2,0) -- (2,2);
	\draw (1,2) .. controls (1,1.75) and (2,1.75) .. (2,2);
	\draw (1,2) .. controls (1,2.25) and (2,2.25) .. (2,2);
}
&
 \tikzdiagh[xscale=.5,yscale=-.5]{
 
  \begin{scope}[shift={(0,2)}]
		%
		\draw (0,0) .. controls (0,1) and (2,1) .. (2,2);
		\draw (1,0) .. controls (1,1) and (3,1) .. (3,2);
		\draw (2,0) .. controls (2,1) and (0,1) .. (0,2);
		\draw (3,0) .. controls (3,1) and (1,1) .. (1,2);
		\draw[dashed] (0,2) .. controls (0,1.75) and (1,1.75) .. (1,2);
		\draw (0,2) .. controls (0,2.25) and (1,2.25) .. (1,2);
		\draw[dashed] (2,2) .. controls (2,1.75) and (3,1.75) .. (3,2);
		\draw (2,2) .. controls (2,2.25) and (3,2.25) .. (3,2);
  \end{scope}

	\draw (2,0) .. controls (2,.-.25) and (3,-.25) .. (3,0);
	\draw (2,0) .. controls (2,.25) and (3,.25) .. (3,0);
	\draw (2,0) -- (2,2);
	\draw (3,0) -- (3,2);
	\draw[dashed] (2,2) .. controls (2,1.75) and (3,1.75) .. (3,2);
	\draw (2,2) .. controls (2,2.25) and (3,2.25) .. (3,2);
	\draw (0,2) .. controls (0,1) and (1,1) .. (1,2);
	\draw[dashed] (0,2) .. controls (0,1.75) and (1,1.75) .. (1,2);
	\draw (0,2) .. controls (0,2.25) and (1,2.25) .. (1,2);
	\draw[<-] (.5,.5) [partial ellipse=90:360:3ex and 1ex];
}
\ &= \ 
 \tikzdiagh[xscale=-.5,yscale=-.5]{
	\draw (2,0) .. controls (2,.-.25) and (3,-.25) .. (3,0);
	\draw (2,0) .. controls (2,.25) and (3,.25) .. (3,0);
	\draw (2,0) -- (2,2);
	\draw (3,0) -- (3,2);
	\draw[dashed] (2,2) .. controls (2,1.75) and (3,1.75) .. (3,2);
	\draw (2,2) .. controls (2,2.25) and (3,2.25) .. (3,2);
	\draw (0,2) .. controls (0,1) and (1,1) .. (1,2);
	\draw[dashed] (0,2) .. controls (0,1.75) and (1,1.75) .. (1,2);
	\draw (0,2) .. controls (0,2.25) and (1,2.25) .. (1,2);
	\draw[->] (.5,.5) [partial ellipse=180:450:3ex and 1ex];
}
&
 \tikzdiagh[xscale=.5,yscale=-.5]{
 
  \begin{scope}[shift={(0,2)}]
		%
		\draw (0,0) .. controls (0,1) and (2,1) .. (2,2);
		\draw (1,0) .. controls (1,1) and (3,1) .. (3,2);
		\draw (2,0) .. controls (2,1) and (0,1) .. (0,2);
		\draw (3,0) .. controls (3,1) and (1,1) .. (1,2);
		\draw[dashed] (0,2) .. controls (0,1.75) and (1,1.75) .. (1,2);
		\draw (0,2) .. controls (0,2.25) and (1,2.25) .. (1,2);
		\draw[dashed] (2,2) .. controls (2,1.75) and (3,1.75) .. (3,2);
		\draw (2,2) .. controls (2,2.25) and (3,2.25) .. (3,2);
  \end{scope}

	\draw (2,0) .. controls (2,.-.25) and (3,-.25) .. (3,0);
	\draw (2,0) .. controls (2,.25) and (3,.25) .. (3,0);
	\draw (2,0) -- (2,2);
	\draw (3,0) -- (3,2);
	\draw[dashed] (2,2) .. controls (2,1.75) and (3,1.75) .. (3,2);
	\draw (2,2) .. controls (2,2.25) and (3,2.25) .. (3,2);
	\draw (0,2) .. controls (0,1) and (1,1) .. (1,2);
	\draw[dashed] (0,2) .. controls (0,1.75) and (1,1.75) .. (1,2);
	\draw (0,2) .. controls (0,2.25) and (1,2.25) .. (1,2);
	\draw[->] (.5,.5) [partial ellipse=180:450:3ex and 1ex];
}
\ &= \ 
 \tikzdiagh[xscale=-.5,yscale=-.5]{
	\draw (2,0) .. controls (2,.-.25) and (3,-.25) .. (3,0);
	\draw (2,0) .. controls (2,.25) and (3,.25) .. (3,0);
	\draw (2,0) -- (2,2);
	\draw (3,0) -- (3,2);
	\draw[dashed] (2,2) .. controls (2,1.75) and (3,1.75) .. (3,2);
	\draw (2,2) .. controls (2,2.25) and (3,2.25) .. (3,2);
	\draw (0,2) .. controls (0,1) and (1,1) .. (1,2);
	\draw[dashed] (0,2) .. controls (0,1.75) and (1,1.75) .. (1,2);
	\draw (0,2) .. controls (0,2.25) and (1,2.25) .. (1,2);
	\draw[<-] (.5,.5) [partial ellipse=90:360:3ex and 1ex];
}
\end{align*}
\begin{align*}
\tikzdiagh[scale=.5]{
  \begin{scope}[shift={(1,2)}]
		%
		\draw (0,0) .. controls (0,1) and (2,1) .. (2,2);
		\draw (1,0) .. controls (1,1) and (3,1) .. (3,2);
		\draw (2,0) .. controls (2,1) and (0,1) .. (0,2);
		\draw (3,0) .. controls (3,1) and (1,1) .. (1,2);
		\draw (0,2) .. controls (0,1.75) and (1,1.75) .. (1,2);
		\draw (0,2) .. controls (0,2.25) and (1,2.25) .. (1,2);
		\draw (2,2) .. controls (2,1.75) and (3,1.75) .. (3,2);
		\draw (2,2) .. controls (2,2.25) and (3,2.25) .. (3,2);
  \end{scope}
	\draw (0,0) .. controls (0,1) and (1,1) .. (1,2);
	\draw (1,0) .. controls (1,1) and (2,1) .. (2,0);
	\draw (3,0) .. controls (3,1) and (2,1) .. (2,2);
	\draw (0,0) .. controls (0,-.25) and (1,-.25) .. (1,0);
	\draw[dashed] (0,0) .. controls (0,.25) and (1,.25) .. (1,0);
	\draw (2,0) .. controls (2,-.25) and (3,-.25) .. (3,0);
	\draw[dashed] (2,0) .. controls (2,.25) and (3,.25) .. (3,0);
	\draw (1,2) .. controls (1,1.75) and (2,1.75) .. (2,2);
	\draw[dashed] (1,2) .. controls (1,2.25) and (2,2.25) .. (2,2);
	\draw [->] (1.15,.15) -- (1.85,.15);
	\draw (4,0) .. controls (4,-.25) and (5,-.25) .. (5,0);
	\draw[dashed] (4,0) .. controls (4,.25) and (5,.25) .. (5,0);
	\draw (3,2) .. controls (3,1.75) and (4,1.75) .. (4,2);
	\draw[dashed] (3,2) .. controls (3,2.25) and (4,2.25) .. (4,2);
	\draw (4,0) .. controls (4,1) and (3,1) .. (3,2);
	\draw (5,0) .. controls (5,1) and (4,1) .. (4,2);
}
\ &= \ 
\tikzdiagh[scale=.5]{
%
%
  \begin{scope}[shift={(-1,4)}]
	\draw (1,0) .. controls (1,.-.25) and (2,-.25) .. (2,0);
	\draw[dashed] (1,0) .. controls (1,.25) and (2,.25) .. (2,0);
	\draw (1,0) -- (1,2);
	\draw (2,0) -- (2,2);
	\draw (1,2) .. controls (1,1.75) and (2,1.75) .. (2,2);
	\draw (1,2) .. controls (1,2.25) and (2,2.25) .. (2,2);
  \end{scope}
%
  \begin{scope}[shift={(2,4)}]
	\draw (0,0) .. controls (0,1) and (1,1) .. (1,2);
	\draw (1,0) .. controls (1,1) and (2,1) .. (2,0);
	\draw (3,0) .. controls (3,1) and (2,1) .. (2,2);
	\draw (0,0) .. controls (0,-.25) and (1,-.25) .. (1,0);
	\draw[dashed] (0,0) .. controls (0,.25) and (1,.25) .. (1,0);
	\draw (2,0) .. controls (2,-.25) and (3,-.25) .. (3,0);
	\draw[dashed] (2,0) .. controls (2,.25) and (3,.25) .. (3,0);
	\draw (1,2) .. controls (1,1.75) and (2,1.75) .. (2,2);
	\draw (1,2) .. controls (1,2.25) and (2,2.25) .. (2,2);
	\draw [->] (1.15,.15) -- (1.85,.15);
  \end{scope}
%
  \begin{scope}[shift={(0,2)}]
		\draw (0,0) .. controls (0,-.25) and (1,-.25) .. (1,0);
		\draw[dashed] (0,0) .. controls (0,.25) and (1,.25) .. (1,0);
		\draw (2,0) .. controls (2,-.25) and (3,-.25) .. (3,0);
		\draw[dashed] (2,0) .. controls (2,.25) and (3,.25) .. (3,0);
		\draw (0,0) .. controls (0,1) and (2,1) .. (2,2);
		\draw (1,0) .. controls (1,1) and (3,1) .. (3,2);
		\draw (2,0) .. controls (2,1) and (0,1) .. (0,2);
		\draw (3,0) .. controls (3,1) and (1,1) .. (1,2);
		%
		%
  \end{scope}
%
%
  \begin{scope}[shift={(3,2)}]
	\draw (1,0) .. controls (1,.-.25) and (2,-.25) .. (2,0);
	\draw[dashed] (1,0) .. controls (1,.25) and (2,.25) .. (2,0);
	\draw (1,0) -- (1,2);
	\draw (2,0) -- (2,2);
	%
  \end{scope}
%
%
%
  \begin{scope}[shift={(-1,0)}]
	\draw (1,0) .. controls (1,.-.25) and (2,-.25) .. (2,0);
	\draw[dashed] (1,0) .. controls (1,.25) and (2,.25) .. (2,0);
	\draw (1,0) -- (1,2);
	\draw (2,0) -- (2,2);
	%
  \end{scope}
%
%
  \begin{scope}[shift={(2,0)}]
		\draw (0,0) .. controls (0,-.25) and (1,-.25) .. (1,0);
		\draw[dashed] (0,0) .. controls (0,.25) and (1,.25) .. (1,0);
		\draw (2,0) .. controls (2,-.25) and (3,-.25) .. (3,0);
		\draw[dashed] (2,0) .. controls (2,.25) and (3,.25) .. (3,0);
		\draw (0,0) .. controls (0,1) and (2,1) .. (2,2);
		\draw (1,0) .. controls (1,1) and (3,1) .. (3,2);
		\draw (2,0) .. controls (2,1) and (0,1) .. (0,2);
		\draw (3,0) .. controls (3,1) and (1,1) .. (1,2);
		%
		%
  \end{scope}
}
&
\tikzdiagh[xscale=.5,yscale=-.5]{
%
%
  \begin{scope}[shift={(-1,4)}]
	\draw[dashed] (1,0) .. controls (1,.-.25) and (2,-.25) .. (2,0);
	\draw (1,0) .. controls (1,.25) and (2,.25) .. (2,0);
	\draw (1,0) -- (1,2);
	\draw (2,0) -- (2,2);
	\draw[dashed] (1,2) .. controls (1,1.75) and (2,1.75) .. (2,2);
	\draw (1,2) .. controls (1,2.25) and (2,2.25) .. (2,2);
  \end{scope}
%
  \begin{scope}[shift={(2,4)}]
	\draw (0,0) .. controls (0,1) and (1,1) .. (1,2);
	\draw (1,0) .. controls (1,1) and (2,1) .. (2,0);
	\draw (3,0) .. controls (3,1) and (2,1) .. (2,2);
	\draw[dashed](0,0) .. controls (0,-.25) and (1,-.25) .. (1,0);
	\draw (0,0) .. controls (0,.25) and (1,.25) .. (1,0);
	\draw[dashed]  (2,0) .. controls (2,-.25) and (3,-.25) .. (3,0);
	\draw(2,0) .. controls (2,.25) and (3,.25) .. (3,0);
	\draw[dashed] (1,2) .. controls (1,1.75) and (2,1.75) .. (2,2);
	\draw (1,2) .. controls (1,2.25) and (2,2.25) .. (2,2);
	\draw [->] (1.25,.45) -- (1.75,-.15);
  \end{scope}
%
  \begin{scope}[shift={(0,2)}]
		\draw[dashed] (0,0) .. controls (0,-.25) and (1,-.25) .. (1,0);
		\draw (0,0) .. controls (0,.25) and (1,.25) .. (1,0);
		\draw[dashed] (2,0) .. controls (2,-.25) and (3,-.25) .. (3,0);
		\draw (2,0) .. controls (2,.25) and (3,.25) .. (3,0);
		\draw (0,0) .. controls (0,1) and (2,1) .. (2,2);
		\draw (1,0) .. controls (1,1) and (3,1) .. (3,2);
		\draw (2,0) .. controls (2,1) and (0,1) .. (0,2);
		\draw (3,0) .. controls (3,1) and (1,1) .. (1,2);
		%
		%
  \end{scope}
%
%
  \begin{scope}[shift={(3,2)}]
	\draw[dashed] (1,0) .. controls (1,.-.25) and (2,-.25) .. (2,0);
	\draw (1,0) .. controls (1,.25) and (2,.25) .. (2,0);
	\draw (1,0) -- (1,2);
	\draw (2,0) -- (2,2);
	%
  \end{scope}
%
%
%
  \begin{scope}[shift={(-1,0)}]
	\draw (1,0) .. controls (1,.-.25) and (2,-.25) .. (2,0);
	\draw (1,0) .. controls (1,.25) and (2,.25) .. (2,0);
	\draw (1,0) -- (1,2);
	\draw (2,0) -- (2,2);
	%
  \end{scope}
%
%
  \begin{scope}[shift={(2,0)}]
		\draw (0,0) .. controls (0,-.25) and (1,-.25) .. (1,0);
		\draw (0,0) .. controls (0,.25) and (1,.25) .. (1,0);
		\draw (2,0) .. controls (2,-.25) and (3,-.25) .. (3,0);
		\draw (2,0) .. controls (2,.25) and (3,.25) .. (3,0);
		\draw (0,0) .. controls (0,1) and (2,1) .. (2,2);
		\draw (1,0) .. controls (1,1) and (3,1) .. (3,2);
		\draw (2,0) .. controls (2,1) and (0,1) .. (0,2);
		\draw (3,0) .. controls (3,1) and (1,1) .. (1,2);
		%
		%
  \end{scope}
}
\ &= \ 
\tikzdiagh[xscale=.5,yscale=-.5]{
  \begin{scope}[shift={(1,2)}]
		%
		\draw (0,0) .. controls (0,1) and (2,1) .. (2,2);
		\draw (1,0) .. controls (1,1) and (3,1) .. (3,2);
		\draw (2,0) .. controls (2,1) and (0,1) .. (0,2);
		\draw (3,0) .. controls (3,1) and (1,1) .. (1,2);
		\draw[dashed] (0,2) .. controls (0,1.75) and (1,1.75) .. (1,2);
		\draw (0,2) .. controls (0,2.25) and (1,2.25) .. (1,2);
		\draw[dashed] (2,2) .. controls (2,1.75) and (3,1.75) .. (3,2);
		\draw (2,2) .. controls (2,2.25) and (3,2.25) .. (3,2);
  \end{scope}
	\draw (0,0) .. controls (0,1) and (1,1) .. (1,2);
	\draw (1,0) .. controls (1,1) and (2,1) .. (2,0);
	\draw (3,0) .. controls (3,1) and (2,1) .. (2,2);
	\draw (0,0) .. controls (0,-.25) and (1,-.25) .. (1,0);
	\draw (0,0) .. controls (0,.25) and (1,.25) .. (1,0);
	\draw (2,0) .. controls (2,-.25) and (3,-.25) .. (3,0);
	\draw (2,0) .. controls (2,.25) and (3,.25) .. (3,0);
	\draw[dashed] (1,2) .. controls (1,1.75) and (2,1.75) .. (2,2);
	\draw (1,2) .. controls (1,2.25) and (2,2.25) .. (2,2);
	\draw [->] (1.25,.45) -- (1.75,-.15);
	\draw (4,0) .. controls (4,-.25) and (5,-.25) .. (5,0);
	\draw (4,0) .. controls (4,.25) and (5,.25) .. (5,0);
	\draw[dashed] (3,2) .. controls (3,1.75) and (4,1.75) .. (4,2);
	\draw (3,2) .. controls (3,2.25) and (4,2.25) .. (4,2);
	\draw (4,0) .. controls (4,1) and (3,1) .. (3,2);
	\draw (5,0) .. controls (5,1) and (4,1) .. (4,2);
}
\end{align*}
and the vertical mirror images of the above relations. 

 \subsubsection{$\bZ$-graded monoidal category of graded abelian groups} \label{sec:gradedmon}
 
 Recall that a \emph{($\bZ$-)graded group} is a group with a decomposition $G = \bigoplus_{z \in \bZ} G_z$, such that $G_{z_2} G_{z_1} \subset G_{z_2 + z_1}$ for all $z_1,z_2 \in \bZ$. 
 A homomorphism $f : G \rightarrow G'$ between graded groups is \emph{homogeneous} of degree $|f| \in \bZ$ if $f(G_z) \subset G'_{z+|f|}$ for all $z \in \bZ$.  
 
 Let $\bZ\gmod$ be the ($\bZ$-)graded category of ($\bZ$-)graded abelian groups, that is the category where objects are $\bZ$-graded abelian groups, and morphisms are finite sums of homogenenous group homomorphisms. 
 Consider the \emph{ $\bZ$-graded tensor product} defined by
 \begin{align*}
 	M \otimes N &:= \bigoplus_{z \in \bZ} (M \otimes N)_z, 
 	&
 	 (M \otimes N)_z &:= \bigoplus_{z_1 + z_2 = z} M_{z_1} \otimes_\bZ N_{z_2}.
 \end{align*}
 The unit object is given by $\bZ$ concentrated in degree zero. 
 Given two homogeneous homomorphisms $f : M \rightarrow M'$ and $g : N \rightarrow N'$, the tensor product morphism is given by
 \[
 	(f \otimes g) : (M \otimes N) \rightarrow (M' \otimes N'), \quad m \otimes n \mapsto (-1)^{|g||m|} f(m) \otimes g(n),
 \]
 for homogeneous elements $m \in M$ and $n \in N$.
 Note that $|f \otimes g| = |f| + |g|$.   
 Also, we have
 \begin{equation}\label{eq:gradcomp}
  (f \otimes g) \circ (f' \otimes g') = (-1)^{|g||f'|} (f \circ f') \otimes (g \circ g').
 \end{equation}
 Note that the sign appearing here means that $- \otimes -$ is not, strictly speaking, a bifunctor, and thus $(\bZ\gmod, \otimes)$ is not a monoidal category. It is however a `$\bZ$-graded monoidal category' as defined below (in the same sense as monoidal supercategory in \cite{supermonoidal}, see also \cite[\S4]{naisseputyra} for similar notions generalized). 

\begin{defn}
Let $\cC$ be a $\bZ\gmod$-enriched category (i.e. the hom-spaces are graded abelian groups and the composition preserves the grading and is bilinear). 
Let $- \otimes_\cC - : \cC \times \cC \rightarrow \cC$ be a map of categories such that $\id_X \otimes_\cC \id_Y = \id_{X \otimes_\cC Y}$ and respecting \cref{eq:gradcomp}. If there are natural isomorphisms respecting the same coherence relations as for the usual notion of a monoidal category, we say that $(\cC,\otimes_\cC)$ is a \emph{$\bZ$-graded monoidal category}. 
\end{defn}
 
 \begin{rem}
 The notion of $\bZ$-graded monoidal category presented here differs from the one introduced in \cite{gradedmonoidal}, and should be think of as a `monoidal structure on a category enriched over $\bZ$-graded vector spaces'. We keep the name `$\bZ$-graded monoidal category' to agree with the more general notion introduced in \cite[\S 4]{naisseputyra}.
 \end{rem}
 
 The category $\bZ\gmod$ also comes with a symmetry $\tau^\bZ_{M,N} : M \otimes N \xrightarrow{\simeq} N \otimes M$ given by
 \[
 \tau^\bZ_{M,N} (m \otimes n) := (-1)^{|m||n|} (n \otimes m).
 \]
 However, $\tau^\bZ_{M,N}$ is not natural (it is natural only up to sign depending on the  grading). 
We call this a `graded symmetry' (with the same flavor as in \cite[\S4.7]{naisseputyra}).
 
 \begin{defn}
 Let $(\cC, \otimes_\cC)$ be a graded monoidal category with a collection of isomorphisms $\tau^\cC_{M,N} : M \otimes_\cC N \xrightarrow{\simeq} N \otimes_\cC M$ for each pair $M,N \in \cC$. Suppose $\tau^\cC_{M,N}$ is \emph{graded natural} in $M,N$, meaning that the diagram
  \[
 \begin{tikzcd}[column sep=7ex, row sep=5ex]
 M \otimes_\cC N  \ar[description,phantom]{dr}{(-1)^{|f||g|}} \ar{r}{f \otimes g}   \ar[swap]{d}{\tau_{M,N}} & M' \otimes_\cC N' \ar{d}{\tau_{M',N'}}
 \\
 N \otimes_\cC M \ar[swap]{r}{g \otimes f} & N' \otimes_\cC M'
 \end{tikzcd}
 \]
commutes up to $(-1)^{|f||g|}$ for all homogeneous homomorphisms $f : M \rightarrow M'$ and $g : N \rightarrow N'$. 
If it respects the usual coherence relations of a symmetric monoidal category, then we say that $(\cC,\otimes_\cC, \tau^\cC_{-,-})$ is a \emph{$\bZ$-graded symmetric monoidal category}. 
 \end{defn}
 
Observe that $(\bZ\gmod, \otimes_\bZ, \tau^\bZ_{-,-})$ is a $\bZ$-graded symmetric monoidal category. 
 
 \smallskip
 
 Finally, a \emph{graded (symmetric) monoidal functor} is simply a $\bZ\gmod$-linear functor (i.e. it is $\bZ\,$-linear and preserves the grading) coming with the usual natural isomorphisms and respecting the same coherence relations as a (symmetric) monoidal functor. Note that there are no additional signs involved in this definition because we ask the functor to preserve the grading.
 
 \subsubsection{Definition of $\oddTQFT$}\label{ssec:oddTQFT}

We now define the odd TQFT 
\[
\oddTQFT : \chcobcat \rightarrow \bZ\gmod, 
\]
which, as we will see, is an `almost monoidal' functor. 
Let $A := v_+ \bZ \oplus v_- \bZ$ be a 2-dimensional graded abelian group with $|v_+| := 0$ and $|v_-| := 1$.  
We put
\[
\oddTQFT(S^1) := A,
\]
and in general a collection of $m$ circles is sent to the $\bZ$-graded tensor product of $m$ copies of $A$. 

\smallskip

We define $\oddTQFT$ on each elementary cobordism: 
\begin{align*}
\oddTQFT\left(\tikzdiagh[scale=.5]{
	\draw (0,0) .. controls (0,1) and (1,1) .. (1,2);
	\draw (1,0) .. controls (1,1) and (2,1) .. (2,0);
	\draw (3,0) .. controls (3,1) and (2,1) .. (2,2);
	\draw (0,0) .. controls (0,-.25) and (1,-.25) .. (1,0);
	\draw[dashed] (0,0) .. controls (0,.25) and (1,.25) .. (1,0);
	\draw (2,0) .. controls (2,-.25) and (3,-.25) .. (3,0);
	\draw[dashed] (2,0) .. controls (2,.25) and (3,.25) .. (3,0);
	\draw (1,2) .. controls (1,1.75) and (2,1.75) .. (2,2);
	\draw (1,2) .. controls (1,2.25) and (2,2.25) .. (2,2);
	\draw [->] (1.15,.15) -- (1.85,.15);
}\right) : A \otimes A \rightarrow A &:= 
\begin{cases}
v_+ \otimes v_+ \mapsto v_+, & v_+ \otimes v_- \mapsto v_-, \\
v_- \otimes v_- \mapsto 0, & v_- \otimes v_+ \mapsto v_-,
\end{cases}
\\
\oddTQFT\left(\tikzdiagh[xscale=.5,yscale=-.5]{
	\draw (0,0) .. controls (0,1) and (1,1) .. (1,2);
	\draw (1,0) .. controls (1,1) and (2,1) .. (2,0);
	\draw (3,0) .. controls (3,1) and (2,1) .. (2,2);
	\draw (0,0) .. controls (0,-.25) and (1,-.25) .. (1,0);
	\draw (0,0) .. controls (0,.25) and (1,.25) .. (1,0);
	\draw (2,0) .. controls (2,-.25) and (3,-.25) .. (3,0);
	\draw (2,0) .. controls (2,.25) and (3,.25) .. (3,0);
	\draw[dashed] (1,2) .. controls (1,1.75) and (2,1.75) .. (2,2);
	\draw (1,2) .. controls (1,2.25) and (2,2.25) .. (2,2);
	\draw [->] (1.25,.45) -- (1.75,-.15);
}\right) : A  \rightarrow A \otimes A &:= 
\begin{cases}
v_+ \mapsto v_- \otimes v_+ - v_+ \otimes v_-,  &\\
v_- \mapsto v_- \otimes v_-, &
\end{cases}
\\
\oddTQFT\left(\tikzdiagc[scale=.5]{
	\draw (1,2) .. controls (1,1) and (2,1) .. (2,2);
	\draw (1,2) .. controls (1,1.75) and (2,1.75) .. (2,2);
	\draw (1,2) .. controls (1,2.25) and (2,2.25) .. (2,2);
}\right) : \bZ \rightarrow A  &:= 
\begin{cases}
1 \mapsto v_+, & 
\end{cases}
\\
\oddTQFT\left(\tikzdiagh[scale=.5]{
	\draw (1,0) .. controls (1,1) and (2,1) .. (2,0);
	\draw (1,0) .. controls (1,-.25) and (2,-.25) .. (2,0);
	\draw[dashed] (1,0) .. controls (1,.25) and (2,.25) .. (2,0);
	\draw[->] (1.5,1.5) [partial ellipse=0:270:3ex and 1ex];
}\right) : A  \rightarrow \bZ &:= 
\begin{cases}
v_+ \mapsto 0,  &\\
v_- \mapsto -1, &
\end{cases}
\\
\oddTQFT\left(\tikzdiagh[xscale=-.5,yscale=.5]{
	\draw (1,0) .. controls (1,1) and (2,1) .. (2,0);
	\draw (1,0) .. controls (1,-.25) and (2,-.25) .. (2,0);
	\draw[dashed] (1,0) .. controls (1,.25) and (2,.25) .. (2,0);
	\draw[->] (1.5,1.5) [partial ellipse=0:270:3ex and 1ex];
}\right) : A  \rightarrow \bZ &:= 
\begin{cases}
v_+ \mapsto 0,  &\\
v_- \mapsto 1, &
\end{cases}
\end{align*}
and the twist is sent to $\tau_{A,A}$. 
We extend by $\oddTQFT(f \otimes g) := \oddTQFT(f) \otimes_\bZ \oddTQFT(g)$.

Note that we have
\begin{align*}
\oddTQFT\left( \tikzdiagh[scale=.5]{
	\draw (0,0) .. controls (0,1) and (1,1) .. (1,2);
	\draw (1,0) .. controls (1,1) and (2,1) .. (2,0);
	\draw (3,0) .. controls (3,1) and (2,1) .. (2,2);
	\draw (0,0) .. controls (0,-.25) and (1,-.25) .. (1,0);
	\draw[dashed] (0,0) .. controls (0,.25) and (1,.25) .. (1,0);
	\draw (2,0) .. controls (2,-.25) and (3,-.25) .. (3,0);
	\draw[dashed] (2,0) .. controls (2,.25) and (3,.25) .. (3,0);
	\draw (1,2) .. controls (1,1.75) and (2,1.75) .. (2,2);
	\draw (1,2) .. controls (1,2.25) and (2,2.25) .. (2,2);
	\draw [->] (1.15,.15) -- (1.85,.15);
} \right)   &=  \oddTQFT\left(  \tikzdiagh[scale=.5]{
	\draw (0,0) .. controls (0,1) and (1,1) .. (1,2);
	\draw (1,0) .. controls (1,1) and (2,1) .. (2,0);
	\draw (3,0) .. controls (3,1) and (2,1) .. (2,2);
	\draw (0,0) .. controls (0,-.25) and (1,-.25) .. (1,0);
	\draw[dashed] (0,0) .. controls (0,.25) and (1,.25) .. (1,0);
	\draw (2,0) .. controls (2,-.25) and (3,-.25) .. (3,0);
	\draw[dashed] (2,0) .. controls (2,.25) and (3,.25) .. (3,0);
	\draw (1,2) .. controls (1,1.75) and (2,1.75) .. (2,2);
	\draw (1,2) .. controls (1,2.25) and (2,2.25) .. (2,2);
	\draw [<-] (1.15,.15) -- (1.85,.15);
}  \right),
&
\oddTQFT\left(  \tikzdiagh[xscale=.5,yscale=-.5]{
	\draw (0,0) .. controls (0,1) and (1,1) .. (1,2);
	\draw (1,0) .. controls (1,1) and (2,1) .. (2,0);
	\draw (3,0) .. controls (3,1) and (2,1) .. (2,2);
	\draw (0,0) .. controls (0,-.25) and (1,-.25) .. (1,0);
	\draw (0,0) .. controls (0,.25) and (1,.25) .. (1,0);
	\draw (2,0) .. controls (2,-.25) and (3,-.25) .. (3,0);
	\draw (2,0) .. controls (2,.25) and (3,.25) .. (3,0);
	\draw[dashed] (1,2) .. controls (1,1.75) and (2,1.75) .. (2,2);
	\draw (1,2) .. controls (1,2.25) and (2,2.25) .. (2,2);
	\draw [->] (1.25,.45) -- (1.75,-.15);
} \right)    &= -  \oddTQFT\left(  \tikzdiagh[xscale=.5,yscale=-.5]{
	\draw (0,0) .. controls (0,1) and (1,1) .. (1,2);
	\draw (1,0) .. controls (1,1) and (2,1) .. (2,0);
	\draw (3,0) .. controls (3,1) and (2,1) .. (2,2);
	\draw (0,0) .. controls (0,-.25) and (1,-.25) .. (1,0);
	\draw (0,0) .. controls (0,.25) and (1,.25) .. (1,0);
	\draw (2,0) .. controls (2,-.25) and (3,-.25) .. (3,0);
	\draw (2,0) .. controls (2,.25) and (3,.25) .. (3,0);
	\draw[dashed] (1,2) .. controls (1,1.75) and (2,1.75) .. (2,2);
	\draw (1,2) .. controls (1,2.25) and (2,2.25) .. (2,2);
	\draw [<-] (1.25,.45) -- (1.75,-.15);
}  \right).
\end{align*}

\subsubsection{The $\bZ\gmod$-enriched category $\linchcobcat$}\label{sec:linchcobcat}

As explained in \cite{putyra14}, the TQFT $\oddTQFT$ factorizes through a $\bZ\gmod$-enriched version of $\chcobcat$. 
For this, we first need to define the \emph{degree of a cobordism} $W$ as
\[
\deg(W) := \#\text{splits} - \#\text{deaths}.
\]

\begin{defn}
Let $\linchcobcat$ be the $\bZ\gmod$-enriched category with the same objects as $\chcobcat$, and with hom-spaces given by formal $\bZ$-linear combinations of chronological cobordisms subject to the relations: 
\begin{align}\label{eq:reverseorientation}
 \tikzdiagh[scale=.5]{
	\draw (0,0) .. controls (0,1) and (1,1) .. (1,2);
	\draw (1,0) .. controls (1,1) and (2,1) .. (2,0);
	\draw (3,0) .. controls (3,1) and (2,1) .. (2,2);
	\draw (0,0) .. controls (0,-.25) and (1,-.25) .. (1,0);
	\draw[dashed] (0,0) .. controls (0,.25) and (1,.25) .. (1,0);
	\draw (2,0) .. controls (2,-.25) and (3,-.25) .. (3,0);
	\draw[dashed] (2,0) .. controls (2,.25) and (3,.25) .. (3,0);
	\draw (1,2) .. controls (1,1.75) and (2,1.75) .. (2,2);
	\draw (1,2) .. controls (1,2.25) and (2,2.25) .. (2,2);
	\draw [->] (1.15,.15) -- (1.85,.15);
}\   &= \ \tikzdiagh[scale=.5]{
	\draw (0,0) .. controls (0,1) and (1,1) .. (1,2);
	\draw (1,0) .. controls (1,1) and (2,1) .. (2,0);
	\draw (3,0) .. controls (3,1) and (2,1) .. (2,2);
	\draw (0,0) .. controls (0,-.25) and (1,-.25) .. (1,0);
	\draw[dashed] (0,0) .. controls (0,.25) and (1,.25) .. (1,0);
	\draw (2,0) .. controls (2,-.25) and (3,-.25) .. (3,0);
	\draw[dashed] (2,0) .. controls (2,.25) and (3,.25) .. (3,0);
	\draw (1,2) .. controls (1,1.75) and (2,1.75) .. (2,2);
	\draw (1,2) .. controls (1,2.25) and (2,2.25) .. (2,2);
	\draw [<-] (1.15,.15) -- (1.85,.15);
} 
 &
 \tikzdiagh[xscale=.5,yscale=-.5]{
	\draw (0,0) .. controls (0,1) and (1,1) .. (1,2);
	\draw (1,0) .. controls (1,1) and (2,1) .. (2,0);
	\draw (3,0) .. controls (3,1) and (2,1) .. (2,2);
	\draw (0,0) .. controls (0,-.25) and (1,-.25) .. (1,0);
	\draw (0,0) .. controls (0,.25) and (1,.25) .. (1,0);
	\draw (2,0) .. controls (2,-.25) and (3,-.25) .. (3,0);
	\draw (2,0) .. controls (2,.25) and (3,.25) .. (3,0);
	\draw[dashed] (1,2) .. controls (1,1.75) and (2,1.75) .. (2,2);
	\draw (1,2) .. controls (1,2.25) and (2,2.25) .. (2,2);
	\draw [->] (1.25,.45) -- (1.75,-.15);
}\   &= - \ \tikzdiagh[xscale=.5,yscale=-.5]{
	\draw (0,0) .. controls (0,1) and (1,1) .. (1,2);
	\draw (1,0) .. controls (1,1) and (2,1) .. (2,0);
	\draw (3,0) .. controls (3,1) and (2,1) .. (2,2);
	\draw (0,0) .. controls (0,-.25) and (1,-.25) .. (1,0);
	\draw (0,0) .. controls (0,.25) and (1,.25) .. (1,0);
	\draw (2,0) .. controls (2,-.25) and (3,-.25) .. (3,0);
	\draw (2,0) .. controls (2,.25) and (3,.25) .. (3,0);
	\draw[dashed] (1,2) .. controls (1,1.75) and (2,1.75) .. (2,2);
	\draw (1,2) .. controls (1,2.25) and (2,2.25) .. (2,2);
	\draw [<-] (1.25,.45) -- (1.75,-.15);
} 
 &
 \tikzdiagh[scale=.5]{
	\draw (1,0) .. controls (1,1) and (2,1) .. (2,0);
	\draw (1,0) .. controls (1,-.25) and (2,-.25) .. (2,0);
	\draw[dashed] (1,0) .. controls (1,.25) and (2,.25) .. (2,0);
	\draw[->] (1.5,1.5) [partial ellipse=0:270:3ex and 1ex];
}\   &= -  \ \tikzdiagh[xscale=-.5,yscale=.5]{
	\draw (1,0) .. controls (1,1) and (2,1) .. (2,0);
	\draw (1,0) .. controls (1,-.25) and (2,-.25) .. (2,0);
	\draw[dashed] (1,0) .. controls (1,.25) and (2,.25) .. (2,0);
	\draw[->] (1.5,1.5) [partial ellipse=0:270:3ex and 1ex];
} 
 \\
 \label{eq:chcobrelbirthmerge}
 \tikzdiagh[scale=.5]{
	\draw (1,0) .. controls (1,.-.25) and (2,-.25) .. (2,0);
	\draw[dashed] (1,0) .. controls (1,.25) and (2,.25) .. (2,0);
	\draw (1,0) -- (1,4);
	\draw (2,0) -- (2,4);
	\draw (1,4) .. controls (1,3.75) and (2,3.75) .. (2,4);
	\draw (1,4) .. controls (1,4.25) and (2,4.25) .. (2,4);
}\   &= \  \tikzdiagh[scale=.5]{
	\draw (2,0) .. controls (2,.-.25) and (3,-.25) .. (3,0);
	\draw[dashed] (2,0) .. controls (2,.25) and (3,.25) .. (3,0);
	\draw (2,0) -- (2,2);
	\draw (3,0) -- (3,2);
	\draw (2,2) .. controls (2,1.75) and (3,1.75) .. (3,2);
	\draw[dashed] (2,2) .. controls (2,2.25) and (3,2.25) .. (3,2);
	\draw (0,2) .. controls (0,1) and (1,1) .. (1,2);
	\draw (0,2) .. controls (0,1.75) and (1,1.75) .. (1,2);
	\draw[dashed] (0,2) .. controls (0,2.25) and (1,2.25) .. (1,2);
	\draw (0,2) .. controls (0,3) and (1,3) .. (1,4);
	\draw (1,2) .. controls (1,3) and (2,3) .. (2,2);
	\draw (3,2) .. controls (3,3) and (2,3) .. (2,4);
	\draw (1,4) .. controls (1,3.75) and (2,3.75) .. (2,4);
	\draw (1,4) .. controls (1,4.25) and (2,4.25) .. (2,4);
	\draw [<-] (1.15,2.15) -- (1.85,2.15);
} 
&
 \tikzdiagh[scale=.5]{
	\draw (1,0) .. controls (1,.-.25) and (2,-.25) .. (2,0);
	\draw[dashed] (1,0) .. controls (1,.25) and (2,.25) .. (2,0);
	\draw (1,0) -- (1,4);
	\draw (2,0) -- (2,4);
	\draw (1,4) .. controls (1,3.75) and (2,3.75) .. (2,4);
	\draw (1,4) .. controls (1,4.25) and (2,4.25) .. (2,4);
}\   &=  \  \tikzdiagh[xscale=.5,yscale=-.5]{
	\draw (2,0) .. controls (2,.-.25) and (3,-.25) .. (3,0);
	\draw (2,0) .. controls (2,.25) and (3,.25) .. (3,0);
	\draw (2,0) -- (2,2);
	\draw (3,0) -- (3,2);
	\draw[dashed]  (2,2) .. controls (2,1.75) and (3,1.75) .. (3,2);
	\draw(2,2) .. controls (2,2.25) and (3,2.25) .. (3,2);
	\draw (0,2) .. controls (0,1) and (1,1) .. (1,2);
	\draw[dashed] (0,2) .. controls (0,1.75) and (1,1.75) .. (1,2);
	\draw (0,2) .. controls (0,2.25) and (1,2.25) .. (1,2);
	\draw[<-] (.5,.5) [partial ellipse=90:360:3ex and 1ex];
	\draw (0,2) .. controls (0,3) and (1,3) .. (1,4);
	\draw (1,2) .. controls (1,3) and (2,3) .. (2,2);
	\draw (3,2) .. controls (3,3) and (2,3) .. (2,4);
	\draw[dashed]  (1,4) .. controls (1,3.75) and (2,3.75) .. (2,4);
	\draw (1,4) .. controls (1,4.25) and (2,4.25) .. (2,4);
	\draw [<-] (1.25,2.45) -- (1.75,1.85);
} 
 &
 \tikzdiagh[scale=.5]{
	\draw (1,0) .. controls (1,.-.25) and (2,-.25) .. (2,0);
	\draw[dashed] (1,0) .. controls (1,.25) and (2,.25) .. (2,0);
	\draw (1,0) -- (1,4);
	\draw (2,0) -- (2,4);
	\draw (1,4) .. controls (1,3.75) and (2,3.75) .. (2,4);
	\draw (1,4) .. controls (1,4.25) and (2,4.25) .. (2,4);
}\   &=   \  \tikzdiagh[xscale=-.5,yscale=-.5]{
	\draw (2,0) .. controls (2,.-.25) and (3,-.25) .. (3,0);
	\draw (2,0) .. controls (2,.25) and (3,.25) .. (3,0);
	\draw (2,0) -- (2,2);
	\draw (3,0) -- (3,2);
	\draw[dashed]  (2,2) .. controls (2,1.75) and (3,1.75) .. (3,2);
	\draw(2,2) .. controls (2,2.25) and (3,2.25) .. (3,2);
	\draw (0,2) .. controls (0,1) and (1,1) .. (1,2);
	\draw[dashed] (0,2) .. controls (0,1.75) and (1,1.75) .. (1,2);
	\draw (0,2) .. controls (0,2.25) and (1,2.25) .. (1,2);
	\draw[<-] (.5,.5) [partial ellipse=90:360:3ex and 1ex];
	\draw (0,2) .. controls (0,3) and (1,3) .. (1,4);
	\draw (1,2) .. controls (1,3) and (2,3) .. (2,2);
	\draw (3,2) .. controls (3,3) and (2,3) .. (2,4);
	\draw[dashed]  (1,4) .. controls (1,3.75) and (2,3.75) .. (2,4);
	\draw (1,4) .. controls (1,4.25) and (2,4.25) .. (2,4);
	\draw [<-] (1.75,2.45) -- (1.25,1.85);
} 
\end{align}
and
\begin{align}
\label{eq:cobcommute1}
 \tikzdiagh[scale=.5]{
	\draw (0,0) .. controls (0,.-.25) and (1,-.25) .. (1,0);
	\draw[dashed] (0,0) .. controls (0,.25) and (1,.25) .. (1,0);
	\draw (0,0) -- (0,4);
	\draw (1,0) -- (1,4);
	\draw (0,4) .. controls (0,3.75) and (1,3.75) .. (1,4);
	\draw (0,4) .. controls (0,4.25) and (1,4.25) .. (1,4);
	\node at(1.5,.25) {\tiny $\dots$};
	\node at(1.5,3.75) {\tiny $\dots$};
	\draw (2,0) .. controls (2,.-.25) and (3,-.25) .. (3,0);
	\draw[dashed] (2,0) .. controls (2,.25) and (3,.25) .. (3,0);
	\draw (2,0) -- (2,4);
	\draw (3,0) -- (3,4);
	\draw (2,4) .. controls (2,3.75) and (3,3.75) .. (3,4);
	\draw (2,4) .. controls (2,4.25) and (3,4.25) .. (3,4);
	\node at(3.5,.25) {\tiny $\dots$};
	\node at(3.5,3.75) {\tiny $\dots$};
	\draw (4,0) .. controls (4,.-.25) and (5,-.25) .. (5,0);
	\draw[dashed] (4,0) .. controls (4,.25) and (5,.25) .. (5,0);
	\draw (4,0) -- (4,4);
	\draw (5,0) -- (5,4);
	\draw (4,4) .. controls (4,3.75) and (5,3.75) .. (5,4);
	\draw (4,4) .. controls (4,4.25) and (5,4.25) .. (5,4);
	\filldraw [fill=white, draw=black,rounded corners] (-.5,.5) rectangle (2.5,1.5) node[midway] { $W'$};
	\filldraw [fill=white, draw=black,rounded corners] (2.5,2.5) rectangle (5.5,3.5) node[midway] { $W$};
}\   &= (-1)^{\deg(W)\deg(W')}  \ \tikzdiagh[scale=.5]{
	\draw (0,0) .. controls (0,.-.25) and (1,-.25) .. (1,0);
	\draw[dashed] (0,0) .. controls (0,.25) and (1,.25) .. (1,0);
	\draw (0,0) -- (0,4);
	\draw (1,0) -- (1,4);
	\draw (0,4) .. controls (0,3.75) and (1,3.75) .. (1,4);
	\draw (0,4) .. controls (0,4.25) and (1,4.25) .. (1,4);
	\node at(1.5,.25) {\tiny $\dots$};
	\node at(1.5,3.75) {\tiny $\dots$};
	\draw (2,0) .. controls (2,.-.25) and (3,-.25) .. (3,0);
	\draw[dashed] (2,0) .. controls (2,.25) and (3,.25) .. (3,0);
	\draw (2,0) -- (2,4);
	\draw (3,0) -- (3,4);
	\draw (2,4) .. controls (2,3.75) and (3,3.75) .. (3,4);
	\draw (2,4) .. controls (2,4.25) and (3,4.25) .. (3,4);
	\node at(3.5,.25) {\tiny $\dots$};
	\node at(3.5,3.75) {\tiny $\dots$};
	\draw (4,0) .. controls (4,.-.25) and (5,-.25) .. (5,0);
	\draw[dashed] (4,0) .. controls (4,.25) and (5,.25) .. (5,0);
	\draw (4,0) -- (4,4);
	\draw (5,0) -- (5,4);
	\draw (4,4) .. controls (4,3.75) and (5,3.75) .. (5,4);
	\draw (4,4) .. controls (4,4.25) and (5,4.25) .. (5,4);
	\filldraw [fill=white, draw=black,rounded corners] (-.5,2.5) rectangle (2.5,3.5) node[midway] { $W'$};
	\filldraw [fill=white, draw=black,rounded corners] (2.5,.5) rectangle (5.5,1.5) node[midway] { $W$};
} 
 \\
\label{eq:cobcommute2}
 \tikzdiagh[scale=.5]{
	\draw (0,0) .. controls (0,.-.25) and (1,-.25) .. (1,0);
	\draw[dashed] (0,0) .. controls (0,.25) and (1,.25) .. (1,0);
	\draw (0,0) -- (0,4);
	\draw (1,0) -- (1,4);
	\draw (0,4) .. controls (0,3.75) and (1,3.75) .. (1,4);
	\draw (0,4) .. controls (0,4.25) and (1,4.25) .. (1,4);
	\node at(1.5,.25) {\tiny $\dots$};
	\node at(1.5,3.75) {\tiny $\dots$};
	\draw (2,0) .. controls (2,.-.25) and (3,-.25) .. (3,0);
	\draw[dashed] (2,0) .. controls (2,.25) and (3,.25) .. (3,0);
	\draw (2,0) -- (2,4);
	\draw (3,0) -- (3,4);
	\draw (2,4) .. controls (2,3.75) and (3,3.75) .. (3,4);
	\draw (2,4) .. controls (2,4.25) and (3,4.25) .. (3,4);
	\filldraw [fill=white, draw=black,rounded corners] (-.5,.5) rectangle (3.5,1.5) node[midway] { $W'$};
} \  \tikzdiagh[scale=.5]{
	\draw (0,0) .. controls (0,.-.25) and (1,-.25) .. (1,0);
	\draw[dashed] (0,0) .. controls (0,.25) and (1,.25) .. (1,0);
	\draw (0,0) -- (0,4);
	\draw (1,0) -- (1,4);
	\draw (0,4) .. controls (0,3.75) and (1,3.75) .. (1,4);
	\draw (0,4) .. controls (0,4.25) and (1,4.25) .. (1,4);
	\node at(1.5,.25) {\tiny $\dots$};
	\node at(1.5,3.75) {\tiny $\dots$};
	\draw (2,0) .. controls (2,.-.25) and (3,-.25) .. (3,0);
	\draw[dashed] (2,0) .. controls (2,.25) and (3,.25) .. (3,0);
	\draw (2,0) -- (2,4);
	\draw (3,0) -- (3,4);
	\draw (2,4) .. controls (2,3.75) and (3,3.75) .. (3,4);
	\draw (2,4) .. controls (2,4.25) and (3,4.25) .. (3,4);
	\filldraw [fill=white, draw=black,rounded corners] (-.5,2.5) rectangle (3.5,3.5) node[midway] { $W$};
} \   &= (-1)^{\deg(W)\deg(W')} \  \tikzdiagh[scale=.5]{
	\draw (0,0) .. controls (0,.-.25) and (1,-.25) .. (1,0);
	\draw[dashed] (0,0) .. controls (0,.25) and (1,.25) .. (1,0);
	\draw (0,0) -- (0,4);
	\draw (1,0) -- (1,4);
	\draw (0,4) .. controls (0,3.75) and (1,3.75) .. (1,4);
	\draw (0,4) .. controls (0,4.25) and (1,4.25) .. (1,4);
	\node at(1.5,.25) {\tiny $\dots$};
	\node at(1.5,3.75) {\tiny $\dots$};
	\draw (2,0) .. controls (2,.-.25) and (3,-.25) .. (3,0);
	\draw[dashed] (2,0) .. controls (2,.25) and (3,.25) .. (3,0);
	\draw (2,0) -- (2,4);
	\draw (3,0) -- (3,4);
	\draw (2,4) .. controls (2,3.75) and (3,3.75) .. (3,4);
	\draw (2,4) .. controls (2,4.25) and (3,4.25) .. (3,4);
	\filldraw [fill=white, draw=black,rounded corners] (-.5,2.5) rectangle (3.5,3.5) node[midway] { $W'$};
} \ \tikzdiagh[scale=.5]{
	\draw (0,0) .. controls (0,.-.25) and (1,-.25) .. (1,0);
	\draw[dashed] (0,0) .. controls (0,.25) and (1,.25) .. (1,0);
	\draw (0,0) -- (0,4);
	\draw (1,0) -- (1,4);
	\draw (0,4) .. controls (0,3.75) and (1,3.75) .. (1,4);
	\draw (0,4) .. controls (0,4.25) and (1,4.25) .. (1,4);
	\node at(1.5,.25) {\tiny $\dots$};
	\node at(1.5,3.75) {\tiny $\dots$};
	\draw (2,0) .. controls (2,.-.25) and (3,-.25) .. (3,0);
	\draw[dashed] (2,0) .. controls (2,.25) and (3,.25) .. (3,0);
	\draw (2,0) -- (2,4);
	\draw (3,0) -- (3,4);
	\draw (2,4) .. controls (2,3.75) and (3,3.75) .. (3,4);
	\draw (2,4) .. controls (2,4.25) and (3,4.25) .. (3,4);
	\filldraw [fill=white, draw=black,rounded corners] (-.5,.5) rectangle (3.5,1.5) node[midway] { $W$};
} 
\end{align}
for all cobordisms $W'$ and $W$.
\end{defn}

\begin{rem}\label{rem:redundantrelchcob}
Since $\linchcobcat$ is defined as a quotient of $\chcobcat$, which is topologically defined, all pictures in \cref{eq:chcobrelbirthmerge} can be rotated along the vertical axis by 180 degrees. Alternatively, this can be deduced from the exchange relations and \cref{eq:reverseorientation}.
Also note that the third relation of \cref{eq:chcobrelbirthmerge} is redundant with the second one together with \cref{eq:reverseorientation}.
\end{rem}

Note that $\linchcobcat$ is a $\bZ$-graded symmetric monoidal category, with the graded symmetry given by using the twist. It gives a graded natural isomorphism because of \cref{eq:cobcommute1} and \cref{eq:cobcommute2}.

\smallskip

Then, there is an obvious functor $\chcobcat \rightarrow \linchcobcat$, and we obtain an induced functor $\linoddTQFT : \linchcobcat \rightarrow \bZ\gmod$, with a commutative diagram:
\[
\begin{tikzcd}
\chcobcat \ar{dr} \ar{rr}{\oddTQFT} &&  \bZ\gmod \\
&\linchcobcat \ar[swap]{ur}{\linoddTQFT}&
\end{tikzcd}
\]
It is easy to see that the functor $\oddTQFT$ verifies the relations  Eq.~\eqref{eq:reverseorientation} - \eqref{eq:cobcommute2}, and thus $\linoddTQFT$ is well-defined. 
Moreover, it is a graded symmetric monoidal functor.

\subsection{Pullbacks and pushforwards}

Let us recall some standard facts about the cohomology of manifolds. 
Given a continuous map $\varphi : X \rightarrow Y$ between topological spaces, one can consider the pullback and pushforward 
\begin{align*}
\varphi^*: H^*(Y) &\rightarrow H^*(X) \text{ and}\\
\varphi_*: H_*(X) &\rightarrow H_*(Y).
\intertext{If $X$ and $Y$ are furthermore closed (i.e. compact and without boundary) oriented manifolds of dimension $|X|$ and $|Y|$ respectively, one can also construct an \emph{exceptional} pushforward (also called \emph{Umkehr map}, \emph{shriek map} or \emph{transfer map} in the literature)}
\varphi_! : H^*(X) &\rightarrow H^*(Y).
\end{align*}
For this, we fix a choice of orientations for $X$ and $Y$, that is a choice of fundamental classes $[f_X] \in H_{|X|}(X)$ and $[f_Y] \in H_{|Y|}(Y)$. 
Then, Poincar\'e duality gives an isomorphism $P_X : H_*(X)  \xrightarrow{\simeq}  H^*(X)$ by capping  with the fundamental class:
\begin{align*}
P_X &: H^*(X) \xrightarrow{\simeq} H_*(X), & \omega &\mapsto [f_X] \frown \omega,
\end{align*}
and similarly for $Y$. 
The pushforward $\varphi_!$ is defined in terms of the following commutative diagram:
\[
\begin{tikzcd}
H^*(X) 
\ar["P_X"']{d} 
\ar{d}{\vsimeq}
\ar{r}{\varphi_!} 
&
 H^*(Y) 
\\
H_*(X) \ar[swap]{r}{\varphi_*} & 
H_*(Y). 
\ar["P_Y^{-1}"']{u}
\ar{u}{\vsimeqop}
\end{tikzcd}
\]

We consider $H^*(X)$ and $H_*(X)$  in $\bZ\gmod$, where $H^i(X)$ is in degree $i$ and  $H_i(X)$ in degree $-i.$  
With this convention, the cap product $\frown : H_i(X) \times H^j(X) \rightarrow H_{j-i}(X)$  preserves the grading. 
In particular the Poincar\'e isomorphism $P_X : H^*(X) \xrightarrow{\simeq} H_*(X)$  has degree $-|X|$. 
Thus, the pushforward $\varphi_! : H^*(X) \rightarrow H^*(Y)$ has degree $|\varphi_!| = |Y| - |X|$.
\subsubsection{Pushfowards and the K\"unneth formula}

Let $X, X', Y$ and $Y'$ be closed oriented manifolds with torsion-free homology groups. Let $[f_X]$ and $[f_Y]$ be their respective fundamental classes. 
 Consider the products $X \times Y$ (resp. $X' \times Y'$) equipped with orientation induced by $X$ and $Y$ (resp. $X'$ and $Y'$), that is $[f_{X \times Y}] := [f_X] \times [f_Y]$.
 Suppose we have continuous maps $\varphi : X \rightarrow X'$ and $\psi : Y \rightarrow Y'$.

\begin{lem}\label{lem:poincarekunneth}
Under the conditions above, we have a commutative diagram
\begin{equation*}
\begin{tikzcd}
H^*(X) \otimes H^*(Y) 
\ar{r}{P_X \otimes P_Y}
\ar{d}{\vsimeq}
\ar[swap]{d}{- \smile -}
&[10ex]
H_*(X) \otimes H_*(Y)
\ar[swap]{d}{\vsimeqop}
\ar{d}{- \times -}
\\
H^*(X \times Y)
\ar[swap]{r}{P_{X \times Y}}
&
H_*(X \times Y)
\end{tikzcd}
\end{equation*}
where the vertical arrows are given by the K\"unneth formula.
\end{lem}

\begin{proof}
For $x \otimes y \in H^*(X) \otimes H^*(Y)$, we compute
\begin{align*}
(-1)^{|x||P_Y|} P_X(x) \otimes P_Y(y) &= 
(-1)^{-|x||Y|} ([f_X] \frown x) \times ([f_Y] \frown y)
\\
&= 
[f_{X \times Y}] \frown (x \smile y)
= P_{X \times Y}(x \smile y),
\end{align*}
using the cross/cap product formula
\begin{equation}\label{eq:crosscap}
([a] \frown x) \times ([b] \frown y) = (-1)^{|x||[b]|} ([a] \times [b]) \frown (x \smile y),
\end{equation}
from \cite[Chapter XI, Theorem 5.4]{bredon}.
\end{proof}

\begin{lem}\label{prop:comppushkunneth}
Under the conditions above, we have a commutative diagram
\begin{equation*}
\begin{tikzcd}
H^*(X) \otimes H^*(Y) 
\ar{r}{(-1)^{|\varphi_!||Y'|} (\varphi_! \otimes \psi_!)}
\ar{d}{\vsimeq}
\ar[swap]{d}{- \smile -}
& [10ex]
H^*(X') \otimes H^*(Y')
\ar[swap]{d}{\vsimeqop}
\ar{d}{- \smile -}
\\
H^*(X \times Y) 
\ar[swap]{r}{(\varphi \times \psi)_!}
& 
H^*(X' \times Y'),
\end{tikzcd}
\end{equation*}
where the vertical arrows are given by the K\"unneth formula. 
\end{lem}

\begin{proof}
Consider the following cube
\begin{equation*}
\begin{tikzcd}[column sep=-2ex]
&H_*(X) \otimes H_*(Y) \ar{rr}{\varphi_* \otimes \psi_*} 
\ar[pos=.8]{dd}{\vsimeq}
 &&
 H_*(X') \otimes H_*(Y') 
 \ar{dd}{\vsimeq}
\\
H^*(X) \otimes H^*(Y) \ar[crossing over,"\varphi_! \otimes \psi_!", pos=.7]{rr} 
\ar{dd}{\vsimeq}
 \ar["P_X \otimes P_Y" description]{ru} && H^*(X') \otimes H^*(Y')  \ar["P_{X'} \otimes P_{Y'}" description]{ru}&
\\
&H_*(X \times Y) \ar[pos=.7]{rr}{(\varphi \times \psi)_*} && H_*(X' \times Y')
\\
H^*(X \times Y) \ar[swap]{rr}{(\varphi \times \psi)_!}  \ar["P_{X \times Y}" description]{ru} &&H^*(X' \times Y') \ar[crossing over,leftarrow,pos=.2]{uu}{\vsimeqop} \ar[swap, "P_{X' \times Y'}" description]{ru} &
\end{tikzcd}
\end{equation*}
where all the vertical arrows are given by the K\"unneth theorem. 
The back face commutes by naturality of the K\"unneth formula. 
The face at the bottom commutes by definition of the pushforward. 
By \cref{lem:poincarekunneth}, the left and right faces commute. 
The face at the top commutes up to 
$(-1)^{|\varphi_!||Y'|} = (-1)^{|Y'|(|X| - |X'|)} $ 
for degree reasons, and by the definition of the pushforwards. 
Therefore, the front face 
also commutes up to this term. 
\end{proof}

Denote the category of closed oriented manifolds with torsion-free homology groups and continuous maps by $\Top.$ Cohomology with pullback and pushforward gives a contravariant and a covariant functor  from $\Top$ to $\bZ\gmod$.
The pullback functor is monoidal, while pushforward is only monoidal up to sign as shown in \cref{prop:comppushkunneth}. To address this issue, and to combine pullback and pushforwards functors, we define various enhancements of the category $\Top.$

\begin{defn}\label{def:doubletop}
Let $\DoubleTop$ be the category whose objects are closed oriented smooth manifolds with torsion-free homology groups and morphisms are chains of continuous maps (between such spaces) going alternatingly in one direction and the other. Here we identify $(X \xrightarrow{1_X} X) = (X \xleftarrow{1_X} X)$. Composition is given by concatenation of chains and reduction by composition of maps whenever two continuous maps go in the same direction. 
\end{defn}
\begin{rem} The categories $\Top$ and $\Top^{\operatorname{op}}$ embed into $\DoubleTop$ and generate the morphism spaces. We call morphisms in the image of $\Top,$ resp. $\Top^{\operatorname{op}}$, \emph{covariant}, resp. \emph{contravariant}.
\end{rem}
Let $H^* : \DoubleTop \rightarrow \bZ\gmod$ be the functor that associates to a space its cohomology. A chain of continuous map is sent to the composition of their respective pullbacks or pushfowards (depending on the direction). For example:
\[
H^*\bigl(M \xrightarrow{\varphi_0} N_0 \xleftarrow{\varphi_1} N_1 \xrightarrow{\varphi_2} \cdots \xleftarrow{\varphi_r} M' \bigr)
=
H^*(M) \xrightarrow{\varphi_r^* \circ \cdots \circ (\varphi_2)_! \circ \varphi_1^* \circ (\varphi_0)_! } H^*(M').
\]

Recall that a pair of maps $X \xrightarrow{\varphi} Z \xleftarrow{\psi} Y$ between smooth manifolds are \emph{transverse} whenever $T_p(Z) = \Image(D\varphi_x) + \Image(D\psi_y)$ for all $z = \varphi(x) = \psi(y)$. 
In this situation we can consider the commutative diagram
\[
\begin{tikzcd}
X \times_Z Y 
\ar[swap]{d}{\psi'}
\ar{r}{\varphi'}
&
Y \ar{d}{\psi}
\\
X \ar[swap]{r}{\varphi} & Z
\end{tikzcd}
\]
where $X \times_Z Y$ is the pullback of $X \xrightarrow{\varphi} Z \xleftarrow{\psi} Y$ in the category of smooth manifolds. 
Then, depending on the orientation of $X \times_Z Y $, we obtain a change of basis formula $(\varphi')_! \circ (\psi')^*  =  \pm \psi^* \circ \varphi_!$ (this is similar to \cite[\S1.7]{quillenelementary}). 
We say that $X \times_Z Y$ is the \emph{oriented pullback} if
\begin{equation}\label{eq:transpushpull}
 (\varphi')_! \circ (\psi')^*  = \psi^* \circ \varphi_! : H^*(X) \rightarrow H^*(Y).
\end{equation}

\begin{rem}
Note that in general the oriented pullback $X \times_Z Y$ is isomorphic to $Y \times_Z X$ as manifolds, but the orientation is not preserved. 
\end{rem}

\begin{defn}\label{def:Mdoubletop}
Let $\MDoubleTop$ be the category given by the same objects as $\DoubleTop$ and the same hom-spaces but with the relation
\begin{equation}\label{eq:pushpullrel}
\bigl(X \xrightarrow{\varphi} Z \xleftarrow{\psi} Y\bigr) = \bigl(X \xleftarrow{\varphi'} Z' \xrightarrow {\psi'} Y\bigr),
\end{equation}
whenever  $\varphi$ and $\psi$ are transverse and there is an orientation preserving diffeomorphism $h : Z' \xrightarrow{\simeq} X \times_{Z} Y$ such that the diagram
\[
\begin{tikzcd}
Z' \ar{dr}{h} \ar{r}{\psi'} \ar[swap]{d}{\varphi'} & Y \\
X & X \times_{Z} Y \ar{l} \ar{u}
\end{tikzcd}
\]
commutes.
\end{defn}

The functor $H^* : \DoubleTop \rightarrow \bZ\gmod$ descends to $\MDoubleTop$ by \cref{eq:transpushpull}. 
\begin{rem}
It is very suggestive to consider a category of manifolds where morphisms are given by correspondences/spans  
$$\bigl(X \leftarrow A \rightarrow Y\bigr)$$
where $A$ is a manifold and composition is given by the fiber product
$$\bigl(X \leftarrow A \xrightarrow {\varphi} Y\bigr)\circ\bigl(Y\xleftarrow{\psi} B \rightarrow Z\bigr)=\bigl(X \leftarrow A\times_YB \rightarrow Y\bigr).$$
Or in other words, extend relation \eqref{eq:pushpullrel} to non-transverse maps. However, this does not work without overcoming major technical difficulties. First, $A\times_YB$ does not need to be a manifold and composition is not well-defined. Rather, one would need to work with some kind of \emph{derived manifold}, which is not in the scope of this paper. 
Second, even if  $A \times_Y B$ is a manifold (and an additional \emph{cleanliness} assumption of $\varphi$ and $\psi$ is fulfilled), the composition would not be compatible with cohomology, since equation \eqref{eq:transpushpull} becomes wrong in general. This could be corrected by considering the correspondence $\bigl(X \leftarrow A\times_YB \rightarrow Y\bigr)$ together with the additional data of the \emph{excess bundle} on $A \times_Y B$ and by incorporating a multiplication with the Euler class of this bundle into the functor $H^*.$
\end{rem}
For any covariant map $\varphi : X \rightarrow X'$ (resp. contravariant map $\psi : X \leftarrow X'$) of closed oriented manifolds, we can consider the zigzags 
\begin{align*}
&\bar \varphi := \bigl( X \xleftarrow{1_X} X^{-1} \xrightarrow{\varphi} X' \bigr) \overset{\eqref{eq:pushpullrel}}{=} \bigl( X  \xrightarrow{\varphi} (X')^{-1} \xleftarrow{1_{X'}} X' \bigr),  
\\
\text{(resp. } 
&\bar \psi := \bigl(  X \xleftarrow{\psi} (X')^{-1} \xrightarrow{1_{X'}} X' \bigr) 
 \overset{\eqref{eq:pushpullrel}}{=}  
\bigl( X \xrightarrow{1_{X}} X^{-1} \xleftarrow{\psi} X' \bigr)
\text{)}
\end{align*}
in $\MDoubleTop$, where $X^{-1}$ is the manifold $X$ with opposite orientation. 
We extend the definition to any zigzag by declaring that $\overline{(\varphi \circ \psi)} := \bar \varphi \circ \psi \overset{\eqref{eq:pushpullrel}}{=} \varphi \circ \bar \psi$ for any pair of zigzags $\varphi, \psi$. 
Note that we have $\overline{\overline{\varphi}} = \varphi$, and  
\[
H^*(\bar \varphi) = - H^*(\varphi),
\] 
for any zigzag $\varphi$.

We linearize $\MDoubleTop$ over $\bZ$ by allowing finite sums of morphisms, and identifying $-\varphi := \bar \varphi$, giving a $\bZ$-linear category $\linMDoubleTop$. 
Moreover, elements in the hom-spaces of $\linMDoubleTop$ carry a $\bZ$-degree with contraviant maps being in degree zero, and a covariant map $\varphi : X \rightarrow X'$ being in degree $\deg(\varphi) := |X'| - |X|$. 
Thus, $\linMDoubleTop$ is a $\bZ\gmod$-enriched category. 

\smallskip

We would like to equip $\linMDoubleTop$ with a $\bZ$-graded monoidal product by using the cartesian product, so that the cohomological functor becomes graded monoidal. 
However, because of \cref{prop:comppushkunneth} we need to use a twisted version of the cartesian product where we fix the sign when taking the product with a covariant maps. 
Moreover, it is a priori not clear how to define the cartesian product of a covariant map with a contravariant map. We define it by taking the cartesian product of each element with the identity and composing the two maps we obtain. There are two ways to do that (see below), but because of \cref{eq:pushpullrel}, they are the same.

We turn $\linMDoubleTop$  into a  $\bZ$-graded monoidal category by equipping it with the \emph{twisted cartesian} product $\htimes$ given by:
\begin{itemize}
\item $X \htimes Y := X \times Y$ is the cartesian product on objects, with orientation induced by the cross product $[f_{X \times Y}] := [f_X] \times [f_Y]$;
\item the unit object is given by the $0$-dimensional point manifold $\{\star\}$, with left and right unitor given by the obvious inclusion maps $\{\star\} \times X \leftarrow X$ and $X \times \{ \star \} \leftarrow X$ respectively;
\item for a pair of contravariant maps $\varphi : X \leftarrow X'$ and $\psi : Y \leftarrow Y'$, we put 
\[
(\varphi \htimes \psi) := (\varphi \times \psi);
\]
\item for a contravariant map $\varphi : X \leftarrow X'$ and a covariant map $\psi : Y \rightarrow Y'$, we put 
\begin{align*}
(\varphi \htimes \psi) :&= 
\bigl(X \times Y \xrightarrow{1_X \times \psi} X \times Y' \xleftarrow{\varphi \times 1_{Y'}} X' \times Y' \bigr)
\\
&\overset{\eqref{eq:pushpullrel}}{=}
\bigl(X \times Y \xleftarrow{\varphi \times 1_Y} X' \times Y  \xrightarrow{1_{X'} \times \psi} X' \times Y'\bigr),
\end{align*}
and for the opposite case  of a contravariant $\varphi : X \rightarrow X'$ and a covariant map $\psi : Y \leftarrow Y'$, we put 
\begin{align*}
(\varphi \htimes \psi) :&= 
\bigl(X \times Y \xleftarrow{1_X \times \psi} (X \times Y')^{(-1)^{(|X'|-|X|)|Y'|}} \xrightarrow{\varphi \times 1_{Y'}} X' \times Y' \bigr)
\\
&\overset{\eqref{eq:pushpullrel}}{=}
\bigl(X \times Y \xrightarrow{\varphi \times 1_Y} (X' \times Y)^{(-1)^{(|X'|-|X|)|Y|}}  \xleftarrow{1_{X'} \times \psi} X' \times Y'\bigr);
\end{align*}
\item for a pair of covariant maps $\varphi : X \rightarrow X'$ and $\psi : Y \rightarrow Y'$, we put 
\begin{equation}\label{eq:htimescov}
(\varphi \htimes \psi) := 
(-1)^{(|X'|-|X|)|Y'|}(\varphi \times \psi).
\end{equation}
\end{itemize}
Note that because of \cref{eq:pushpullrel} we can extend the definition to any pair of zigzag unambiguously, so that 
$\htimes$ becomes a bifunctor. 
The coherence isomorphism is given by the contravariant map $(X \times Y) \times Z \leftarrow X \times (Y \times Z)$ given by the coherence isomorphism of the cartesian product. 
The coherence relations are immediate by coherence on the cartesian product. 
Naturality of the coherence isomorphisms is lengthy, but straightforward to verify.

Then, we equip $\linMDoubleTop$ with a graded symmetry given by the isomorphism $\tau_{X,Y} : X \htimes Y \rightarrow Y \htimes X$  obtained from the contravariant map 
\[
(\tau_{X,Y} : X \htimes Y \xrightarrow{\simeq} Y \htimes X) := \bigl(X \times Y \xleftarrow{\sigma} Y \times X, \quad \sigma(y,x) := (x,y) \bigr).
\]
Note that the  isomorphism $\tau_{X,Y} : X \htimes Y \xrightarrow{\simeq} Y \htimes X$ is graded natural in $X$ and $Y$, because of \cref{eq:pushpullrel} again.

The functor $H^* : \MDoubleTop \rightarrow \bZ\gmod$ extends to a functor $H^* : \linMDoubleTop \rightarrow \bZ\gmod$ which is  graded symmetric monoidal. The coherence isomorphism $H^*(X \htimes Y) \xrightarrow{\simeq} H^*(X) \otimes H^*(Y)$ is given by the K\"unneth formula, and it is natural by \cref{prop:comppushkunneth}.

\subsection{Geometric construction of $\oddTQFT$}

Recall $T^0 = \{\star\}$ is a point manifold, and thus $T^m \times T^0 \cong T^m$. 
We equip $T^0$ with orientation $[f_{T^0}] := [\{\star\}]$, $T^1$ with a fixed orientation $[f_{T^1}]$,  
and 
$T^m = T^1 \times \cdots \times T^1$ 
 with induced orientation $[f_{T^m}] := [f_{T^1}] \times \cdots \times [f_{T^1}]$.

Recall that
\[
H^*(T^m) \cong A^{\otimes m},
\]
as a graded abelian group, for all $m > 0$, and $H^*(T^0) \cong \bZ$. 
We fix the isomorphisms 
$
f^0 : H^*(T^0) \xrightarrow{\simeq}  \bZ
$
 by sending the dual of $[f_{T^0}]$ to $1$, 
and
$
f^1 : H^*(T^1) \xrightarrow{\simeq} A
$
by sending the dual of $[f_{T^1}]$ to $v_-$. 
By the K\"unneth theorem, we have a commutative diagram
\[
\begin{tikzcd}
H^*(T^{m_1}) \otimes H^*(T^{m_2}) 
\ar[swap]{d}{- \smile -}
\ar["f^{m_1} \otimes f^{m_2}"']{r}{\simeq}
&
A^{\otimes m_1} \otimes A^{\otimes m_2}
\ar{d}{\vsimeq}
\\
H^*(T^{m_1+m_2})
\ar["f^{m_1+m_2}"']{r}{\simeq}
&
A^{\otimes (m_1+m_2)},
\end{tikzcd}
\]
which we use to fix inductively an isomorphism $f^m : H^*(T^m) \xrightarrow{\simeq} A^{\otimes m}$.

Consider the maps
\begin{align*}
\Delta &: T^1 \hookrightarrow T^2,  & x &\mapsto (x,x),
\\
\varepsilon &: T^1 \twoheadrightarrow T^0, & x&\mapsto \star,
\\
\eta &: T^0 \hookrightarrow T^1, & \star &\mapsto  p,
\\
\tau &: T^2 \rightarrow T^2, & (x,y) &\mapsto (y,x).
\end{align*}

\begin{lem}\label{lem:pullcompute}
We have
\begin{align*}
\Delta^* &: H^*(T^2) \cong A^{\otimes 2} \rightarrow A \cong H^*(T^1), & 
&\begin{cases}
v_+ \otimes v_+ \mapsto v_+, & v_+ \otimes v_- \mapsto v_-, \\
v_- \otimes v_- \mapsto 0, & v_- \otimes v_+ \mapsto v_-,
\end{cases}
\\
\varepsilon^* &: H^*(T^0) \cong \bZ \rightarrow A \cong H^*(T^1), &
&\begin{cases}
1 \mapsto v_+,
\end{cases}
\\
\eta^* &: H^*(T^1) \cong A \rightarrow \bZ \cong  H^*(T^0), &
&\begin{cases}
v_+ \mapsto 1,
&
v_- \mapsto 0,
\end{cases}
\end{align*}
and $\tau^* \cong \tau_{A,A}$. 
\end{lem}

\begin{proof}
This is immediate by working with cellular cohomology and identifying $v_+$ in $A$ with the dual of the $0$-cell and $v_-$ with the dual of the $1$-cell inside of $T^1$. 
\end{proof}

\begin{lem}\label{lem:pushcomputeDelta}
We have
\begin{align*}
\Delta_! &: H^*(T^1) \cong A \rightarrow A^{\otimes 2} \cong H^*(T^2) , 
&
\begin{cases}
v_+ \mapsto - (v_- \otimes v_+ - v_+ \otimes v_-),  &\\
v_- \mapsto - (v_- \otimes v_-). &
\end{cases}
\end{align*}
\end{lem}

\begin{proof}
The idea of the proof is to use the cup/cap formula 
\begin{equation}\label{eq:cupcap}
\gamma([a] \frown \beta) = (\beta \smile \gamma)([a]),
\end{equation}
for $\beta,\gamma \in H^*(X)$ and $[a] \in H_*(X)$, and evaluate against elements in cohomology. 

Let $\alpha$ be the dual of $[a] := [f_{T^1}]$, so that $P_{T_1}(1) = [a]$ and $P_{T_1}(\alpha) = 1$. 
Let $[b] := [f_{T^1}] \times 1$ and $[c] := 1 \times [f_{T_1}]$ be the generating $1$-cells in $T^2 = T^1 \times T^1$ such that $\Delta_*([a]) = [b] + [c]$, and let $\beta$ and $\gamma$ be their respective dual. 
We have $[f_{T^2}] = [b] \times [c]$.  
Then, we obtain
\begin{align}\label{eq:betagammabtimesc}
(\beta \smile \gamma)\bigl( [b] \times [c] \bigr) 
\overset{\eqref{eq:cupcap}}{=}
 ([b] \times [c]) \frown (\beta \smile \gamma)
\overset{\eqref{eq:crosscap}}{=} -1.
\end{align}
Also, we have, $P_{T_2}(1) = [b] \times [c]$, and we compute 
\begin{align*}
\beta \bigl(P_{T_2}(\beta)\bigr) &= 0,
&
\beta \bigl(P_{T_2}(\gamma)\bigr) &=  \beta \bigl( ([b] \times [c]) \frown \gamma \bigr) 
\\
&
&
&\overset{\eqref{eq:cupcap}}{=}   (\gamma \smile \beta) \bigl( [b] \times [c] \bigr)  \overset{\eqref{eq:betagammabtimesc}}{=} 1,
\\
\gamma \bigl(P_{T_2}(\beta)\bigr) &=  \gamma \bigl( ([b] \times [c]) \frown \beta \bigr) 
&
\gamma \bigl(P_{T_2}(\gamma)\bigr) &=  0.
\\
&\overset{\eqref{eq:cupcap}}{=}    (\beta \smile \gamma) \bigl( [b] \times [c] \bigr)  \overset{\eqref{eq:betagammabtimesc}}{=} -1,
&
&
\end{align*}
Thus, $P_{T_2}(\beta) = -[c]$ and $P_{T_2}(\gamma) = [b]$, and by \cref{eq:betagammabtimesc}, $P_{T_2}(\beta \smile \gamma) = -1$. 

Using these results, we compute
\begin{align*}
\Delta_!(1) &= P_{T_2}^{-1} \circ \Delta_* \circ P_{T_1}(1) = -(\beta - \gamma), 
&
\Delta_!(\alpha) &= - (\beta \smile \gamma).
\end{align*}
We conclude by identifying $\alpha \mapsto v_-$, and $\beta \mapsto v_- \otimes v_+, \gamma \mapsto v_+ \otimes v-$ and $\beta \smile \gamma \mapsto v_- \otimes v_-$. 
\end{proof}

\begin{lem}\label{lem:pushcomputeEpsilon}
We have
\begin{align*}
\varepsilon_! &: H^*(T^1) \cong A \rightarrow \bZ \cong H^*(T^0), 
\quad 
\begin{cases}
v_+ \mapsto 0,  &\\
v_- \mapsto 1. &
\end{cases}
\end{align*}
\end{lem}

\begin{proof}
Let $[a] := [f_{T^1}]$ and $\alpha$ be its dual.
 We compute $\varepsilon_*(1) = 1$ and $\varepsilon_*([a]) = 0$. 
Thus, $\varepsilon_!(\alpha) = 1$ and $\varepsilon_!(1) = 0$. 
\end{proof}

We also compute the following, for later use:

\begin{lem}\label{lem:pushcomputeEta}
We have
\begin{align*}
\eta_! &: H^*(T^0) \cong \bZ \rightarrow A \cong H^*(T^1), 
\quad
1 \mapsto  v_.  
\end{align*}
\end{lem}

\begin{proof}
We have $\eta_*(1) = 1$. Thus, $\eta_!(1) = \alpha$, where $\alpha$ is the dual of $[f_{T^1}]$. 
\end{proof}

\smallskip
Furthermore, there is an (almost monoidal) functor $\TopFunctor : \chcobcat \rightarrow \MDoubleTop$ sending $S^1$ to $T^1$, the disjoint union product to the twisted cartesian product, 
\begin{align*}
\TopFunctor\left(\tikzdiagh[scale=.5]{
	\draw (0,0) .. controls (0,1) and (1,1) .. (1,2);
	\draw (1,0) .. controls (1,1) and (2,1) .. (2,0);
	\draw (3,0) .. controls (3,1) and (2,1) .. (2,2);
	\draw (0,0) .. controls (0,-.25) and (1,-.25) .. (1,0);
	\draw[dashed] (0,0) .. controls (0,.25) and (1,.25) .. (1,0);
	\draw (2,0) .. controls (2,-.25) and (3,-.25) .. (3,0);
	\draw[dashed] (2,0) .. controls (2,.25) and (3,.25) .. (3,0);
	\draw (1,2) .. controls (1,1.75) and (2,1.75) .. (2,2);
	\draw (1,2) .. controls (1,2.25) and (2,2.25) .. (2,2);
	\draw [->] (1.15,.15) -- (1.85,.15);
}\right) 
&:= T^2 \xleftarrow{\Delta} T^1, 
&
\TopFunctor\left(\tikzdiagh[xscale=.5,yscale=-.5]{
	\draw (0,0) .. controls (0,1) and (1,1) .. (1,2);
	\draw (1,0) .. controls (1,1) and (2,1) .. (2,0);
	\draw (3,0) .. controls (3,1) and (2,1) .. (2,2);
	\draw (0,0) .. controls (0,-.25) and (1,-.25) .. (1,0);
	\draw (0,0) .. controls (0,.25) and (1,.25) .. (1,0);
	\draw (2,0) .. controls (2,-.25) and (3,-.25) .. (3,0);
	\draw (2,0) .. controls (2,.25) and (3,.25) .. (3,0);
	\draw[dashed] (1,2) .. controls (1,1.75) and (2,1.75) .. (2,2);
	\draw (1,2) .. controls (1,2.25) and (2,2.25) .. (2,2);
	\draw [<-] (1.25,.45) -- (1.75,-.15);
}\right) &:= T^1 \xrightarrow{\Delta} T^2,
&
\TopFunctor\left(\tikzdiagc[scale=.5]{
	\draw (1,2) .. controls (1,1) and (2,1) .. (2,2);
	\draw (1,2) .. controls (1,1.75) and (2,1.75) .. (2,2);
	\draw (1,2) .. controls (1,2.25) and (2,2.25) .. (2,2);
}\right) &:=T^0 \xleftarrow{\varepsilon} T^1,
\end{align*}
and
\begin{align*}
\TopFunctor\left(\tikzdiagh[xscale=-.5,yscale=.5]{
	\draw (1,0) .. controls (1,1) and (2,1) .. (2,0);
	\draw (1,0) .. controls (1,-.25) and (2,-.25) .. (2,0);
	\draw[dashed] (1,0) .. controls (1,.25) and (2,.25) .. (2,0);
	\draw[->] (1.5,1.5) [partial ellipse=0:270:3ex and 1ex];
}\right) &:=T^1 \xrightarrow{\varepsilon} T^0,
&
\TopFunctor\left(\tikzdiagh[scale=.5]{
	\draw (1,0) .. controls (1,1) and (2,1) .. (2,0);
	\draw (1,0) .. controls (1,-.25) and (2,-.25) .. (2,0);
	\draw[dashed] (1,0) .. controls (1,.25) and (2,.25) .. (2,0);
	\draw[->] (1.5,1.5) [partial ellipse=0:270:3ex and 1ex];
}\right) &:= T^1 \xleftarrow{1} (T^1)^{-1} \xrightarrow{\varepsilon} T^0,
\end{align*}
and the twist is sent to $\tau_{T^1,T^1}$. 

\begin{prop}
 The functor $\TopFunctor : \chcobcat \rightarrow \MDoubleTop$ is well-defined.
\end{prop}

\begin{proof}
We only need to verify the relations involving the twist. \cref{eq:twistcommutes} holds  because of \cref{eq:pushpullrel}.
\cref{eq:symtwist} is clear by definition of $\tau_{T^1,T^1}$. 
Finally, the exchange relations are respected also because of \cref{eq:pushpullrel}.  
\end{proof}

It induces a graded symmetric monoidal functor $\linTopFunctor : \linchcobcat \rightarrow \MDoubleTop$, such that we obtain a commutative diagram:
\begin{equation*}
\begin{tikzcd}[column sep = 7ex]
\chcobcat \ar{d} \ar{r}{\TopFunctor} & \MDoubleTop \ar{d},
\\
\linchcobcat\ar[swap]{r}{\linTopFunctor} & \linMDoubleTop
\end{tikzcd}
\end{equation*}
where the functors $\chcobcat \rightarrow \linchcobcat$ and $\MDoubleTop \rightarrow \linMDoubleTop$ are the obvious one.

\begin{prop}
The functor $\linTopFunctor : \linchcobcat \rightarrow \linMDoubleTop$ is well-defined.
\end{prop}

\begin{proof}
For this, it is enough to check that $\TopFunctor$ respects the relations Eq.~\eqref{eq:reverseorientation} - \eqref{eq:cobcommute2}.
For \cref{eq:reverseorientation}, we verify that
\[
T^2 \xleftarrow{\tau} T^2 \xleftarrow{\Delta} T^1 = T^2 \xleftarrow{\Delta} T^1,
\]
and
\[
T^1 \xleftarrow{1} (T^1)^{-1} \xrightarrow{\Delta} T^2
\overset{\eqref{eq:pushpullrel}}{=} T^1  \xrightarrow{\Delta} T^2 \xleftarrow{\tau} T^2.
\]
For \cref{eq:chcobrelbirthmerge}, we compute
\[
T^1 \cong T^0 \times T^1 \xleftarrow{\varepsilon \times 1} T^2 \xleftarrow{\Delta} T^1 =  T^1 \xleftarrow{1} T^1,
\]
and 
\begin{align*}
&\bigl(T^1 \xrightarrow{\Delta} T^2 \cong T^1 \times T^1 \xleftarrow{1 \htimes 1} (T^1)^{-1} \times T^1 \xrightarrow{\varepsilon \htimes 1} T^0 \times T^1 \cong T^1\bigr)
\\
\overset{\eqref{eq:htimescov}}{=}& \bigl(T^1 \xrightarrow{\Delta} T^2 \cong T^1 \times T^1 \xleftarrow{1 \times 1} (T^1)^{-1} \times T^1\xleftarrow{1} T^1 \times T^1  \xrightarrow{\varepsilon \times 1} T^0 \times T^1 \cong T^1 \bigr)
\\
=& \bigl(T^1 \xrightarrow{\Delta} T^2 \cong T^1 \times T^1 \xrightarrow{\varepsilon \times 1} T^0 \times T^1 \cong T^1 \bigr)
=\bigl( T^1 \xrightarrow{1} T^1\bigr),
\end{align*}
and the third relation is immediate by \cref{rem:redundantrelchcob}.
For \eqref{eq:cobcommute1} and \eqref{eq:cobcommute2}, we need to check all cases. We can suppose $W$ and $W'$ each contain only an elementary cobordism. 
Then, it follows by using the definition of the twisted cartesian product, and checking the orientations of the pullbacks involved using \cref{prop:comppushkunneth}. We leave the details to the reader. 
\end{proof}

\begin{thm}\label{thm:geomTQFT}
The diagram of graded symmetric monoidal functors 
\[
\begin{tikzcd}
\linchcobcat \ar{rr}{\linoddTQFT} \ar[swap]{dr}{\linTopFunctor} && \bZ\gmod \\
&\linMDoubleTop \ar[swap]{ur}{H^*}&
\end{tikzcd}
\]
is commutative. 
\end{thm}

\begin{proof}
Since the functors are graded monoidal, we only need to verify that it coincides on $S^1$ and on each elementary cobordism. This follows from \cref{lem:pullcompute}, \cref{lem:pushcomputeDelta} and \cref{lem:pushcomputeEpsilon}. 
\end{proof}

We could have also considered the following variation of the theorem:
\begin{cor}\label{cor:geomTQFT}
The diagram of functors
\[
\begin{tikzcd}
\chcobcat \ar{rr}{\oddTQFT} \ar[swap]{dr}{\TopFunctor} && \bZ\gmod \\
&\DoubleTop \ar[swap]{ur}{H^*}&
\end{tikzcd}
\]
is commutative.
\end{cor}

However, in this case, the categories $\DoubleTop$ and $\chcobcat$, and the functors $H^*$, $\TopFunctor$ and $\oddTQFT$ are not (graded) monoidal.

%
%


\section{Odd arc algebras}\label{sec:oddArcAlg}

We first define the generalized odd arc algebra $OH_k^{n-k}$ using cohomology of components of the real $(n-k,k)$-Springer fiber. The multiplication rule is defined by doing a sequence of surgeries on circle diagrams, which gives pullbacks and pushfowards between certain spaces. By \cref{thm:geomTQFT}, these maps coincide with the ones obtain using the 
odd TQFT $\oddTQFT$, which we recalled in \cref{ssec:oddTQFT}, and which, in a slightly
different formulation, was also used in  \cite{naissevaz18} to construct an algebra structure on $OH_n^n$.

Then, following \cite{stroppelwebster12,wilbert13}, we use this construction to define a convolution algebra structure on the cohomology of intersections of the components $T_\ba$ of $\oSpgrFib{n-k}{k}$ (equivalently, the components of $\spgrFib{n-k}{k}(\bR)$). 
We show it endows $OH^{n-k}_k$ with an algebra structure equal to the one obtained from $\oddTQFT$. 
Moreover, we explain how certain bimodules and morphisms between them, used in \cite{naisseputyra} to construct the tangle version of odd Khovanov homology, can also be obtained  from a sort of convolution product on some topological spaces, which should coincide with subspaces of Spaltenstein varieties (as in \cite{wilbert13}). 
Finally, we describe the odd analog $OK_k^{n-k}$ of the quasi-hereditary cover algebra $K_k^{n-k}$ from \cite{stroppel09,chenkhovanov06}.

\subsection{The odd arc algebra} \label{sec:oddArcPP}

As a $\bZ$-module, the $(n-k,k)$-\emph{odd arc algebra} is defined as
\begin{align*}
OH^{n-k}_k &:= \bigoplus_{\ba,\bb \in \crossingless{n-k}{k}} {_\bb}(OH^{n-k}_k)_{\ba}, &
 {_\bb}(OH^{n-k}_k)_{\ba} &:= H^*(\torus_{\bb ; n-k,k} \cap \torus_{\ba; n-k,k}).
\end{align*}
We write ${_\bb}1{_\ba}$ for the unit in $H^*(\torus_{\bb} \cap \torus_{\ba})$.

Note that we can describe elements in $ {_\bb}(OH^{n-k}_k)_{\ba}$ by dotted circle diagrams. In particular, ${_\bb}1{_\ba}$ is given by the diagram of shape $\obb\ba$ without any dots. 
\subsubsection{Quantum grading}

In~\cite{khovanov02} the arc algebra $H_n^n$ is $\bZ$-graded. This grading is called the \emph{quantum grading}, and is an important ingredient for Khovanov homology. 
The quantum graded version of the odd arc algebra, is given by taking twice the grading of $H^*(\torus_{\bb} \cap \torus_{\ba})$ (i.e. $\deg_q(X_i) = 2$) and define 
\[
{_\bb}(OH^{n-k}_k)_{\ba} :=  H^*(\torus_{\bb} \cap \torus_{\ba})\{k-|\obb\ba|\},
\]
where $\{m\}$ is a shift up by $m$ in the quantum grading. 
When we refer to the degree $|x|$ of an element $x \in  H^*(\torus_{\bb} \cap \torus_{\ba})$, we will refer to its homological degree and not the quantum grading. 

\subsubsection{Composition rule}\label{sec:OHcombmultrule}

For $\bc \neq \bb$, we let the multiplication map 
\[
{_\bd}(OH^{n-k}_k)_{\bc} \otimes {_\bb}(OH^{n-k}_k)_{\ba} \rightarrow 0 \in OH^{n-k}_k
\]
be zero. Otherwise, for the multiplication
\[
{_\bc}(OH^{n-k}_k)_{\bb} \otimes {_\bb}(OH^{n-k}_k)_{\ba} \rightarrow {_\bc}(OH^{n-k}_k)_{\ba},
\]
we consider the planar diagram $\obc \bb \obb \ba$, given by putting of $\obc \bb$ above the one of $\obb \ba$, where the rays in the middle part $\bb \obb$ are truncated and separated. For example we could have:
\begin{align*}
	\ba = \bb  &=  \  
	\tikzdiagcm{
		\cupdiag{1}{2}{1};
		\raydiag{0}{2};
	 }
&
	\bc  &=  \  
	\tikzdiagcm{
		\cupdiag{0}{1}{1};
		\raydiag{2}{2};
	 }
&
	\obc\bb\obb\ba &= \ 
	\tikzdiagtcm{
		\cupdiag[3]{0}{1}{-1};
		\raydiag[3]{2}{-2};
		\cupdiag[3]{1}{2}{1};
		\raydiag[3]{0}{2};
		\cupdiag[-2]{1}{2}{-1};
		\raydiag[-2]{0}{-2};
		\cupdiag[-2]{1}{2}{1};
		\raydiag[-2]{0}{2};
	 }
\end{align*}
In order to obtain an element in $ {_\bc}(OH^{n-k}_k)_{\ba}$, we want to turn $\obc \bb \obb \ba$ into a diagram of shape $\obc \ba$. This will be achieved by performing surgeries on the symmetric arcs and rays of $\bb \obb$, turning them into straight line segments connecting $\obc$ with $\ba$. This, in turn, defines continuous maps between certain spaces defined below. 
 The multiplication rule is then defined by applying the cohomology functor $H^*$ to this sequence of maps. 

\smallskip

For any planar diagram $I$ consisting of pairs of symmetric arcs, pairs of symmetric truncated rays, and vertical line segments, joining $n$ points on the bottom to $n$ points on the top, we write $(i,j) \in I$ if  there is an arc joining the $i$th point to the $j$th point, $(i) \in I$ if there is a ray at the $i$th point, and $[i] \in I$ if there is a line segment joining the $i$th bottom point to the top one. 
Note that for $\bb\obb$ we can never have $[i] \in \bb\obb$.

Then, we associate the following space to the diagram $I$:
\[
\torus_I :=\left\{(x_1, \dots, x_n, y_1, \dots, y_n) \in \torus^n \times \torus^n \left|
\begin{array}{ll}
 x_i = x_j \text{ and } y_i = y_j, &\text { if } (i,j) \in I, \\
 x_i = (-1)^i p = y_i, &\text{ if } (i) \in I, \\
 x_i = y_i, &\text{ if } [i] \in I,
\end{array}
\right.
\right\}
\]
 and set
\[
{_\bc}(\torus_I){_\ba} := \torus_I \cap (\torus_{\bc} \times \torus_\ba).
\]
One can see that ${_\bc}(\torus_I){_\ba}$ is either the empty set or homeomorphic to $\torus^{|\obc I \ba|}$, where $|\obc I \ba|$ is the number of circle components in the planar diagram $\obc I \ba$. In particular $|\obc \bb \obb \ba| = |\obc\bb| + |\obb\ba|$, and if $\id_n$ consists only of $n$ vertical line segments, then ${_\bc}(\torus_{\id_n}){_\ba} \cong \torus_\bc \cap \torus_\ba$. Therefore, we have
\begin{align*}
H^*({_\bc}(\torus_{\bb\obb}){_\ba}) &\cong {_\bc}(OH^{n-k}_k){_\bb} \otimes {_\bb}(OH^{n-k}_k){_\ba}, 
&
H^*({_\bc}(\torus_{\id_n}){_\ba})  &\cong  {_\bc}(OH^{n-k}_k){_\ba}.
\end{align*}

\begin{rem}
In fact, this procedure can be generalized for more general flat tangle diagrams, see \cref{sec:bimodules} below. 
\end{rem}

Given $I$ and  $(r,s) \in I$, we define $I'$ to be given by removing the pair of symmetric arcs joining the $r$th point to the $s$th one in $I$, and replacing them by two vertical segments $[r]$ and $[s]$:
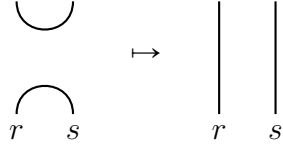
\[
\tikz[thick, xscale=.75, baseline={([yshift=1ex]current bounding box.center)}]{
		\draw  (0,1.5) .. controls (0,1) and (1, 1) ..  (1,1.5) ;
		\draw  (0,0)  node[below]{$r$} .. controls (0,0.5) and (1, 0.5) ..  (1,0)  node[below]{$s$}  ;
	 }
\quad \mapsto \quad
\tikz[thick, xscale=.75, baseline={([yshift=1ex]current bounding box.center)}]{
		\draw  (0,1.5) -- (0,0)   node[below]{$r$} ;
		\draw  (1,1.5) -- (1,0)   node[below]{$s$};
	 }
\]
We call this operation an $(r,s)$-\emph{surgery}. 
Assume that neither  ${_\bc}(\torus_{I}){_\ba}$ nor ${_\bc}(\torus_{I'}){_\ba}$ are empty.
Then, we need to distinguish five different cases:
\begin{enumerate}
 \item The $(r, s)$-surgery merges two different circle components, e.g.
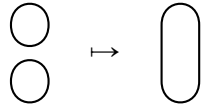
\begin{align*}
	\tikz[thick, xscale=.5,yscale=.75, baseline={([yshift=-.5ex]current bounding box.center)}]{
		\draw  (0,0) .. controls (0,0.5) and (1, 0.5) ..  (1,0) .. controls (1,-.5) and (0,-.5) .. (0,0);
		\draw[yshift=1cm]  (0,0) .. controls (0,0.5) and (1, 0.5) ..  (1,0) .. controls (1,-.5) and (0,-.5) .. (0,0);
	 }
\quad &\mapsto  \quad
	\tikz[thick, xscale=.5,yscale=.75, baseline={([yshift=-.5ex]current bounding box.center)}]{
		\draw  (0,1) .. controls (0,1.5) and (1, 1.5) ..  (1,1) -- (1,0) .. controls (1,-.5) and (0,-.5) .. (0,0) -- (0,1);
	 }
\end{align*}
Suppose that the $i$th and $j$th coordinates, with $i < j$, of $T^m$ correspond with the two circles involved in the $(r, s)$-surgery under the isomorphism ${_\bc}(\torus_{I'}){_\ba} \cong \torus^m$. 
Then, we fix an inclusion ${_\bc}(\torus_{I'}){_\ba} \hookrightarrow {_\bc}(\torus_{I}){_\ba}$ using the diagonal map 
which sends
\[
(x_1, \dots, x_m) \in  \torus^{m} \mapsto (x_1,\dots, x_{j-1}, x_i, x_j, \dots, x_m) \in  \torus^{m+1} \cong {_\bc}(\torus_{I}){_\ba}.
\]
\item The $(r, s)$-surgery  splits a circle into two different circle components, e.g.
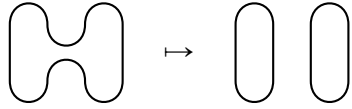
\begin{align*}
	\tikz[thick, xscale=.5,yscale=.75, baseline={([yshift=-.5ex]current bounding box.center)}]{
		\draw  (0,1) .. controls (0,1.5) and (1, 1.5) .. 
			 (1,1).. controls (1,.5) and (2, .5) .. 
			 (2,1).. controls (2,1.5) and (3, 1.5) .. 
			 (3,1) --
			 (3,0) .. controls (3,-.5) and (2,-.5) ..
			 (2,0) .. controls (2,.5) and (1,.5) ..
			 (1,0) .. controls (1,-.5) and (0,-.5) ..
			 (0,0) -- (0,1);
	 }
\quad &\mapsto  \quad
	\tikz[thick, xscale=.5,yscale=.75, baseline={([yshift=-.5ex]current bounding box.center)}]{
		\draw  (0,1) .. controls (0,1.5) and (1, 1.5) ..
			(1,1) -- 
			(1,0) .. controls (1,-.5) and (0,-.5) .. 
			(0,0) -- (0,1);
		\draw[xshift=2cm]  (0,1) .. controls (0,1.5) and (1, 1.5) ..
			(1,1) -- 
			(1,0) .. controls (1,-.5) and (0,-.5) .. 
			(0,0) -- (0,1);
	 }
\end{align*}
and then we fix an inclusion ${_\bc}(\torus_{I}){_\ba} \hookrightarrow {_\bc}(\torus_{I'}){_\ba}$, also by a diagonal map.
\item The $(r, s)$-surgery  spawns a circle component from a ray,  e.g.
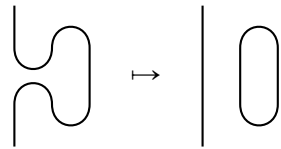
\[
	\tikz[thick, xscale=.5,yscale=.75, baseline={([yshift=-.5ex]current bounding box.center)}]{
		\draw (0,1.75) --
			(0,1) .. controls (0,.5) and (1,.5) ..
			(1,1) .. controls (1,1.5) and (2,1.5) ..
			(2,1) --
			(2,0) .. controls (2,-.5) and (1,-.5) ..
			(1,0) .. controls (1,.5) and (0,.5) ..
			(0,0) -- 
			(0,-.75); 
	 }
\quad \mapsto  \quad
	\tikz[thick, xscale=.5,yscale=.75, baseline={([yshift=-.5ex]current bounding box.center)}]{
		\draw (0,1.75) -- 
			(0,-.75); 
		\draw (1,1) .. controls (1,1.5) and (2,1.5) ..
			(2,1) --
			(2,0) .. controls (2,-.5) and (1,-.5) ..
			(1,0) -- (1,1);
	 }
\]
 and we fix an inclusion ${_\bc}(\torus_{I}){_\ba} \hookrightarrow {_\bc}(\torus_{I'}){_\ba}$ by the map
\[
1 \times \eta \times 1  : {_\bc}(\torus_{I}){_\ba} \cong \torus^{m} \rightarrow \torus^{m+1} \cong {_\bc}(\torus_{I'}){_\ba},
\]
where $\eta$ acts on the coordinates corresponding with the spawned circle component.
\item The $(r, s)$-surgery  merges a circle component into a ray,   e.g.
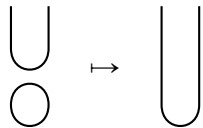
\[
	\tikz[thick, xscale=.5,yscale=.75, baseline={([yshift=-.5ex]current bounding box.center)}]{
		\draw  (0,0) .. controls (0,0.5) and (1, 0.5) .. 
			 (1,0) .. controls (1,-.5) and (0,-.5) .. (0,0);
		\draw (0,1.75) --
			(0,1) .. controls (0,.5) and (1,.5) ..
			(1,1) --
			(1,1.75);
	 }
\quad \mapsto  \quad
	\tikz[thick, xscale=.5,yscale=.75, baseline={([yshift=-.5ex]current bounding box.center)}]{
		\draw (0,1.75) --
			(0,0) .. controls (0,-.5) and (1,-.5) ..
			(1,0) --
			(1,1.75);
	 }
\]
 and then we fix ${_\bc}(\torus_{I'}){_\ba} \hookrightarrow {_\bc}(\torus_{I}){_\ba}$ by $1 \times \eta \times 1$.
\item The $(r, s)$-surgery  cuts and reconnects two rays,  e.g.
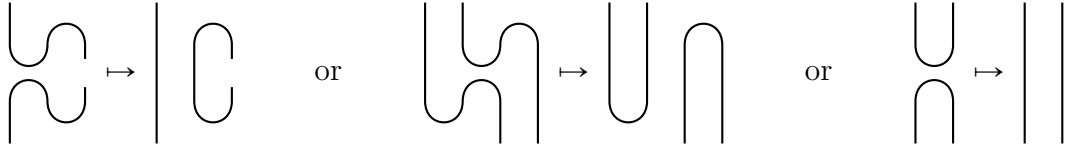
\begin{align*}
	\tikz[thick, xscale=.5,yscale=.75, baseline={([yshift=-.5ex]current bounding box.center)}]{
		\draw (0,1.75) --
			(0,1) .. controls (0,.5) and (1,.5) ..
			(1,1) .. controls (1,1.5) and (2,1.5) ..
			(2,1) --
			(2,.75) (2,.25) --
			(2,0) .. controls (2,-.5) and (1,-.5) ..
			(1,0) .. controls (1,.5) and (0,.5) ..
			(0,0) -- 
			(0,-.75); 
	 }
\  &\mapsto  \ 
	\tikz[thick, xscale=.5,yscale=.75, baseline={([yshift=-.5ex]current bounding box.center)}]{
		\draw (0,1.75) -- 
			(0,-.75); 
		\draw (1,1) .. controls (1,1.5) and (2,1.5) ..
			(2,1) --
			(2,.75) (2,.25) --
			(2,0) .. controls (2,-.5) and (1,-.5) ..
			(1,0) -- (1,1);
	 }
&\text{or}&&
	\tikz[thick, xscale=.5,yscale=.75, baseline={([yshift=-.5ex]current bounding box.center)}]{
		\draw (0,1.75) --
			(0,0) .. controls (0,-.5) and (1,-.5) ..
			(1,0) .. controls (1,.5) and (2,.5) ..
			(2,0) --
			(2,-.75);
		\draw (1,1.75) --
			(1,1) .. controls (1,.5) and (2,.5) ..
			(2,1) .. controls (2,1.5) and (3,1.5) ..
			(3,1) --
			(3,-.75);
	 }
\  &\mapsto  \ 
	\tikz[thick, xscale=.5,yscale=.75, baseline={([yshift=-.5ex]current bounding box.center)}]{
		\draw (0,1.75) --
			(0,0) .. controls (0,-.5) and (1,-.5) ..
			(1,0)  -- 
			(1,1.75);
		\draw (2,-.75) --
			(2,1) .. controls (2,1.5) and (3,1.5) ..
			(3,1) --
			(3,-.75);
	 }
&\text{or}&&
	\tikz[thick, xscale=.5,yscale=.75, baseline={([yshift=-.5ex]current bounding box.center)}]{
		\draw (0,1.75) --
			(0,1) .. controls (0,.5) and (1,.5) ..
			(1,1) --
			(1,1.75);
		\draw (0,-.75) --
			(0,0) .. controls (0,.5) and (1,.5) ..
			(1,0) --
			(1,-.75);
	 }
\  &\mapsto  \ 
	\tikz[thick, xscale=.5,yscale=.75, baseline={([yshift=-.5ex]current bounding box.center)}]{
		\draw (0,-.75) -- (0,1.75);
		\draw (1,-.75) -- (1,1.75);
	 }
\end{align*}
and then ${_\bc}(\torus_{I}){_\ba}  \cong {_\bc}(\torus_{I'}){_\ba}$, or at least one of ${_\bc}(\torus_{I}){_\ba}$ or ${_\bc}(\torus_{I'}){_\ba}$ is empty.
\end{enumerate}

Similarly, the $(r)$-\emph{surgery} is given by removing two symmetric truncated rays $(r)$, and putting a line segment $[r]$ instead:
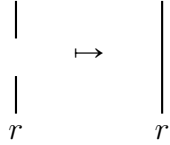
\[
\tikz[thick, xscale=.75, baseline={([yshift=1ex]current bounding box.center)}]{
		\draw  (0,0)  node[below]{$r$}  -- (0,.5) ;
		\draw  (0,1)  -- (0,1.5) ;
	 }
\quad \mapsto \quad
\tikz[thick, xscale=.75, baseline={([yshift=1ex]current bounding box.center)}]{
		\draw  (0,1.5) -- (0,0)   node[below]{$r$} ;
	 }
\]
Then, two things can happen:
\begin{enumerate}
\item The $(r)$-surgery connects two different rays together, and thus ${_\bc}(\torus_{I}){_\ba} \cong {_\bc}(\torus_{I'}){_\ba}$.
\item The $(r)$-surgery closes a ray into a circle component
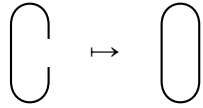
\[
	\tikz[thick, xscale=.5,yscale=.75, baseline={([yshift=-.5ex]current bounding box.center)}]{
		\draw (1,1) .. controls (1,1.5) and (2,1.5) ..
			(2,1) --
			(2,.75) (2,.25) --
			(2,0) .. controls (2,-.5) and (1,-.5) ..
			(1,0) -- (1,1);
	}
\quad \mapsto  \quad
	\tikz[thick, xscale=.5,yscale=.75, baseline={([yshift=-.5ex]current bounding box.center)}]{
		\draw (2,1) .. controls (2,1.5) and (1,1.5) ..
			(1,1) -- 
			(1,0) .. controls (1,-.5) and (2,-.5) ..
			(2,0) -- (2,1);
	}
\]
 and we fix ${_\bc}(\torus_{I}){_\ba} \hookrightarrow {_\bc}(\torus_{I'}){_\ba}$ by the map $1 \times \eta \times 1$.
\end{enumerate}

We fix an arbritrary order on the components of $\bb$ (i.e. on the rays and arcs). 
Then, following this order, we apply surgeries on $\obc \bb \obb \ba$.
 This gives a sequence of surgeries 
 \[
 \obc \bb \obb \ba = \obc I_0 \ba \rightarrow \obc I_1 \ba \rightarrow \dots \rightarrow  \obc I_{n-k} \ba = \obc \id_n \ba = \obc \ba,
 \]
 which in turns gives a sequence of inclusions, with each time either ${_\bc} (\torus_{I_\ell}) {_\ba} \hookrightarrow {_\bc} (\torus_{I_{\ell+1}}) {_\ba}$ or ${_\bc} (\torus_{I_{\ell+1}}) {_\ba} \hookrightarrow {_\bc} (\torus_{I_{\ell}}) {_\ba}$ by a fixed inclusion defined above. 
 
Viewing $(T_\bc \cap T_\bb) \htimes (T_{\bb} \cap T_{\ba})   \cong {_\bc} (\torus_{I_0}) {_\ba}$  and $ {_\bc} (\torus_{I_{n-k}}) {_\ba} \cong (T_{\bc} \cap T_\ba)$ as objects in $\DoubleTop$, this chain defines a map 
\[
\mu_{\bc,\bb,\ba} : (T_\bc \cap T_\bb) \htimes (T_{\bb} \cap T_{\ba}) \rightarrow ( T_{\bc} \cap T_\ba),
\]
 in $\DoubleTop$,  after choosing arbitrary orientations for all the involved spaces. 

\smallskip
 
We let the multiplication in $OH_k^{n-k}$ be given by
\[
\begin{tikzcd}[column sep = 12ex]
{_\bc}(OH^{n-k}_k){_\bb}  \otimes {_\bb}(OH^{n-k}_k){_\ba} \ar{r}
 \ar[equals]{d}
 &  {_\bc}(OH^{n-k}_k){_\ba}
\\
H^* (T_\bc \cap T_\bb) \otimes H^*(T_{\bb} \cap T_{\ba})
\ar{r}{H^*(\mu_{\bc,\bb,\ba})}
&
H^*(T_\bc \cap T_\ba).
 \ar[equals]{u}
\end{tikzcd}
\]
Note that this multiplication law depends on a lot of choices (order for the surgeries and choices of orientations).
Also, there is a priori no reason for it to be associative or preserve the quantum grading. 

\begin{lem}\label{lem:OHnnsame}
The abelian group $OH_n^n$ equipped with the multiplication defined above coincides with the definition of the odd arc algebra in~\cite{naissevaz18},  with the choice of orientation for the splitting cobordism given by the choice of fundamental classes. 
Moreover, in general $OH^{n-k}_k$ preserves the quantum grading, and agrees modulo 2 with the usual generalized arc algebra $H^{n-k}_k$ from \cite{stroppel09,  chenkhovanov06}. 
\end{lem}

\begin{proof}
This follows immediately from \cref{thm:geomTQFT}.
\end{proof}

This also means that the choice of order for doing the surgeries influences the definition of the multiplication rule only up to a global sign, by the results in~\cite{naissevaz18} and \cite{putyra14} (which generalize easily to the $(n-k,k)$-case). 

\smallskip

Again, in contrast to the even case~\cite{stroppelwebster12,wilbert13}, there is no reason to choose one sign convention over another. Hence, in the same spirit as~\cite{naissevaz18}, we consider the whole collection of odd arc algebras, indexed by all choices of signs. Concretely, there are two possible choices for each triple $(\bc, \bb, \ba)$. 
In all cases the construction yields a non-associative (quantum grading preserving) multiplication rule, which becomes graded associative (also called quasi-associative) when working in a certain monoidal category of graded spaces, as in~\cite{naissevaz18,naisseputyra} (this can easily be seen from the perspective of platforms from \cite{chenkhovanov06} for the $(n-k,k)$-case).

\subsubsection{Convolution algebra}
The composition rule of \cref{sec:OHcombmultrule} can also be understood as a certain convolution product using ideas from Stroppel--Webster \cite{stroppelwebster12}.

\smallskip

For $\bc, \bb, \ba \in \crossingless{n-k}{k}$, the inclusions of $\torus_{\bc}  \cap \torus_{\bb} \cap \torus_{\ba}$ in the three different spaces $\torus_{\bb} \cap \torus_{\ba}$, $\torus_{\bc} \cap \torus_{\bb}$ and $\torus_{\bc} \cap \torus_{\ba}$ 
 induce morphisms on the cohomology:
\[
\begin{tikzcd}
H^*(\torus_{\bb} \cap \torus_{\ba}) \arrow[rd,"\imath_{\bb\ba}^*"]&& \\
& H^*(\torus_{\bc}  \cap \torus_{\bb} \cap \torus_{\ba}) \arrow[r,"(\imath_{\bc\ba})_!"] & H^*(\torus_{\bc}\cap \torus_{\ba}).\\
H^*(\torus_{\bc} \cap \torus_{\bb}) \arrow[ru,"\imath_{\bc\bb}^*"']&&
\end{tikzcd}
\]
Then, for any $f_{\bc,\bb,\ba} \in H^*(\torus_{\bc}  \cap \torus_{\bb} \cap \torus_{\ba})$, we define a composition law $\star_{f_{\bc,\bb,\ba}}$ by
\[
\begin{tikzcd}
 {_\bc}(OH^{n-k}_k)_{\bb} \otimes {_\bb}(OH^{n-k}_k)_{\ba} 
\ar[equals]{d}
 \arrow[r,"\star_{f_{\bc,\bb,\ba}}"] &  {_\bc}(OH^{n-k}_k)_{\ba} \\
H^*(\torus_{\bc} \cap \torus_{\bb}) \otimes H^*(T_{\bb} \cap \torus_{\ba}) \arrow[r] & H^*(\torus_{\bc} \cap \torus_{\ba})
\ar[equals]{u}
 ,
\end{tikzcd}
\]
where the arrow below is given by
\[
(\imath_{\bc\ba})_!(f_{\bc,\bb,\ba} \smile \imath_{\bc\bb}^* \smile \imath_{\bb\ba}^*).
\]
Then, for $f = \{ f_{\bc,\bb,\ba} | {\bc,\bb,\ba \in B_k^{n-k}} \}$ we put $\star_{f} := \sum_{\bc,\bb,\ba} \star_{f_{\bc,\bb,\ba}}$.

\begin{thm}\label{thm:convOHn}
There exists an $f = \{ f_{\bc,\bb,\ba} \in  H^*(\torus_{\bc}  \cap \torus_{\bb} \cap \torus_{\ba}) | {\bc,\bb,\ba \in B_k^{n-k}} \}$ such that the composition law $\star_f$ coincides with the one defined in \cref{sec:OHcombmultrule}.
\end{thm}

\begin{proof}
As observed in~\cite[Theorem~35]{stroppelwebster12}, the sequence of inclusions
${_\bc} (\torus_{I_\ell}) {_\ba} \hookrightarrow {_\bc} (\torus_{I_{\ell+1}}) {_\ba}$ or ${_\bc} (\torus_{I_{\ell+1}}) {_\ba} \hookrightarrow {_\bc} (\torus_{I_{\ell}}) {_\ba}$ 
 of \cref{sec:OHcombmultrule} can be used to construct an element $f_{\bc,\bb,\ba}$. 
Indeed, given closed manifolds with inclusions $X \xhookrightarrow{\varphi} Z \xhookleftarrow{\psi} Y$, if $X$ and $Y$ intersect cleanly, then the commutative diagram
\[
\begin{tikzcd}[column sep= 1ex]
&Z&
\\
X \ar[hookrightarrow]{ur}{\varphi}&&Y \ar[hookrightarrow,swap]{ul}{\psi}
\\
&X \cap Y \ar[hookrightarrow]{ul}{\imath_X} \ar[hookrightarrow,swap]{ur}{\imath_Y} &
\end{tikzcd}
\]
gives 
\[
\psi^* \circ \varphi_! = (\imath_X)_! \bigl( e \smile \imath_Y^* \bigr),
\]
where $e$ is the Euler class of the excess bundle of the intersection  (see \cite{stroppelwebster12}, in the same spirit as in~\cite[Proposition 2.6.47]{chrissginzburg97}). 
Applying this recursively such that we do first all pullbacks and then all pushforwards yields an element $f_{\bc,\bb,\ba} \in H^*(\torus_{\bc}  \cap \torus_{\bb} \cap \torus_{\ba})$ such that $\star_{f_{\bc,\bb,\ba}}$ coincides with the multiplication rule of \cref{sec:OHcombmultrule}. 
Then, the result follows from \cref{thm:geomTQFT}. 
\end{proof}

Note that the choice of sign can be encoded in the choice $\pm f_{\bc,\bb,\ba}$, after having fixed an orientation for all the spaces $T_\bb \cap T_\ba$ and $T_\bc \cap T_\bb \cap T_\ba$. 

\subsubsection{Odd center}

Following~\cite{naissevaz18}, we define the  \emph{odd center} $OZ(OH^{n-k}_{k}) $ of $OH^{n-k}_{k}$ as 
\[
OZ(OH^{n-k}_{k}) := \{ z \in OH^{n-k}_{k} | zx = (-1)^{|z||x|} xz \text{ for all } x \in OH^{n-k}_{k} \}.
\]

\begin{rem}
The homological grading does not give $OH^{n-k}_k$ the structure of a superalgebra since the multiplication rule does not preserve it. However, it becomes a genuine grading when using the framework of \cite{naisseputyra}. Also, the odd center becomes the center in the general setting of algebra objects in a monoidal category. 
\end{rem}

\begin{prop}
We have $OZ(OH^{n-k}_{k})  \subset \bigoplus_{\ba}  {_\ba}(OH^{n-k}_{k}){_\ba}$. 
\end{prop}

\begin{proof}
The proof is similar to \cite[Proposition 3.11]{naissevaz18}.
\end{proof}

Therefore, the algebra structure of $OZ(OH^{n-k}_{k})$ does not depend on the choice of signs in the multiplication rule, since multiplying by an element of $\bigoplus_{\ba}  {_\ba}(OH^{n-k}_{k}){_\ba}$  involves only pullbacks. 
Recall from \cref{eq:monotorus} that there are inclusions of rings
\[
 H^*(\oSpgrFib{n-k}{k}) \hookrightarrow \bigoplus_{\ba \in \crossingless{n-k}{k}} H^*(\torus_{\ba}) \hookrightarrow OH^{n-k}_k.
\]
The image of this inclusion coincides with the odd center, and we get the following:
\begin{thm}\label{thm:oddcenter}
There is an isomorphism of rings
\[
H^*(\spgrFib{n-k}{k}(\bR))
 \xrightarrow{\simeq}
OZ(OH^{n-k}_k), 
\]
induced by the inclusions $T_\ba \hookrightarrow \spgrFib{n-k}{k}(\bR)$. 
\end{thm}

\begin{proof}
Using the same arguments as in the $(n,n)$-case~\cite{naissevaz18}, which are basically the same as in~\cite{khovanov04}, using \cref{thm:hiso} and \cref{thm:exactOH}, we obtain $OZ(OH^{n-k}_k) \cong H^*(\oSpgrFib{n-k}{k})$. Then, we conclude by \cref{thm:topdescofrealspringer} . 
\end{proof}

\subsection{Geometric bimodules}\label{sec:bimodules}
For $m \geq 0$, let
\begin{align*}
\crossingless{m}{} &:= 
\dot{\bigcup}_{k = 0}^{\lfloor m/2 \rfloor} 
\crossingless{m-k}{k} \text{ and}
\\
OH^{m} &:= \bigoplus_{k=0}^{\lfloor m/2 \rfloor} OH_{k}^{m - k}.
\end{align*}
To a flat tangle, we want to associate some (quasi-)bimodule over $OH^m$-$OH^n$, whose action is induced by a convolution-like product.  
The `quasi-' here means that the action is not necessarily strictly associative, but only graded associative, in the sense of \cite{naisseputyra}. 

\smallskip

Let $t$ be a flat tangle connecting $m_1$ points on the bottom to $m_2$ points on the top. 
We order the endpoints from left to right and then top to bottom. 
We write $(i,j) \in t$ if the $i$th endpoint of $t$ is connected with the $j$th one. 
We also write $|t|$ for the number of circle components of $t$. 
We associate to $t$ the space
\begin{align*}
\torus_t := &\left\{ 
(x_1, \dots, x_{m_2}, y_1, \dots, y_{|t|}, x_{m_2+1}, \dots, x_{m_2+m_1}) | x_i = x_j \text{ if $(i,j) \in t$}
\right\} 
\\
&\subset T^{m_2} \times T^{|t|} \times T^{m_1}.
\end{align*}
Then, we define 
\[
T_{\bb} \times_{t} T_\ba := (T_\bb \times T^{|t|} \times T_\ba) \cap T_t, 
\]
for $\bb \in \crossingless{m_2}{},\ba \in \crossingless{m_1}{}$, 
and
\[
\oddTQFT(t) :=\bigoplus_{\bb \in \crossingless{m_2}{},\ba \in \crossingless{m_1}{}} H^*(T_{\bb} \times_{t} T_\ba). 
\]
Note that there is a homeomorphism $T_\bb \cap T_\ba \cong T_\bb \times_{\id_n} T_\ba$ where $\id_n$ is the identity tangle. Thus, $OH^{n-k}_k \cong \oddTQFT(\id_n)$. 

\smallskip

Fix an $(m_3,m_2)$-flat tangle $t_2$ and an $(m_2,m_1)$-flat tangle $t_1$. Define the spaces
\begin{align*}
T_{t_2;t_1}
 :=
&\left\{
(x_1, \dots, x_{m_3}, y'_1, \dots, y'_{|t_2|}, x_{m_3+1}, \dots, x_{m_3+m_2}, y_1, \dots, y_{|t_2|}, x_{m_3+m_2+1}, \dots, x_{m_3+m_2+m_1} | \right.
\\
&
\left.
\begin{array}{ll}
\quad x_i = x_j, &\text{ if $(i,j) \in t_2$},
\\
\quad x_{i-m_3} = x_{j-m_3}, &\text{ if $(i,j) \in t_1$},
\end{array}
\right\} \subset T^{m_3} \times T^{|t_2|} \times T^{m_2} \times T^{|t_1|} \times T^{m_1},
\end{align*}
and
\[
T_\bc \times_{t_2} T_\bb \times_{t_1} T_\ba := (T_\bc \times T^{|t_2|} \times T_\bb \times T^{|t_1|} \times T_\ba) \cap T_{t_2;t_1},
\]
for $\bc \in B^{m_3}, \bb \in B^{m_2}$ and $\ba \in B^{m_1}$. Note that we have canonical inclusions where the images are identified with certain diagonals
\begin{align*}
T_\bc \times_{t_2} T_\bb \times_{t_1} T_\ba &\hookrightarrow T_\bc \times_{t_2} T_\bb \times T^{|t_1|}, \\
T_\bc \times_{t_2} T_\bb \times_{t_1} T_\ba &\hookrightarrow T_\bb \times_{t_1} T_\ba \times T^{|t_2|}, \\
T_\bc \times_{t_2} T_\bb \times_{t_1} T_\ba &\xhookrightarrow{\imath_{\bc t_2t_1\ba}} T_\bc \times_{t_2t_1} T_\ba.
\end{align*}
Let $\imath_{\bc t_2 \bb}$ be the map obtained by composition of the inclusion $T_\bc \times_{t_2} T_\bb \times_{t_1} T_\ba \hookrightarrow T_\bc \times_{t_2} T_\bb \times T^{|t_1|}$ above, followed by a projection on $T_\bc \times_{t_2} T_\bb$, and similarly for $\imath_{\bb t_1 \ba}$. 
We obtain a (non-necessarily associative) composition map
\[
\oddTQFT(t_2) \otimes \oddTQFT(t_1) \rightarrow \oddTQFT(t_2 t_1),
\]
taking the form of a convolution product
\[
\begin{tikzcd}
H^*(\torus_{\bb} \times_{t_1} \torus_{\ba}) \arrow{rd}{\imath_{\bb t_1 \ba}^*}& & \\
& H^*(T_\bc \times_{t_2} T_\bb \times_{t_1} T_\ba ) \arrow{r}{(\imath_{\bc t_2 t_1 \ba})_!} &  H^*(\torus_{\bc}  \times_{t_2t_1} \torus_{\ba}),\\
H^*(\torus_{\bc} \times_{t_2}\torus_{\bb}) \arrow[swap]{ru}{\imath_{\bc t_2 \bb}^*}&&
\end{tikzcd}
\]
that can be obtained from a chain of pullbacks/pushforwards by doing similar surgeries as in \cref{sec:OHcombmultrule}. 
In particular, the space $\oddTQFT(T)$ comes with a (non-associative) $OH^{m_1}$-$OH^{m_2}$-bimodule structure, which coincides with the multiplication inside of $OH^n$ for $t = \id_n$. This also coincides with the bimodules constructed in \cite{naisseputyra} (and therefore the action is graded associative) because of \cref{thm:geomTQFT}. 
 
\smallskip

Furthermore, given two flat tangles $t$ and $t'$ related by a surgery, we have (up to exchanging the role of $t$ and $t'$) a canonical inclusion $\Delta : T_{\bb} \times_{t'} T_\ba \hookrightarrow T_{\bb} \times_{t} T_\ba$, which in turn gives, by taking pullbacks and pushforwards, bimodule morphisms 
\begin{align*}
\oddTQFT(t) &\rightarrow \oddTQFT(t'), 
&
\oddTQFT(t') &\rightarrow \oddTQFT(t).
\end{align*}
Note that in general the morphim $\oddTQFT(t) \rightarrow \oddTQFT(t')$ is a mix of pullbacks and pushfowards, depending on the closure of $t$ and $t'$ by $\bb$ and $\ba$. 
Again,  by \cref{thm:geomTQFT}, it coincides with merge and split cobordisms for $\oddTQFT$, and thus gives the maps used to construct odd Khovanov homology (also for tangles as in \cite{naisseputyra}), after fixing the signs using some extra information coming from the grading that controls the associativity. 

\begin{rem}
We do not know if these choices of signs can be interpreted as coming from a sort of coherent choice of orientation for the involved spaces. 
\end{rem}

\subsection{The odd algebra $OK^{n-k}_k$} \label{sec:qhcarcalgebra}

In the even case, the algebra $H^{n-k}_k$ admits a quasi-hereditary cover $K^{n-k}_k$. The construction of $K^{n-k}_k$  is similar to the one of $H^{n-k}_k$, but allowing other kinds of crossingless matchings. These can be interpreted as adding some new subspaces of $\spgrFib{n-k}{k}(\bC)$ in the machinery. These subspaces correspond to closures of Bia\l{}ynicki-Birula cells in the paving of the complex Springer fiber, see~\cite{stroppelwebster12}. We add the odd version of these subspaces in $\oSpgrFib{n-k}{k}$ in order to construct an odd version of $K^{n-k}_k$.

\subsubsection{Weighted crossingless matchings}

We follow~\cite{stroppelwebster12}.

\begin{defn}
A \emph{weight sequence} of  type $(n-k,k)$ is given by sequence of $n$ elements in $\{\wedge, \vee\}$  with $\vee$ appearing exactly $k$ times. \\
A \emph{weighted crossingless matching} is the datum of a crossingless matching $\ba$ on $n$ points (we allow any number of arcs), and a weight sequence $\lambda$ of type $(n-k,k)$, such that if we label the endpoints of $\ba$ by the elements of $\lambda$, then any arc in $\ba$ is bounded by $\vee\wedge$, in that order, e.g. 
\[
 	\tikzdiagcm{
 		\cupdiag{1}{2}{1.5} [orientedcup];
	 }
\]
We write $WB^{n-k}_k$ for the set of weighted crossingless matchings of type $(n-k,k)$. 
\end{defn}

Note that a weighted crossingless matching $\lambda\ba$ possesses up to $k$ arcs and a least $n-2k$ rays, and thus $\ba \in \bigsqcup_{(\ell \leq k)} \crossingless{n-\ell}{\ell}$.

For a weight sequence $\lambda = w_1 \cdots w_n$, with $w_j \in \{\vee,\wedge\}$, and $ 1 \leq i \leq n$, we write
\[
i(\lambda) := \begin{cases}
1, & \text{ if } w_i = \vee, \\
0, & \text{ if } w_i = \wedge.
\end{cases}
\]
To a weighted crossingless matching $\lambda\ba$ we associate the following subspace of  $\oSpgrFib{n-k}{k}$:
\[
\torus_{\lambda\ba;n-k,k}  := \left\{(x_1, \dots, x_n) \in \torus^n \left|
\begin{array}{ll}
 x_i = x_j,  &\text{ if } (i,j) \in \ba,\\
 x_i =  (-1)^{i+i(\lambda)} p, &\text{ if } (i) \in \ba ,
\end{array}
\right.
\right\}.
\]
As before, we will usually write $\torus_{\lambda\ba}$ instead of $\torus_{\lambda\ba;n-k,k}$, assuming it is clear from the context. 

\begin{defn}
We say that a weighted crossingless matching is \emph{valid} if there is no couple of rays with endpoints labeled $\vee\wedge$, in that order from left to right. \\
We write $WB^{n-k}_{k}$ for the set of valid weighted crossingless matchings of type $(n-k,k)$.
\end{defn}

\begin{exe}\label{exe:weightedcrossinglessmatchings}
Take $n = 4$ and $k=1$. Then there are four valid weighted crossingless matchings:
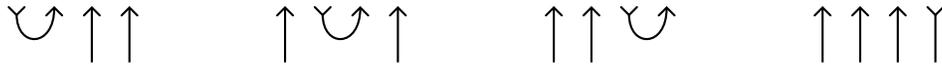
\begin{align*} 
	\tikz[thick, yscale=.75, xscale=.5, baseline={([yshift=1ex]current bounding box.center)}]{
		\draw[orientedcup]  +(0,0) .. controls (0,-0.75) and (1, -0.75) ..  +(1,0) ;
		\draw[-angle 90, xshift=2cm] (0,-1) -- (0,0);
		\draw[-angle 90, xshift=3cm] (0,-1) -- (0,0);
	 }\ &
&
	\tikz[thick, yscale=.75, xscale=.5, baseline={([yshift=1ex]current bounding box.center)}]{
		\draw[orientedcup, xshift=1cm]  +(0,0) .. controls (0,-0.75) and (1, -0.75) ..  +(1,0) ;
		\draw[-angle 90, xshift=0cm] (0,-1) -- (0,0);
		\draw[-angle 90, xshift=3cm] (0,-1) -- (0,0);
	 }\ &
& 
	\tikz[thick, yscale=.75, xscale=.5, baseline={([yshift=1ex]current bounding box.center)}]{
		\draw[orientedcup, xshift=2cm]  +(0,0) .. controls (0,-0.75) and (1, -0.75) ..  +(1,0) ;
		\draw[-angle 90, xshift=0cm] (0,-1) -- (0,0);
		\draw[-angle 90, xshift=1cm] (0,-1) -- (0,0);
	 }\ &
& 
	\tikz[thick, yscale=.75, xscale=.5, baseline={([yshift=1ex]current bounding box.center)}]{
		\draw[-angle 90, xshift=0cm] (0,-1) -- (0,0);
		\draw[-angle 90, xshift=1cm] (0,-1) -- (0,0);
		\draw[-angle 90, xshift=2cm] (0,-1) -- (0,0);
		\draw[-angle 90 reversed, xshift=3cm] (0,-1) -- (0,0);
	 }\ &
\end{align*}
corresponding to the spaces
\begin{align*}
\{x,x,-p,p\} &,
&\{-p,x,x,p\} &, 
&\{-p,p,x,x\} &, 
&\{-p,p,-p,-p\} &.
\end{align*}
\end{exe}

Given a crossingless matching of type $(n-k,k)$, one can construct a valid weighted crossingless matching of type $(n-k,k)$ by labeling all cups with $\vee\wedge$ and all rays with $\wedge$. Moreover, given such a weighted crossingless matching, one can construct all other valid weighted crossingless matchings of type $(n-k,k)$ by splitting some cups into two rays, and exchanging the corresponding labels $\vee\wedge$ to $\wedge\vee$:
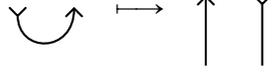
\[
 \tikz[thick, yscale=1.25, xscale=.75, baseline={([yshift=1ex]current bounding box.center)}]{
		\draw[orientedcup]  +(1,0) .. controls (1,-0.5) and (2, -0.5) ..  +(2,0) ; 
	 }
\quad \longmapsto \quad
 \tikz[thick, yscale=1.25, xscale=.75, baseline={([yshift=1ex]current bounding box.center)}]{
		\draw[-angle 90, xshift=0cm] (0,-.75) -- (0,0);
		\draw[-angle 90 reversed, xshift=1cm] (0,-.75) -- (0,0);
	 }
\]
One can obtain $\lambda' \ba'$ from $\lambda\ba$ in such a way, if and only if
$
T_{\lambda'\ba'} \subset T_{\lambda\ba}.
$
This is for example what happens with the two elements on the right in \cref{exe:weightedcrossinglessmatchings}.

\smallskip

Furthermore, given a weight sequence $\lambda$ of type $(n-k,k)$, there is a unique way to construct a crossingless matching $\ba$ such that $\lambda\ba$ is valid. The procedure consists of looking for adjacent unmatched pairs of $\vee \wedge$ and connecting them by a cup, repeating these operations until there is no more left. Then, we put rays on the remaining points. This means there are exactly $\binom{n}{k}$ elements in $WB^{n-k}_{k}$. Moreover, it means we can denote elements of $WB^{n-k}_{k}$ by only writing their weight sequence, and thus we will write $\torus_{\lambda}$ instead of $\torus_{\lambda\ba}$.

\smallskip

Given two $\lambda, \mu  \in WB^{n-k}_{k}$, we are interested in the intersection $\torus_{\lambda} \cap \torus_{\mu}$. For this, we consider the diagram $\obb\mu$  given by taking the horizontal mirror of $\bb$, putting the labels of $\mu$ at the bottom endpoints (without reversing them). 
Then, we consider the diagram $\obb(\mu\lambda)\ba$ given by first forgetting all labels on the endpoints on cups/caps (only keeping labels on rays) and gluing $\obb\mu$ on top of $\lambda\ba$. 
We obtain $\torus_{\lambda} \cap \torus_{\mu} = \emptyset$ if two rays are connected in $\obb(\mu\lambda)\ba$ with opposite orientations, considering the cups and caps as reversing the orientation. Otherwise, we have $\torus_{\lambda} \cap \torus_{\mu} \cong T^{|\obb(\mu\lambda)\ba|}$, where $|\obb(\mu\lambda)\ba|$ counts the number of circle components in $\obb(\mu\lambda)\ba$. 
Note that if we consider elements of $B^{n-k}_{k}$ with their default weight sequence, then this rules is equivalent to the one of not having turnbacks.

\subsubsection{$OK^{n-k}_k$ arc algebras}

We put:
\begin{align*}
OK_{k}^{n-k} &:= \bigoplus_{\lambda, \mu \in WB^{n-k}_{k}} {_\mu}(OK_k^{n-k})_{\lambda}, &
 {_\mu}(OK_k^{n-k})_{\lambda} := H^*\bigl( \torus_{\mu} \cap \torus_{\lambda} \bigr). 
\end{align*}

\begin{rem}
Again, there is a quantum graded version with ${_\mu}(OK_k^{n-k})_{\lambda} := H^*\bigl( \torus_{\mu} \cap \torus_{\lambda} \bigr) \{k - |\obb(\mu\lambda)\ba|\}$.
\end{rem}

The multiplication map ${_{\gamma}}(OK^{n-k}_k)_{ \nu } \otimes {_{\mu}}(OK^{n-k}_k)_{ \lambda} \rightarrow 0$ is zero for $\nu \neq \mu$. Otherwise, for the multiplication
\[
{_\nu}(OK^{n-k}_k)_{\mu} \otimes {_\mu}(OK^{n-k}_k)_{\lambda} \rightarrow {_\nu}(OK^{n-k}_k)_{\lambda},
\]
we consider the planar diagram $\obc(\nu\mu)\bb \obb(\mu\lambda)\ba$ where the rays of $\bb \obb$ are truncated, separated and oriented. Then, we follow exactly the same procedure as~\cref{sec:OHcombmultrule}, the only difference being that now truncated rays in $I$ are oriented by $\mu$. Thus, we consider the spaces
\[
\torus_I :=\left\{(x_1, \dots, x_n, y_1, \dots, y_n) \in \torus^{ n} \times T^{ n} \left|
\begin{array}{ll}
 x_i = x_j \text{ and } y_i = y_j, &\text { if } (i,j) \in I, \\
 x_i = (-1)^{i+i(\mu)} p = y_i, &\text{ if } (i) \in I, \\
 x_i = y_i, &\text{ if } [i] \in I,
\end{array}
\right.
\right\},
\]
and
\[
{_{\bc\nu}(\torus_I){_{\lambda\ba}}} := \torus_I \cap (\torus_\nu \times \torus_\lambda).
\]
This yields a quantum graded, non-associative convolution algebra by converting the sequence of surgeries to a product
\[
(\imath_{\nu\lambda})_!(f_{\nu,\mu,\lambda} \smile \imath_{\nu\mu}^* \smile \imath_{\mu\lambda}^*),
\]
 obtained from
\[
\begin{tikzcd}
H^*(\torus_{\mu} \cap \torus_{\lambda}) \arrow[rd,"\imath_{\mu\lambda}^*"]&& \\
& H^*(\torus_{\nu}  \cap \torus_{\mu} \cap S_{\lambda}) \arrow[r,"(\imath_{\nu\lambda})_!"] & H^*(\torus_{\nu}\cap \torus_{\lambda}).\\
H^*(\torus_{\nu} \cap \torus_{\mu}) \arrow[ru,"\imath_{\nu\mu}^*"']&&
\end{tikzcd}
\]
As before, this gives a composition law $\star_{f}$ for each choice of classes $f = \{ f_{\nu,\mu,\lambda} | \nu, \mu, \lambda \in  WB^{n-k}_{k}\}$. 

To sum up, we have:

\begin{thm}\label{thm:oddqhcover}
For each choice of classes $f = \{ f_{\nu,\mu,\lambda} | \nu, \mu, \lambda \in  WB^{n-k}_{k}\}$, $(OK^{n-k}_k, \star_f)$ is a quantum graded, non-associative algebra, which agrees modulo 2 with the quasi-hereditary cover $K^{n-k}_k$ from \cite{stroppel09}. 
\end{thm}

%
%


\bibliographystyle{bibliography/habbrv}


\end{document}